\documentclass[a4paper,11pt,nosumlimits]{amsart}
\usepackage{latexsym,amsmath,amssymb,color,mathrsfs,enumerate,bbm,fullpage, diagrams,cancel}
\usepackage[english]{babel}

\numberwithin{equation}{section}

\title[Gauge Theory for Spectral Triples and the Unbounded Kasparov Product]{Gauge Theory for Spectral Triples and the \\[2pt] Unbounded Kasparov Product}
\author{Simon Brain, Bram Mesland and Walter D. van Suijlekom} 
\address{Institute for Mathematics, Astrophysics and Particle Physics, Faculty of Science, Radboud University Nijmegen, Heyendaalseweg 135, 6525 AJ Nijmegen, The Netherlands} 
\address{Mathematics Institute,  Zeeman Building, University of Warwick, Coventry CV4 7AL, United Kingdom}
\address{Institute for Mathematics, Astrophysics and Particle Physics, Faculty of Science, Radboud University Nijmegen, Heyendaalseweg 135, 6525 AJ Nijmegen, The Netherlands} 
\email{s.brain@math.ru.nl, b.mesland@warwick.ac.uk, waltervs@math.ru.nl}

\date{\today}

\newtheorem{thm}{Theorem}[section]
\newtheorem{cor}[thm]{Corollary}
\newtheorem{lem}[thm]{Lemma}
\newtheorem{prop}[thm]{Proposition}
\newtheorem{rem}[thm]{Remark}

\theoremstyle{definition}
\newtheorem{defn}[thm]{Definition}
\newtheorem{example}[thm]{Example}

\def\black{\color{black}}

\newcommand{\End}{\textup{End}}
\newcommand{\D}{\textup{d}}

\newcommand{\n}{\nabla}

\newcommand{\M}{\textup{M}}

\newcommand{\Dom}{\mathfrak{Dom}}
\newcommand{\lrh}{\leftrightharpoons}
\newcommand{\im}{\mathfrak{Im}}

\renewcommand{\d}{\textup{d}}
\newcommand{\KK}{\textup{KK}}
\newcommand{\LL}{\textup{L}}

\DeclareFontFamily{OT1}{pzc}{}
\DeclareFontShape{OT1}{pzc}{m}{it}{<-> s * [1.20] pzcmi7t}{}
\DeclareMathAlphabet{\mathpzc}{OT1}{pzc}{m}{it}

\newcommand{\A}{\mathcal{A}}
\newcommand{\B}{\mathcal{B}}
\newcommand{\E}{\mathcal{E}}
\newcommand{\h}{\mathcal{H}}

\newcommand{\ZZ}{\mathbb{Z}}
\newcommand{\C}{\mathbb{C}}
\newcommand{\RR}{\mathbb{R}}
\newcommand{\R}{\mathbb{R}}
\newcommand{\TT}{\mathbb{T}}

\newcommand{\hotimes}{\tilde{\otimes}}

\newcommand{\Lip}{\textnormal{Lip}}

\def\A{\mathcal{A}}

\DeclareMathOperator{\ad}{ad}
\def\Aut{\mathrm{Aut}}
\def\bT{\mathbb{T}}
\def\B{\mathcal{B}}
\def\BB{\mathbb{B}}
\def\bar{\overline}

\def\C{\mathbb{C}}
\def\cC{\mathcal{C}}

\def\D{\mathfrak{D}}

\def\dirac{\partial\mkern-9.5mu/\,}

\def\E{\mathcal{E}}
\def\EE{\mathpzc{E}}

\def\End{\textup{End}}

\def\F{\mathcal{F}}

\def\G{\mathcal{G}}
\def\GG{\mathpzc{G}}

\def\H{\mathcal{H}}

\def\Hom{\mathrm{Hom}}
\def\half{\tfrac{1}{2}}

\def\into{\hookrightarrow}

\def\K{\mathbb{K}}
\def\L{\mathcal{L}}

\def\bN{\mathbb{N}}

\def\R{\mathbb{R}}

\def\cS{\mathcal{S}}
\def\bS{\mathbb{S}}

\DeclareMathOperator{\SU}{SU}

\DeclareMathOperator{\Spin}{Spin}

\def\T{\mathbb{T}_\theta}

\def\tilde{\widetilde}

\DeclareMathOperator{\U}{U}

\def\Z{\mathbb{Z}}

\newbox\ncintdbox \newbox\ncinttbox
\setbox0=\hbox{$-$}
\setbox2=\hbox{$\displaystyle\int$}
\setbox\ncintdbox=\hbox{\rlap{\hbox
    to \wd2{\kern-.1em\box2\relax\hfil}}\box0\kern.1em}
\setbox0=\hbox{$\vcenter{\hrule width 4pt}$}
\setbox2=\hbox{$\textstyle\int$}
\setbox\ncinttbox=\hbox{\rlap{\hbox
    to \wd2{\kern-.05em\box2\relax\hfil}}\box0\kern.1em}

\begin{document}
\bibliographystyle{plainmath}

\begin{abstract}
We explore factorizations of noncommutative Riemannian spin geometries over commutative base manifolds in unbounded KK-theory. After setting up the general formalism of unbounded KK-theory and improving upon the construction of internal products, we arrive at a natural bundle-theoretic formulation of gauge theories arising from spectral triples. We find that the unitary group of a given noncommutative spectral triple arises as the group of endomorphisms of a certain Hilbert bundle;  the inner fluctuations split in terms of connections on, and endomorphisms of, this Hilbert bundle. Moreover, we introduce an extended gauge group of unitary endomorphisms and a corresponding notion of gauge fields. We work out several examples in full detail, to wit Yang--Mills theory, the noncommutative torus and the $\theta$-deformed Hopf fibration over the two-sphere.
\end{abstract}

\maketitle

\tableofcontents

\parskip 0.5ex
\section{Introduction}

In this paper we use the internal product of cycles in unbounded KK-theory to introduce a new framework for studying gauge theories in noncommutative geometry. In the current literature one finds some equally appealing but mutually incompatible ways of formulating the notion of a gauge group associated to a noncommutative algebra. Herein we extend and then employ the formulation of the unbounded Kasparov product to study fibrations and factorizations of manifolds in noncommutative geometry, yielding a fresh approach to gauge theory which provides a unifying framework for some of the various existing constructions.

Gauge theories arise very naturally in noncommutative geometry: rather notably they arise from spectral triples \cite{C96}. In fact, one of the main features of a noncommutative $*$-algebra is that it possesses a non-trivial group of inner automorphisms coming from the group of unitary elements of the algebra. In many situations and applications, this group of inner automorphisms is identified with the gauge group of the spectral triple. Moreover, the so-called inner fluctuations ---again a purely noncommutative concept--- are recognized as gauge fields, upon which this gauge group acts naturally.

In a more conventional approach, gauge theories are described by vector bundles and connections thereon, with the gauge group appearing as the group of unitary endomorphisms of the space of sections of the bundle. This approach has also been extended to the noncommutative world (see \cite{C94} and references therein, also \cite{Lnd97}). 
The present paper is an attempt to see where these approaches can be unified, in the setting of noncommutative gauge theories on a {\em commutative} base. That is to say, we explore the question of whether (or when) the unitary gauge group of an algebra can be realized as endomorphisms of a vector bundle and whether the inner fluctuations arise as connections thereon. 

We do this by factorizing noncommutative spin manifolds, i.e. spectral triples, into two pieces consisting of a commutative `horizontal' base manifold and a part which describes the `vertical' noncommutative geometry. The vertical part is described by the space of sections of a certain Hilbert bundle over the commutative base, upon which the unitary gauge group acts as bundle endomorphisms. Moreover, the inner fluctuations of the original spectral triple decompose into  connections on this Hilbert bundle and endomorphisms thereof.

Thus, the setting of the paper is that of spectral triples, the basic objects of Connes' noncommutative geometry \cite{C96}. Such a spectral triple, denoted $(A,\mathcal{H},D)$ consists of a $C^{*}$-algebra represented on a Hilbert space $\mathcal{H}$, together with a self-adjoint operator $D$ with compact resolvent. Moreover, the $*$-subalgebra $\mathcal{A}\subset A$ consisting of elements $a\in A$ for which $a\left(\Dom (D)\right)\subseteq \Dom (D)$ and $[D,a]$ is bounded on $\Dom (D)$ is required to be dense in $A$. The prototype of a spectral triple is obtained by representing the $*$-algebra of continuous functions on a compact spin manifold $M$ upon the Hilbert space of  $L^{2}$-sections of its spinor bundle, on which the Dirac operator acts with all of the desired properties. Over twenty years of active research on spectral triples has yielded a heap of noncommutative examples of such structures, coming from dynamics, quantum groups and various deformation techniques.

The main idea explored in this paper is that of fibering an arbitrary spectral triple $(A,\mathcal{H}, D)$ over a second, commutative spectral triple $(B,\mathcal{H}_{0},D_{0})$, that is to say over a classical Riemannian spin manifold. The notion of fibration we will be using is that of a \emph{correspondence}, adopting the point of view of \cite{CS} that bounded KK-cycles are generalizations of algebraic correspondences. In the setting of unbounded KK-theory, a correspondence is defined in \cite{Mes09b} as a triple $(\mathcal{E},S,\nabla)$ consisting of: an $\mathcal{A}$-$\mathcal{B}$-bimodule $\mathcal{E}$ that is an orthogonal summand of the countably generated free module $\mathcal{H}_{\mathcal{B}}$;  a self-adjoint regular operator $S$ on $\mathcal{E}$;  a connection $\nabla:\mathcal{E}\rightarrow\mathcal{E}\hotimes_{\mathcal{B}}\Omega^{1}(\mathcal{B})$ on this module. The module $\mathcal{E}$ admits a natural closure as a $C^{*}$-module $\mathpzc{E}$ over $B$, such that $(\mathpzc{E},S)$ is an unbounded cycle for Kasparov's KK-theory \cite{Kas} in the sense of Baaj and Julg \cite{baaj-julg}, and hence represents a correspondence in the sense of \cite{CS}. The datum $(\mathcal{E},S,\nabla)$ is required to relate the two spectral triples, in the sense that $(A,\mathcal{H},D)$ is unitarily equivalent to 
\[(\mathcal{E},S,\nabla)\otimes_{B}(B,\mathcal{H}_{0},D_{0}):=(A,\mathcal{E}\hotimes_{\mathcal{B}}\mathcal{H}_{0}, S\otimes 1+1\otimes_{\nabla}D_{0}),\] and thus in particular represents the Kasparov product of $(\mathpzc{E},S)$ and $(B,\mathcal{H}_{0},D_{0})$. In \cite{Mes09b} this was shown to be the case when $[\nabla,S]$ is bounded; a similiar construction, which is simpler and more general, was presented in \cite{KaLe}. In particular this allows for the use of connections for which $[\nabla,S](S\pm i)^{-1}$ is bounded.\black

Since the $C^*$-algebra $B$ is commutative, by Gel'fand duality it is isomorphic to $C(X)$ for some (compact) topological Hausdorff space $X$. Moreover, by \cite{Tak79} the $C^*$-module $\mathpzc{E}$ consists of continuous sections of a Hilbert bundle over $X$, upon which the algebra $A$ acts by endomorphisms. As a consequence, the unitary group $\U(A)$ acts by unitary endomorphisms on this bundle, thus putting the inner automorphism group of the algebra $A$ in the right place, as a subgroup of the group of unitary bundle endomorphisms. Moreover, the inner fluctuations of $(A,\mathcal H,D)$ can be split into connections on the Hilbert bundle and endomorphisms thereof. Summarizing, this puts into place all ingredients necessary for doing gauge theory on $X$. 

In order to deal with the examples in this paper, we enlarge the class of modules $\mathcal{E}$ used to construct unbounded Kasparov products. The class of modules used in \cite{KaLe,Mes09b} is not closed under arbitrary countable direct sums. This inconvenience is due to the fact that the module $\mathcal{H}_{\mathcal{B}}$ admits projections of arbitrarily large norm. The theta-deformed Hopf fibration treated in the last section of the present paper illustrates this phenomenon, which is present in full force already in the classical case.  It is proved in \cite{Kaad} that, for $\B$ commutative, the bounded projections in $\H_{\B}$ correspond to bundles of bounded geometry, a class which indeed does not contain the bundle appearing in the Hopf fibration.

Indeed, the Peter-Weyl theorem for $\SU(2)$ tells us that module $\E$ necessary for expressing the Hopf fibration $\bS^{3}\rightarrow\bS^{2}$ as a Kasparov product is isomorphic to a direct sum over $n\in\Z$ of rank one modules $\mathcal{L}_{n}$. As a $C^{*}$-module this yields a well defined direct sum, yet the projections $p_{n}$ defining the bundles $\mathcal{L}_{n}$ have the property that their differential norms with respect to the Dirac operator on $\bS^{2}$ grow increasingly with $n$. To accommodate this phenomenon, we develop a theory of unbounded projections on the free module $\H_\mathcal{B}$. We show that the range of such projections define certain closed submodules of $\mathcal{H}_\B$ 
and that such modules admit connections and regular operators. We then proceed to show that the unbounded Kasparov product can be constructed in this setting in very much the same way as in \cite{KaLe, Mes09b}.

In this way we obtain an explicit description of the noncommutative Hopf fibration in terms of an unbounded KK-product, thus going beyond the projectivity studied in \cite{DS10,DSZ13}. A similar construction, in the context of modular spectral triples, appeared in \cite{KS11} to construct Dirac operators on a total space carrying a circle action (namely, on $\SU_q(2)$). There the base space is the standard noncommutative Podle\'s sphere, whereas here we aim for a commutative base. 

Our proposal for gauge theories has a potential application in the study of instanton moduli spaces. Namely, in \cite{BL09a,bvs,B13} additional gauge parameters were introduced to describe the moduli space of instantons on a certain noncommutative four-sphere $\bS^4_\theta$. We expect that these extra gauge parameters can be accommodated inside the group of unitary endomorphisms of the corresponding Hilbert bundle described above.


\bigskip

The paper is organized as follows. In Section \ref{sect:KK} we set up the operator formalism of unbounded KK-theory and describe the (unbounded) internal Kasparov product. We then extend this to the setting of so-called {\em Lipschitz} modules, as called for by the examples that we discuss later. 

In Section \ref{sect:gauge-KK} we explain how a factorization of a noncommutative spin manifold gives rise to a natural (commutative) geometric setup and describe how the inner automorphisms and inner fluctuations can be described in terms of (vertical) Hilbert bundle data.

In the remaining part of the paper we illustrate our factorization in unbounded KK-theory by means of three classes of examples. Namely, in Section \ref{sect:YM} we recall \cite{CC,BoeS10} how ordinary Yang--Mills theory can be described by a spectral triple and explain how this is naturally formulated using a KK-factorization. Section \ref{sect:nc-torus} contains another example of our construction, namely the factorization of the noncommutative torus as a circle bundle over a base given by a circle \cite{Mes10}. Most importantly, Section \ref{sect:nc-hopf} contains a topologically non-trivial example, which is the noncommutative Hopf fibration of the theta-deformed three-sphere $\bS^3_\theta$ over the classical two-sphere $\mathbb S^2$.

\subsubsection*{Notation and terminology} 
In this paper, all $C^*$-algebras are assumed to be unital; we denote them by $A,B$, {etc.}, with densely contained $*$-algebras denoted by $\A,\B$, respectively. For the general theory of $C^*$-modules over a $C^*$-algebra, we refer for example to \cite{Lance}. We write $\mathpzc{E}\leftrightharpoons B$ to denote a right $C^*$-module $\mathpzc{E}$ over the $C^*$-algebra $B$. 

We assume some familiarity with the representation theory of $C^*$-algebras on $C^*$-modules, writing $\End^*_B(\mathpzc{E})$ for the $C^*$-algebra of adjointable operators on a right $C^*$-module $\mathpzc{E}\leftrightharpoons B$. If $\mathpzc{E}$ is equipped with a representation $\pi:A\to\End^*_B(\mathpzc{E})$, we write $A\to \mathpzc{E}\leftrightharpoons B$ and say that $\mathpzc{E}$ is a \textbf{Hilbert} $A$-$B$-\textbf{bimodule}. The algebra of compact operators on $\mathpzc{E}\lrh B$ is denoted by $\K_B(\mathpzc{E})$. In the special case where $B=\C$ and so $\EE=\H$ is simply a Hilbert space, we write $\K(\H)$ for the compact operators and $\BB(\H)$ for the $C^*$-algebra of bounded operators on $\H$.

By a {\bf grading} of a vector space $V$ we shall always mean a $\ZZ_2$-grading, i.e. a self-adjoint linear operator $\Gamma:V\to V$ such that $\Gamma^2=1_V$. By a representation of a graded $C^*$-algebra $A$ on a graded Hilbert module  $\mathpzc{E}\leftrightharpoons B$, we shall always mean a graded representation. We describe (possibly unbounded) linear operators on $\mathpzc{E}\leftrightharpoons B$ using the notation $D:\Dom(D)\to\mathpzc{E}$, where $\Dom(D)\subseteq \mathpzc{E}$ denotes the domain of $D$, a dense linear subspace of $\mathpzc{E}$.

Similarly, we assume some familiarity with operator algebras and their representation theory \cite{db:book}. Recall that a linear map $\phi:\A\to\B$ between operator spaces $\A$ and $\B$ is said to be \textbf{completely bounded} if its extension $\phi\otimes 1:\A\otimes\K(\H)\to\B\otimes\K(\H)$ is bounded.  We shall often abbreviate our terminology by describing maps such as completely bounded isomorphisms, completely bounded isometries etc. as ``cb-isomorphisms'', ``cb-isometries'' and so on. In this paper, the correct tensor product of operator spaces $\A$, $\B$ is given by the (graded) {\bf Haagerup tensor product}, which we denote by $\A\,\tilde\otimes\,\B$. Here it is understood that the tensor product is over the complex numbers $\C$. Its balanced variant, for a right operator module $\mathcal{E}$ and a left operator module $\mathcal{F}$ over an operator algebra $\mathcal{B}$ is denoted $\mathcal{E}\hotimes_{\mathcal{B}}\mathcal{F}$.

For each $i=0,1,2,\ldots,$ we write $\C_i$ for the $i$th {\bf Clifford algebra}, {i.e.} the graded complex unital $*$-algebra generated by the even unit $\gamma^0$ and the odd elements $\gamma^k$, $k=1,2,\ldots,i$, modulo the relations $\gamma^k\gamma^l+\gamma^l\gamma^k=2\delta^{kl}$ and $(\gamma^k){}^*=\gamma^k$. As a complex vector space, the Clifford algebra $\C_i$ is $2^i$-dimensional.

\subsubsection*{Acknowledgments} SB was supported by fellowships granted through FNR (Luxembourg) and INdAM (Italy),  each cofunded under the Marie Curie Actions of the European Commission FP7-COFUND. BM was supported by the EPSRC grant EP/J006580/2.
We would like to thank Nigel Higson, Adam Rennie and Magnus Goffeng for several useful discussions and Alain Connes for some helpful suggestions. We are grateful to an anonymous referee for noticing a serious gap in an earlier version of the paper.

\section{Operator Modules and Unbounded KK-Theory}
\label{sect:KK}

As already mentioned, this article is concerned with the study of spectral triples in noncommutative geometry and the extent to which these define a gauge theory \cite{C94,C96,CC}. \black
Our investigation will for the most part be facilitated by the unbounded version of Kasparov's bivariant KK-theory for $C^*$-algebras. In this section we explain the main definitions and techniques that we shall need later in the paper.

\subsection{Noncommutative spin geometries}
Recall that a noncommutative spin manifold (in the sense of Connes) is defined in terms of a spectral triple, which in turn is defined as follows.

\begin{defn}\label{de:st}
A  {\bf spectral triple} $(A,\H,D)$  consists of:
\begin{enumerate}[\hspace{0.5cm} (i)]
\item a unital $C^*$-algebra $A$, a Hilbert space $\H$ and a faithful representation $\pi:A\to\BB(\H)$ of $A$ on $\H$;
\item an unbounded self-adjoint linear operator $D:\Dom(D)\to \H$ with compact resolvent,
\end{enumerate}
such that the $*$-subalgebra 
\[
\A:=\{a\in A: [D,\pi(a)]\textnormal{ extends to an element of }\mathbb{B}(\H)\}
\] 
is dense in $A$.
Such a triple is said to be {\bf even} if it is graded, i.e. if it is equipped with a self-adjoint operator $\Gamma:\H\to\H$ with $\Gamma^2=1_\H$ such that $\Gamma D+D\Gamma=0$ and $\Gamma \pi(a)=\pi(a)\Gamma$ for all $a\in A$. Otherwise the spectral triple is said to be {\bf odd}. With $0<m<\infty$, the triple $(A,\H,D)$ is said to be {\bf $m^+$-summable} if the operator $(1+D^2)^{-1/2}$ is in the Dixmier ideal $\mathcal{L}^{m^+}(\H)$.
\end{defn}

The latter definition is motivated by the following classical example, which we will need throughout the present paper. Let $M$ be a closed Riemannian spin manifold and let $A=C(M)$ be the unital $C^*$-algebra of continuous complex-valued functions on $M$. Write $\H=L^2(M,\cS)$ for the Hilbert space of square-integrable sections of the spinor bundle $\cS$ and denote by $\dirac_M$ the Dirac operator on $M$, which we recall is defined to be the composition
$$
\dirac_M:\Dom(\dirac_M)\to\H,\qquad \dirac_M:=c\circ \n_\cS,
$$
where $c$ denotes ordinary Clifford multiplication and $\n_\cS$ is the canonical spin connection on $\cS$ for the Riemannian metric. Then $A$ is faithfully represented upon $\H$ by pointwise multiplication and $\A=\Lip(M)$ is nothing other that the pre-$C^*$-algebra of Lipschitz functions on $M$.

\begin{defn}\label{de:can-st}
The datum $(C(M),L^2(M,\cS),\dirac_M)$ is called the {\bf canonical spectral triple} over the closed Riemannian spin manifold $M$.
\end{defn}

The canonical spectral triple is even if and only if the underlying manifold $M$ is even-dimensional, with the grading of the Hilbert space $\H$ induced by the corresponding $\ZZ_2$-grading of the spinor bundle $\cS$. If $M$ is an $m$-dimensional manifold then the corresponding spectral triple can be shown to be $m^+$-summable.

Crucially, every spectral triple admits a canonical first order differential calculus over the dense $*$-algebra $\A$. Indeed, given a spectral triple $(A,\H,D)$, the associated differential calculus is defined to be the $\A$-$\A$-bimodule
\begin{equation}\label{ccalc}
\Omega^1_D(\A):=\{\sum_j a_j[D,\pi(b_j)]~|~a_j,b_j\in\A\}\subseteq\mathbb{B}(\H),
\end{equation}
where the sums are understood to be convergent in the norm topology of $\mathbb{B}(\H)$ (in contrast with the definition given in e.g. \cite{C94}). For the canonical spectral triple $(C(M),\H,\dirac_M)$ over a closed Riemannian spin manifold $M$, the differential calculus $\Omega^1_D(\A)$ is isomorphic to the $\Lip(M)$-bimodule $\Omega^1(M)$ of continuous one-forms on $M$.%

Next we come to recall the main definitions and techniques of the unbounded version of Kasparov's bivariant KK-theory for $C^*$-algebras \cite{Kas, baaj-julg}. 
Given a Banach space $X$, recall that a linear operator $D:\Dom (D)\rightarrow X$ is said to be {\bf closed} whenever its graph
\begin{equation}
\label{graph} \mathfrak{G}(D):=\left\{\begin{pmatrix}x\\Dx\end{pmatrix}~|~x\in \Dom (D)\right\}\subseteq X\oplus X\end{equation}
is a closed subspace of $X\oplus X$. A closed, densely defined, self-adjoint linear operator $D$ on a $C^{*}$-module $\mathpzc{E}\lrh B$ is said to be {\bf regular} if and only if the operators $D\pm i:\Dom (D)\rightarrow \mathpzc{E}$ have dense range, which in turn happens if and only if these operators are bijective.

With these concepts in mind, let $A$ and $B$ be graded $C^*$-algebras, let $\mathpzc{E}\leftrightharpoons B$ be a graded right $C^*$-module over $B$ equipped with a representation $\pi:A\to\End^{*}_{B}(\mathpzc{E})$ and let $D:\Dom(D)\rightarrow\mathpzc{E}$ be an odd unbounded self-adjoint regular operator. In this situation, we make the following definition.

\begin{defn}[\cite{baaj-julg}]
\label{defn:KK-unbounded} 
The pair $(\mathpzc{E},D)$ is said to be an {\bf even unbounded $(A,B)$ KK-cycle} if:
\begin{enumerate}[\hspace{0.5cm} (i)]
\item the operator $D:\Dom(D)\rightarrow\mathpzc{E}$ has compact resolvents, i.e. $(D\pm i)^{-1}\in\K_{B}(\mathpzc{E})$;
\item the unital $*$-subalgebra
$$
\A:=\{a\in A~:~[D,\pi(a)]\in\End^{*}_{B}(\mathpzc{E})\}\subseteq A
$$
is dense in $A$. 
\end{enumerate}
We write $\Psi_0(A,B)$ for the set of even unbounded $(A,B)$ KK-cycles modulo unitary equivalence. 
\end{defn}

\begin{example}
It is clear from the definitions that every even spectral triple $(A,\H,D)$ over a $C^*$-algebra $A$ determines an unbounded cycle in $\Psi_0(A,\C)$. We will deal with the case of odd spectral triples at the end of this section.
\end{example}

Kasparov's KK-groups \cite{Kas} are homotopy quotients of the sets of unbounded cycles $\Psi_{0}(A,B)$ in the following sense. Associated to a given self-adjoint regular operator $D$ on $\mathpzc{E}$ is its \emph{bounded transform}
\[ \mathfrak{b}(D):=D(1+D^{2})^{-1/2},\]
which determines $D$ uniquely (see \cite{Lance} for details). The pair $(\mathpzc{E},\mathfrak{b}(D))$ is a \textbf{Kasparov module}: these are the bounded analogues of the elements of $\Psi_{0}(A,B)$, defined to be pairs $(\mathpzc{E},F)$ with $F\in\End^{*}_{B}(\mathpzc{E})$ such that, for all $a\in A$, we have
\[ F^{2}-1,~ [F,a], ~F-F^{*}\in\K_{B}(\mathpzc{E}).\] 
In \cite{baaj-julg} it is shown that, for every unbounded KK-cycle $(\mathpzc{E},D)\in\Psi_{0}(A,B)$, its bounded transform $(\mathpzc{E},\mathfrak{b}(D))$ is a Kasparov module;  conversely every Kasparov module arises in this way as the bounded transform of some unbounded KK-cycle.
Two elements in $\Psi_{0}(A,B)$ are said to be {\bf homotopic} if their bounded transforms are so; the set of homotopy equivalence classes is denoted $\KK_{0}(A,B)$. Kasparov proved in \cite{Kas} that this is an Abelian group under the operation of taking direct sums of bimodules.

An important feature of the KK-groups $\KK_0(A,B)$ is that they admit an internal product
$$
\otimes_B:\KK_0(A,B)\times \KK_0(B,C)\to \KK_0(A,C).
$$ 
Kasparov proved existence and uniqueness of this product at the homotopy level
but, as of yet, a concrete expression for the product of two bounded Kasparov modules is still lacking and might not exist at all. As already mentioned, Connes and Skandalis \cite{CS} gave an insightful interpretation of Kasparov's product at the bounded level in terms of {\em correspondences}, in which the fingerprints of the geometric nature of the construction are clearly visible. On the other hand, in the unbounded picture, the work of Kucerovsky \cite{Kuc97} provides sufficient conditions for an unbounded cycle to represent the product of two given cycles. 

Indeed, the pair $(\mathpzc{E},D)$ is said to be the {\bf unbounded Kasparov product} of the cycles $(\mathpzc{E}_1,D_1)$ and $(\mathpzc{E}_2,D_2)$, denoted
$$
(\mathpzc{E},D)\simeq (\mathpzc{E}_1,D_1)\otimes_B (\mathpzc{E}_2,D_2),
$$
if together they satisfy the conditions of \cite[Thm~13]{Kuc97}. The conditions of the latter theorem give a hint of the actual form of the product operator in the unbounded picture. Indeed, the constructions of \cite{KaLe,Mes09b} yield an explicit description of the unbounded Kasparov product, under certain smoothness assumptions imposed on the KK-cycles involved. Later on we shall sketch the details of how this unbounded product is formed: to do so we need first to introduce some background theory.

\subsection{Projective operator modules and their properties} The key observation in \cite{KaLe,Mes09b} is that, in order to define the product of a pair of unbounded KK-cycles, one needs to impose certain differentiability conditions upon the underlying $C^*$-modules. This section is devoted to giving a precise meaning to this notion of differentiability and a description of the class of modules that we shall need in the present paper. \black

The required notion of differentiability for $C^*$-modules is motivated by the special case of spectral triples. Indeed, let $(B,\mathcal{H},D)$ be a spectral triple as in Definition~\ref{de:st}. The corresponding dense subalgebra
\[\mathcal{B}:=\{b\in B~:~ [D,\pi(b)]\in\mathbb{B}(\mathcal{H})\}\]
will be called the {\bf Lipschitz subalgebra} of $B$. We will always consider it with the topology given by the representation
\begin{equation}\label{Liprep}\pi_{D}:\mathcal{B}\to \BB(\H\oplus\H),\qquad b\mapsto\begin{pmatrix}\pi(b) &0 \\ [D,\pi(b)] & \pi(b)\end{pmatrix}.\end{equation} As such it is a closed subalgebra of the $C^*$-algebra of operators on a Hilbert space, that is to say it is an \textbf{operator algebra}. Moreover, the involution in $\mathcal{B}$ satisfies the identity
\begin{equation}\label{v}\pi_{D}(b)^{*}=v\,\pi_{D}(b^{*})v^{*},\qquad \quad \textnormal{where}\quad v=\begin{pmatrix}0 & -1 \\1&0\end{pmatrix}.\end{equation}
More generally, recall that such algebras have a name \cite{Mes09b}.

\begin{defn} An {\bf involutive operator algebra} is an operator space $\mathcal{B}$ with completely bounded multiplication $\mathcal{B}\,\hotimes\,\mathcal{B}\rightarrow \mathcal{B}$, together with an involution $b\mapsto b^{*}$ which becomes a completely bounded anti-isomorphism when extended to matrices in the usual way.
\end{defn}

Note that $C^{*}$-algebras in particular fit this definition, as do Lipschitz algebras according to property \eqref{v}. Throughout the remainder of this section, we let $\B$ denote an arbitrary involutive operator algebra (although always keeping the special Lipschitz case in mind). As one might expect, involutive operator algebras admit a class of modules analogous to $C^{*}$-modules, which we now describe. 

First of all, let us denote $\hat{\Z}:=\Z\setminus\{ 0\}$. Then the Hilbert space $\ell^{2}(\hat{\Z})$ comes equipped with a natural $\Z/2$-grading. We define $\mathcal{H}_{\mathcal{B}}$ to be the right $\B$-module $\H_\B:=\ell^{2}(\hat{\Z})\,\tilde{\otimes}\,\mathcal{B}$, where $\tilde{\otimes}$ denotes the graded Haagerup tensor product. This module can be visualized as the space of $\ell^{2}$ column vectors with entries in $\mathcal{B}$, in the sense that a given column vector $(a_{i})_{i\in \hat{\Z}}$ is an element of $\mathcal{H}_{\mathcal{B}}$ if and only if $\sum_{i}\pi(a_{i})^{*}\pi (a_{i})\in \BB(\mathcal{H})$ for some completely bounded representation $\pi:\mathcal{B}\rightarrow \BB(\mathcal{H})$.

\begin{lem} 
The module $\mathcal{H}_{\mathcal{B}}$ admits a canonical inner product defined by
\begin{equation}\label{inprod} \langle (a_{i}), (b_{i})\rangle :=\sum_{i} a_{i}^{*}b_{i}\end{equation}
for each pair of column vectors $(a_i), (b_i)\in \H_\B$.
\end{lem}

\begin{proof} We must show the series on the right-hand side converges. To this end we write the inner product as a matrix product of column vectors and estimate (using complete boundedness of the involution) that
\[ \begin{split} \| \langle (a_{i}),(b_{i})\rangle \| & = \|(a_{i}^{*})^{t} \cdot (b_{i}) \|
\leq C\|(a_{i}^{*})^{t}\|\, \|(b_{i})\|\leq C\|(a_{i})\|\, \|(b_{i}) \| \end{split}\]
for some constant $C>0$. Now since $(a_{i}),(b_{i})\in\mathcal{H}_{\mathcal{B}}$, the norm of their tails will tend to zero and so the above estimate shows the inner product series is indeed convergent.
\end{proof}

\begin{rem}\textup{
It is important to note that it is only in the case where $\mathcal{B}$ is an honest $C^{*}$-algebra that the inner product \eqref{inprod} determines the topology of $\mathcal{H}_{\mathcal{B}}$. Nevertheless, just as in the $C^{*}$-module case, we define $\End^{*}_{\mathcal{B}}(\mathcal{H}_{\mathcal{B}})$ to be the $*$-algebra of operators on $\mathcal{H}_{\mathcal{B}}$ that admit an adjoint with respect to the inner product \eqref{inprod}.
}
\end{rem}

The elements of the $*$-algebra $\End^{*}_{\mathcal{B}}(\mathcal{H}_{\mathcal{B}})$ are automatically $\mathcal{B}$-linear and completely bounded. As such it is perfectly natural to consider {\em stably rigged} $\B$-modules, that is to say right $\mathcal{B}$-modules  which are cb-isomorphic to $p\mathcal{H}_{\mathcal{B}}$ for some (completely bounded) projection $p\in\End^{*}_{\mathcal{B}}(\mathcal{H}_{\mathcal{B}})$. Stably rigged modules were the cornerstone of the construction in \cite{KaLe,Mes09b}, however in the present paper we shall need a larger class of modules.

\begin{defn}\label{proj}Let $\mathcal{B}$ be an involutive operator algebra. A {\bf projection operator} on $\mathcal{H}_{\mathcal{B}}$ is a densely defined self-adjoint operator $p:\Dom (p)\rightarrow \mathcal{H}_{\mathcal{B}}$ such that $p^{2}=p$.
\end{defn}

The latter definition thus allows for the possibility of {\em unbounded} projection operators. Note that for an unbounded idempotent operator we necessarily have $\im (p)\subseteq\Dom (p)$. It is shown in \cite{Mes09b} that a closed, densely defined, self-adjoint operator $D$ on $\mathcal{H}_{\mathcal{B}}$ is regular if and only if there is a unitary isomorphism $\mathfrak{G}(D)\oplus v\,\mathfrak{G}(D)\cong \mathcal{H}_{\mathcal{B}}\oplus\mathcal{H}_{\mathcal{B}}$, with $v$ as in \eqref{v} and the isomorphism being given by coordinatewise addition.  This fact yields the following characterization of when a given projection is bounded.

\begin{prop} A projection $p$ on $\H_\B$ is bounded if and only if it is regular.
\end{prop} 

\begin{proof} If $p$ is bounded then the operators $p\pm i$ are invertible, whence $p$ is regular. Conversely, suppose that $p$ is regular. Then there is a unitary isomorphism
\[
\begin{split}\mathfrak{G}(p)\oplus v\,\mathfrak{G}(p)\xrightarrow{\sim}& \mathcal{H}_{\mathcal{B}}\oplus\mathcal{H}_{\mathcal{B}},\\
\begin{pmatrix}\begin{pmatrix} x\\ px\end{pmatrix},\begin{pmatrix} -py\\ y\end{pmatrix}\end{pmatrix}\mapsto &\begin{pmatrix}x-py\\px+y\end{pmatrix},\end{split}
\]
and so in particular we have that
\[
\mathcal{H}_{\mathcal{B}}=\{x-py~|~x,y\in\Dom (p)\}.
\]

Since $p$ is a projection we know that $py\in\Dom (p)$ and so $\mathcal{H}_{\mathcal{B}}\subseteq\Dom (p)$, whence $p$ is adjointable and therefore bounded.
\end{proof}

\begin{lem} Let $p$  be a closed idempotent operator on $\mathcal{H}_{\mathcal{B}}$. Then $\im (p) =p(\Dom (p))$ is a closed submodule of $\mathcal{H}_{\mathcal{B}}$.
\end{lem}

\begin{proof} Let $(px_{n})$ be a Cauchy sequence in $\im (p)\subseteq\Dom (p)$ with limit $y$. Since $p^{2}x_{n}=px_{n}$ and $p$ is closed, we have that $y\in\Dom (p)$ and $py=y$, from which it follows that $y\in\im (p)$.
\end{proof}

As already mentioned, $C^*$-algebras and their modules are very well behaved within the class of operator algebras and their modules. Indeed, in $C^{*}$-modules there are no unbounded projections, as the following lemma shows.

\begin{lem}\label{C*B} Let $B$ be a $C^{*}$-algebra and $p$ a projection on $\mathcal{H}_{B}$. Then $p$ is bounded.
\end{lem}

\begin{proof} For all $x$ in the domain of $p$ we have the estimate
\[0\leq\langle(1-p)x,(1-p)x\rangle=\langle x,x \rangle - 2\langle px,x\rangle + \langle px,px \rangle = \langle x,x\rangle -\langle px,px\rangle.\]
Therefore $\langle px,px\rangle \leq \langle x,x\rangle$ and so $p$ is bounded.
\end{proof}

Later in the paper it will be necessary to consider modules which are not $C^*$ but nevertheless have a certain projectivity property. The following definition makes this idea precise.  

\begin{defn} Let $\mathcal{B}$ be an involutive operator algebra. A {\bf projective operator module} $\E\lrh\mathcal{B}$ is a right operator $\mathcal{B}$-module $\mathcal{E}$, equipped with a completely bounded $\mathcal{B}$-valued inner product, with the property that $\E$ is completely isometrically unitarily isomorphic to $\im (p)$ for some projection operator $p$ on $\mathcal{H}_{\mathcal{B}}$.
\end{defn}

\begin{rem}\textup{As opposed to the definition of stably rigged module, we require an \emph{isometric} isomorphism with $\im (p)=p\,\Dom (p)$ in the above definition. This is in view of the following proposition concerning infinite direct sums: the isometry condition is needed to prevent the norms going to infinity in the direct sum (clearly not a problem for finite sums of stably rigged modules).
}
\end{rem}

\begin{prop}\label{directsum} Let $(\mathcal{E}_{i})_{i\in I}$ be  countable family of projective operator modules. Their algebraic direct sum can be completed into a projective operator module $\bigoplus_{i\in I}\mathcal{E}_{i}$, unique up to cb-isomorphism. 
\end{prop}

\begin{proof} By assumption, each $\mathcal{E}_{i}$ is isometrically isomorphic to $\im (p_{i})\subseteq \mathcal{H}_{\mathcal{B}}$ for some $p_{i}$ a projection. As $\mathcal{H}_{\mathcal{B}}$ is a rigged module, the direct sum $\bigoplus_{i\in I}\mathcal{H}_{\mathcal{B}}$ is canonically defined in \cite{Blech} and isometrically isomorphic to $\mathcal{H}_{\mathcal{B}}$. As such, the algebraic direct sum of the modules $\im (p_{i})$ sits naturally in $\mathcal{H}_{\mathcal{B}}$ and we define $\bigoplus_{i\in I} \im (p_{i})$ to be its closure. It is straightforward to check that $\bigoplus_{i\in I}p_{i}$ defines a self-adjoint idempotent on $\mathcal{H}_{\mathcal{B}}$. We define $\bigoplus_{i\in I}\mathcal{E}_{i}$ by identifying it with $\bigoplus_{i\in I}\im (p_{i})$. In the case where $I$ is finite, this yields a space which is cb-isomorphic to the column direct sum $\bigoplus_{i\in I}^{c}\mathcal{E}_{i}$ ({\em cf}. \cite{Blech}).
\end{proof}

\begin{cor}\label{cor:finitesum} Let $(\E_{i})_{i\in I}$ be a countable family of algebraically finitely generated projective $\mathcal{B}$-modules. Then $\bigoplus_{i\in I}\E_{i}$ is a projective operator module.
\end{cor}

\begin{proof} Each of the finitely generated projective modules $\E_i$ is in particular a projective operator module (for which the projection can be chosen to be bounded). The result now follows from the previous proposition.
\end{proof}

\begin{rem}\textup{We stress that there is a difference here between internal and external direct sums. For an unbounded projection $p$ on $\mathcal{H}_{\mathcal{B}}$, the internal direct sum $p(\Dom (p))+(1-p)(\Dom (p))$ is orthogonal, but it is not closed. The external direct sum $p(\Dom (p))\oplus(1-p)(\Dom (p))$ is closed by construction and therefore cannot be isomorphic to the internal sum. This phenomenon illustrates the difference between $C^{1}$- and $C^{*}$-modules on one hand and projective operator modules on the other.}
\end{rem}
Given a right projective $\B$-module $\E$, we define the algebra $\End^{*}_{\mathcal{B}}(\mathcal{E})$ to be the collection of completely bounded maps $T:\mathcal{E}\rightarrow\mathcal{E}$ which admit an adjoint $T^{*}$, so that $\langle Te,f\rangle=\langle e,T^{*}f\rangle$ for all $e,f\in\mathcal{E}$. Note that unitary operators in $\End^{*}_{\mathcal{B}}(\E)$ are invertible but need not be isometric. We define the algebra $\K_{\mathcal{B}}(\mathcal{E})\subseteq \End^{*}_{\mathcal{B}}(\mathcal{E})$ of {\bf compact operators} to be the norm closure of the space of finite rank operators. Note that the proof of self-duality $\E\cong\E^{*}$ as described in \cite{KaLe, Mes09b} in the case of bounded projections breaks down for unbounded projections.
The next results explain the behaviour of projective operator modules upon taking their tensor products with $C^*$-modules. Indeed,  let $\mathcal{E}\lrh \mathcal{B} $ be a projective operator module, let $\mathpzc{F}\lrh C$ be a $C^{*}$-module and let $\pi:\mathcal{B}\rightarrow\End^{*}_{C}(\mathpzc{F})$ be a completely bounded homomorphism (but not necessarily a $*$-homomorphism).

\begin{prop}The Haagerup tensor product $\mathcal{E}\,\hotimes_{\mathcal{B}}\,\mathpzc{F}$ is canonically cb-isomorphic to a $C^{*}$-module. 
\end{prop}

\begin{proof} By definition we may identify $\mathcal{E}$ with a module $p(\Dom( p))\subseteq\mathcal{H}_{B}$. Replacing $\mathpzc{F}$ with the essential submodule  $\overline{\pi(\mathcal{B})\mathpzc{F}}$, the closure of the linear span of elements of the form $\pi(b)f$, we may assume that $\pi$ is a unital homomorphism. From 
\cite{Blech} we know that there is a cb-isomorphism
\[ \mathcal{H}_{\mathcal{B}}\,\hotimes_{\mathcal{B}}\,\mathpzc{F}\cong\bigoplus_{i\in\hat{\Z}}\mathpzc{F},
\]
where the left-hand side is a $C^{*}$-module. We define the closed idempotent operator $p\otimes1$ via its graph, that is 
\[\mathfrak{G}(p\otimes 1):=\mathfrak{G}(p)\,\hotimes_{\mathcal{B}}\,\mathpzc{F}\subseteq(\mathcal{H}_{\mathcal{B}}\oplus\mathcal{H}_{\mathcal{B}})\,\hotimes\,\mathpzc{F}\cong (\mathcal{H}_{\mathcal{B}}\,\hotimes\,\mathpzc{F})\oplus(\mathcal{H}_{\mathcal{B}}\,\hotimes\,\mathpzc{F}),\]
so that
\[\Dom (p\otimes 1) =(\textnormal{pr}_{1}\otimes 1)(\mathfrak{G}(p\otimes 1)),\qquad\im (p\otimes 1)=(\textnormal{pr}_{2}\otimes 1)(\mathfrak{G}(p\otimes 1)).\] 
Then 
\[\mathcal{E}\,\hotimes_{\mathcal{B}}\,\mathpzc{F}\cong \im (p\otimes 1)\subseteq\mathcal{H}_{\mathcal{B}}\,\hotimes\,\mathpzc{F},\] 
whence it is a closed submodule of a $C^{*}$-module and hence itself a $C^{*}$-module.
\end{proof}

\begin{cor}\label{co:*hom} If $\mathcal{B}$ is a Lipschitz algebra and $\pi$ is a $*$-homomorphism then $\mathcal{E}\hotimes_{\mathcal{B}}\mathpzc{F}\cong \mathpzc{E}\hotimes_{B}\mathpzc{F}$, where $B$ and $\mathpzc{E}$ are the $C^{*}$-envelopes of $\mathcal{B}$ and $\mathcal{E}$ respectively.
\end{cor}

\begin{proof} When $\pi$ is a $*$-homomorphism, it is automatically continuous (even contractive) with respect to the $C^{*}$-norm on $\mathcal{B}$. The idempotent $p\otimes 1$ is a projection which is bounded by Lemma~\ref{C*B}. Therefore \[\mathcal{E}\hotimes_{\mathcal{B}}\,\mathpzc{F}\cong(p\otimes 1)(\mathcal{H}_{\mathcal{B}}\hotimes_{\mathcal{B}}\,\mathpzc{F})=(p\mathcal{H}_{B})\hotimes_{B}\,\mathpzc{F}\cong\mathpzc{E}\hotimes_{B}\mathpzc{F},\]
and the latter is isomorphic to the standard $C^{*}$-module tensor product, {\em cf}. \cite{Blech2}. 
\end{proof}

The notion of a self-adjoint regular operator extends to projective operator modules. In this setting a densely defined self-adjoint operator $D:\mathfrak{Dom} (D)\rightarrow \mathcal{E}$ is said to be {\bf regular} if the operators $D\pm i$ are surjective. In the $C^{*}$-situation, the resolvents $(D\pm i )^{-1}$ are automatically contractive and so there it is sufficient to require that $D\pm i$ have dense range. To construct regular operators in practice the following lemma is useful. It is proved in the same way as in \cite{Mes09b}.

\begin{lem}\label{pmi} Let $D$ be a densely defined closed symmetric operator on $\mathcal{E}$. Then the following are equivalent: 
\begin{enumerate}[\hspace{0.5cm} (i)]
\item $D$ is self-adjoint and regular; 
\item $\im (D\pm i)$ are dense in $\mathcal{E}$ and $(D\pm i)^{-1}$ are completely bounded for the operator space norm on $\mathcal{E}$.
\end{enumerate} 
 If either (and hence both) of these conditions holds, then $(D\pm i)^{-1}\in\End^{*}_{\mathcal{B}}(\mathcal{E})$.
\end{lem}

\begin{proof} If $D$ is self-adjoint and regular, then the resolvents $D\pm i:\Dom (D)\rightarrow \mathcal{E}$ are surjective. They are also injective by a standard argument. The inverses $(D\pm i)^{-1}:\mathcal{E}\rightarrow\Dom (D)$ are mutually adjoint, whence they must be bounded and adjointable. 

To obtain the converse, we denote by $r_{\pm}$ the extensions of the operators $(D\pm i)^{-1}$ from $\im (D\pm i)$ to $\E$.  Given a sequence $(x_{n})\subset \im(D\pm i)$ converging to $x\in\mathcal{E}$, boundedness implies that $r_{\pm}x_{n}\rightarrow r_{\pm}x$ and that $D(r_{\pm}x_{n})=(1\mp i r_{\pm})x_{n}$ is convergent. Since $D$ is closed, we deduce that $r_{\pm}x\in \Dom (D)$ and so $\im (r_{\pm})\subseteq \Dom (D)$. Therefore $r_{\pm}=(D\pm i)^{-1}$ and $D\pm i:\Dom D\rightarrow \mathcal{E}$ are bijective, so $\Dom(D)=\im (D+i)^{-1}$. Now for all $e\in\E$ and each $f\in \Dom (D^{*})$ we compute that
\[\langle (D+i)^{-1}e,D^{*}f\rangle=\langle D(D+i)^{-1}e,f\rangle=\langle e,f\rangle-\langle i(D+i)^{-1}e,f\rangle\]
and hence that
\[\langle e,f\rangle=\langle (D+i)^{-1}e,D^{*}f\rangle+\langle i(D+i)^{-1}e,f\rangle=\langle e,(D-i)^{-1}(D^{*}-i)f\rangle. \]
It follows that $f=(D-i)^{-1}(D^{*}-i)f\in\Dom (D)$ and so $D$ is self-adjoint, as required.
\end{proof}

Immediately we are led to the following result, which gives the aforementioned practical characterization of self-adjoint regular operators on projective operator modules.

\begin{prop}\label{selfreg} Let $\mathcal{E}$ be a projective operator module and $D:\Dom (D)\rightarrow \mathcal{E}$ a self-adjoint regular operator on $\mathcal{E}$. Then $\mathfrak{G}(D)$ is a projective operator module and $\mathfrak{G}(D)\oplus v\, \mathfrak{G}(D)\cong \mathcal{E}\oplus\mathcal{E}$.
\end{prop}

\begin{proof} By Lemma \ref{pmi} the map 
\[u:\mathfrak{G}(D)\to \E,\qquad \begin{pmatrix}e\\ De\end{pmatrix}\mapsto (D+i)e,\] 
is a bijection whose adjoint is given by
\[u^{*}:\E\to\mathfrak{G}(D),\qquad e\mapsto \begin{pmatrix}(D+i)^{-1}e\\ D(D+i)^{-1}e\end{pmatrix}.\] It is straightforward to check that these maps preserve the inner product and hence that the map $u:\mathfrak{G}(D)\rightarrow \mathcal{E}$ is a unitary isomorphism. Therefore $\mathfrak{G}(D)$ is a projective operator module. Moreover, the matrix
\[g=\begin{pmatrix}
(D+i)^{-1} & -D(D+i)^{-1} \\ D(D+i)^{-1} & (D+i)^{-1}
\end{pmatrix}\]
defines a unitary operator mapping $\mathcal{E}\oplus\mathcal{E}$ onto $\mathfrak{G}(D)\oplus v\,\mathfrak{G}(D)$.
\end{proof}

\subsection{Lipschitz modules and connections} In the previous section we studied projective modules over an arbitrary involutive operator algebra $\B$. However, in the special case where $\B$ is the Lipschitz algebra associated to a given spectral triple $(B,\mathcal{H},D)$ there is a spectrally invariant dense embedding $\mathcal{B}\hookrightarrow B$. For such algebras we will restrict to the subclass of projections that admit a decomposition as a countable direct sum of projections in $\End^{*}_{\B}(\H_{\B})$. In other words, we restrict to direct sums of stably rigged modules. The main motivation for this restriction is to deal with connections and regularity of the operators they induce. In order to establish this regularity and to deal with connections on projective operator modules we need also to modify the notion of universal differential forms used in \cite{Mes09b}.

Recall that, given a spectral triple $(B,\mathcal{H},D)$, the associated space of one-forms $\Omega^{1}_{D}(\B)$ defined in eq.~\eqref{ccalc} 
is a $\mathcal{B}$-bimodule. It is in fact a left module over the $C^{*}$-algebra $B$: since $\mathcal{B}$ is dense in $B$, we can choose for each $b\in B$ a sequence $(b_{i})\subset\mathcal{B}$ with $b_{i}\rightarrow b$. Then for each $\omega\in\Omega^{1}_{D}(\B)$, the sequence $(b_{i}\omega)\subset\Omega^{1}_{D}(\B)$ is Cauchy and hence has a limit in $\Omega^{1}_{D}(\B)$ which, by uniqueness of limits, must be $b\omega$. The map
\[ \mathcal{B}\rightarrow \Omega^{1}_{D}(\B),\qquad b \mapsto [D,b],\]
is thus a graded bimodule derivation into a $(B,\mathcal{B})$-bimodule. This motivates the following  

\begin{defn} The space of {\bf universal one-forms} $\Omega^{1}(B,\mathcal{B})$ over a graded Lipschitz algebra $\mathcal{B}$ is defined to be the kernel of the map
\[m: B\,\hotimes\,\mathcal{B}\rightarrow B,\qquad a\otimes b \mapsto a\gamma(b)\]
where $\gamma$ denotes the grading on $\mathcal{B}$. This is a $(B,\mathcal{B})$-bimodule map when $B$ is viewed as a $(B,\gamma(\mathcal{B}))$-bimodule. The map 
\[\d:\mathcal{B}\rightarrow\Omega^{1}(B,\mathcal{B}), \qquad
b \mapsto 1\otimes b-\gamma(b)\otimes 1,\]
is called the {\bf universal derivation}. 
\end{defn}

\begin{prop}The derivation $\d$ is indeed universal: for any cb-derivation $\delta:\mathcal{B}\rightarrow M$ into a $(B,\mathcal{B})$ cb-operator bimodule $M$, there is a unique completely bounded $(B,\mathcal{B})$-bimodule map $j_{\delta}:\Omega^{1}(B,\mathcal{B})\rightarrow M$ such that $1\cdot\delta(b)\cdot 1=j_{\delta}\circ \d(b)$. 
\end{prop}
\begin{proof}
By replacing $M$ by $1\cdot M \cdot 1$ and $\delta$ by $1\cdot \delta \cdot 1$ we may assume that $M$ is an essential bimodule and $\delta(1)=0$.
The map 
\[
B\times\mathcal{B}\rightarrow M,\qquad (a,b)\mapsto a\,\delta(b),
\]
is bilinear and completely bounded, whence it determines a unique completely bounded linear map on the Haagerup tensor product $B\,\hotimes\,\mathcal{B}\rightarrow M $. We define $j_{\delta}$ to be the restriction of this map to $\Omega^{1}(B,\mathcal{B})$. Then for each
$\omega=\sum a_{i}\otimes b_{i}\in\Omega^{1}(B,\mathcal{B})$ we have
\[\omega=\sum a_{i}\otimes b_{i}=\sum a_{i}\otimes b_{i}-a_{i}\gamma(b_{i})\otimes 1=\sum a_{i}\,\d b_{i},\]
since  $\sum a_{i}\gamma(b_{i})=0$ by definition of $\Omega^{1}(B,\mathcal{B})$. Thus $j_{\delta}$ is 
determined by the condition $j_{\delta}(\d b)=1\cdot\delta(b)-b\cdot\delta(1)$. It is obvious that $j_{\delta}$ is a left $B$ -module map. For the right $\mathcal{B}$-module structure we have
\[\begin{split}
j_{\delta}(\omega b) &=j_{\delta}\left(\sum a_{i}\otimes b_{i}b\right)
=\sum a_{i}\delta(b_{i}b)\\
&=\sum a_{i}\gamma(b_{i})\delta(b)+a_{i}\delta(b_{i})b \\
&=\sum a_{i}\delta(b_{i})b
=j_{\delta}(\omega)b,
\end{split}\]
from which it follows that $j_{\delta}$ is a $(B,\mathcal{B})$-bimodule map.
\end{proof}

 
Recall \cite{Mes09b} that $\Omega^{1}(\B):=\ker (m: \B\hotimes \B\rightarrow \B)$ is the universal module for $\B$-bimodule derivations, in the sense that for every $\B$-module derivation $\delta:\B\rightarrow M$ there is a unique bimodule map $\j_{\delta}:\Omega^{1}(\B)\rightarrow M$ satisfying $1\cdot\delta(b)\cdot 1=j_{\delta}\circ \d(b)$. Since every $(B,\B)$-bimodule is a $\B$-bimodule, the following observation relating the the two universal structures is useful.
\begin{lem}The natural multiplication map $$B\,\hotimes_{\B}\,\Omega^{1}(\mathcal{B}) \to \Omega^{1}(B,\mathcal{B}), $$
is a complete isometry compatible with the induction map $j_{\delta}$, for any cb-$(B,\B)$-bimodule derivation, in the sense that $j_{\delta}^{(B,\B)}=1\otimes j_{\delta}^{\B}$.
\end{lem}
\begin{proof}The multiplication is completely contractive, so it suffices to show that the inverse
 \[\sum a_{i}\otimes b_{i}\mapsto \sum a_{i}\otimes 1 \otimes b_{i},\]
is completely contractive as well. By definition, the Haagerup norm can be computed as
\[\|\sum a_{i}\otimes b_{i}\|_{h}=\inf\{ \|(a_{i}')^{t}\|\|(b_{i}')\|:\sum a_{i}'\otimes b_{i}'=\sum a_{i}\otimes b_{i}\},\]
where $(a_{i}')^{t}$ denotes a row vector and $(b'_{i})$ a column vector. Thus we can compute 
\[\begin{split}\|\sum a_{i}\otimes 1 \otimes b_{i}\|&=\inf \|(a'_{i})^{t}_{i}\|\|(1\otimes b_{i})_{i}\| \\
&=\inf \|(a'_{i})^{t}_{i}\|\|(b_{i})_{i}\| \\
&=\|\sum a_{i}\otimes b_{i}\|_{h} ,\end{split}\]
where the middle equality follows 
since $b\mapsto 1\otimes b$ is a complete isometry $\B\mapsto B\otimes \B$. Compatibility with the induction maps is immediate since $j_{\delta}^{\B}$ is a bimodule map and $j_{\delta}^{(B,\B)}$ is universal, so must therefore coincide with $1\otimes j_{\delta}^{\B}$.\end{proof}

 The symmetric module of forms $\Omega^{1}(\B)$ carries an involution $\omega\mapsto \omega^{*}$ induced by $a\otimes b\mapsto b^{*}\otimes a^{*}$. We have the following completely isometric maps relating the various structures.
\begin{lem}\label{iso} The Haagerup tensor product $\B\hotimes \B$ is an involutive operator algebra for the involution $a\otimes b\mapsto b^{*}\otimes a^{*}$. This involution restricts to $\Omega^{1}(\B)$. The natural multiplication map
\begin{equation}\label{mult}
\mathcal{H}_{\mathcal{B}}\,\hotimes_{\mathcal{B}}\,B\rightarrow\mathcal{H}_{B},
\end{equation}
is a complete isometry. Consequently, for a projective operator module $\E$ there are pairings
\begin{equation}\label{pairing}\mathcal{E}\times \mathcal{E}\hotimes_{\B}\Omega^{1}(B,\B)\rightarrow B\hotimes_{\B}\Omega^{1}(\B),\quad  \mathcal{E}\hotimes_{\B}\Omega^{1}(B,\B)\times \mathcal{E}\rightarrow \Omega^{1}(\B)\hotimes_{\B} B,\end{equation}via $\langle e,f\otimes\omega\rangle:=\langle e,f\rangle\omega$ and $\langle e\otimes\omega, f\rangle:=\omega^{*}\langle e,f\rangle$.
\end{lem}

\begin{proof} The involution is completely anti-isometric:
\[\begin{split} \|\sum b_{i}^{*}\otimes a_{i}^{*}\|_{h}^{2}&=\inf\{ \|\sum\pi_{D}(b_{i}^{*})\pi_{D}(b_{i}^{*})^{*}\|\| \sum\pi_{D}(a_{i}^{*})^{*}\pi_{D}(a_{i})\|\} \\
 & = \inf\{ \|\sum v\pi_{D}(b_{i})^{*}\pi_{D}(b_{i})v^{*}\|\| \sum v\pi_{D}(a_{i})\pi_{D}(a_{i})^{*}v^{*}\|\} \\
&=  \inf\{ \|\sum \pi_{D}(b_{i})^{*}\pi_{D}(b_{i})\|\| \sum v\pi_{D}(a_{i})\pi_{D}(a_{i})^{*}\|\} \\
&= \|\sum a_{i}\otimes b_{i}\|_{h}^{2}.\end{split}\]
It is straightforward to check that the involution  preserves $\Omega^{1}(\B)$. 
The map \eqref{mult} is completely contractive, since  the inclusion map $\mathcal{B}\rightarrow B$ and the multiplication  are so. Moreover, the inverse map $h\otimes b\mapsto h\otimes 1\otimes b$ is completely contractive as well. Consequently, there is a completely isometric isomorphism
\[\E\hotimes_{\B}\Omega^{1}(B,\B)\xrightarrow{\sim} \E\hotimes_{\B} B\hotimes_{\B}\Omega^{1}(\B)\xrightarrow{\sim} \mathpzc{E}\hotimes_{\B} \Omega^{1}(\B),\]
and the formulae
\[\langle e,f\otimes\omega\rangle:=\langle e,f\rangle\omega,\quad \langle e\otimes\omega, f\rangle:=\omega^{*}\langle e,f\rangle,\]
give the well-defined pairings \eqref{pairing}.
\end{proof}

The spaces $\Omega^{1}(\B)\hotimes_{\B}B$ and $B\hotimes_{\B}\Omega^{1}(\B)$ embed completely contractively into the $B$-bimodule $B\hotimes_{\B}\Omega^{1}(\B)\hotimes_{\B}B$ via $b\otimes \omega\mapsto b\otimes\omega\otimes 1$ and $\omega\otimes b\mapsto 1\otimes \omega\otimes b$.

\begin{defn}\label{connection} A {\bf connection} on a projective operator module $\E$ is a completely bounded map $\nabla:\E\rightarrow \mathpzc{E}\hotimes_{B}\Omega^{1}(B,\B)$ satisfying the Leibniz rule
$$\nabla(eb)=\nabla(e)b + e\otimes \d b$$
for all $e\in \E$ and all $b\in\B$.  The connection $\nabla$ is {\bf Hermitian} if it satisfies the equation 
$$\langle e,\nabla (f)\rangle\otimes 1 + 1\otimes\langle \nabla(e),f\rangle = 1\otimes\d \langle e,f\rangle\otimes 1 \in B\hotimes_{\B}\Omega^{1}(\B)\hotimes_{\B} B$$
for all $e,f\in\E$. 
\end{defn}

The reason for introducing the bimodule $B\hotimes_{\B}\Omega^{1}(\B)\hotimes_{\B} B$ is that to state the property of being Hermitian, we need to map the forms $\langle e,\nabla (f)\rangle$ and $\langle \nabla(e),f\rangle$ into the same space. We will later see that this definition is compatible with the notion of induced operator.

\begin{prop} Let $p$ be a projection operator on $\mathcal{H}_{\mathcal{B}}$. Then
the module $\im (p)\subseteq \mathcal{H}_{\mathcal{B}}$ admits a completely contractive Hermitian connection $$\nabla:\im (p)\rightarrow \im (p)\,\hotimes_{\mathcal{B}}\,\Omega^{1}(B,\mathcal{B}).$$ 
\end{prop}

\begin{proof} Consider the Grassmann connection 
\[\begin{split} \d:\mathcal{H}_{\mathcal{B}}&\rightarrow\mathcal{H}_{\mathcal{B}}\,\tilde{\otimes}_{\mathcal{B}}\,\Omega^{1}(B,\mathcal{B})\\
(b_{i})_{i\in\Z} &\mapsto (\d b_{i})_{i\in\Z},\end{split}\]
which is Hermitian in the above sense because
\[\begin{split} 
\langle (a_{i}),(\d b_{i})\rangle\otimes 1 -1\otimes \langle (\d a_{i}),(b_{i})\rangle 
 &=\sum a_{i}^{*}\otimes 1 \otimes b_{i}\otimes 1- a_{i}^{*}\otimes b_{i}\otimes 1\otimes 1\\ &\quad\quad\quad -1\otimes a_{i}^{*}\otimes 1\otimes b_{i}+1\otimes 1\otimes  a_{i}^{*}\otimes b_{i} \\
&= \sum 1\otimes 1\otimes a_{i}^{*}b_{i}\otimes 1 - 1\otimes a_{i}^{*}b_{i}\otimes 1 \otimes 1 \\
&=1\otimes\d \langle (a_{i}), (b_{i})\rangle \otimes 1.
\end{split}\]
We wish to show that the compression $\nabla:=(p\otimes 1) \d p: p\H_{\B}\rightarrow p\H_{B}\hotimes_{B}\Omega^{1}(B,\B)$ is completely contractive and Hermitian. It is obvious that the restriction \[\d:\im(p)\rightarrow \mathcal{H}_{\mathcal{B}}\,\tilde{\otimes}_{\mathcal{B}}\,\Omega^{1}(B,\mathcal{B})\]
is a completely contractive map. Thus it remains to show that the operator 
\[p\otimes 1 : \mathcal{H}_{\mathcal{B}}\,\tilde{\otimes}_{\mathcal{B}}\,\Omega^{1}(B,\mathcal{B})\rightarrow \mathcal{H}_{\mathcal{B}}\,\tilde{\otimes}_{\mathcal{B}}\,\Omega^{1}(B,\mathcal{B})\]
is completely contractive. This follows from the completely isometric isomorphism
\[\mathcal{H}_{\mathcal{B}}\,\tilde{\otimes}_{\mathcal{B}}\,\Omega^{1}(B,\mathcal{B})\rightarrow \mathcal{H}_{B}\,\tilde{\otimes}_{B}\,\Omega^{1}(B,\mathcal{B}),\]
from Lemma \ref{iso} and from the fact that $p$ extends to a bounded projection on $\mathcal{H}_{B}$ ({\em cf}. Lemma~\ref{C*B}). Finally, the compression is Hermitian because for $e_{1}=pe_{1}, e_{2}=pe_{2}\in p\,\Dom (p)\subset \H_{\B}$ we have
\[\langle e_{1}, p\d e_{2}\rangle\otimes 1 -1\otimes \langle p\d e_{1}, e_{2}\rangle=\langle e_{1},\d e_{2}\otimes 1\rangle - 1\otimes \langle \d e_{1}, e_{2}\rangle=1\otimes \d\langle e_{1},e_{2}\rangle\otimes 1,\]
since $\d$ is Hermitian.
\end{proof}

Having dealt with the generalities of projective operator modules, we now restrict to the class of such modules that we shall need in the present paper.

\begin{defn} Let $\B$ be a Lipschitz algebra. An inner product operator module $\E\lrh\B$ over $\B$ is a {\bf Lipschitz module} if there is a countable set $I$ and a collection of stably rigged modules $\{\E_{i}:i\in I\}$, such that $\E$ is unitarily  isometrically isomorphic to the direct sum $\bigoplus_{i\in I}\E_{i}$.
\end{defn}

We stress that a Lipschitz module is in particular a projective operator module. To state the main theorem below, we need the following simple lemma on direct sums of regular operators between $C^{*}$-modules. Given a countable collection of regular operators
$$
D_{i}:\Dom (D_{i})\rightarrow \mathpzc{F}_{i}, \quad \Dom (D_{i})\subset \mathpzc{E}_{i},\qquad (i\in I),
$$ 
we define the \emph{algebraic direct sum operator}
$D:=\oplus_{i\in I}^{\textnormal{alg}}D_{i}$ to be the operator acting by $D_{i}$ in the respective component of the algebraic direct sum of the $\mathpzc{E}_{i}$. We remark here that, although have already introduced the unadorned symbol $\otimes$ to represent algebraic tensor products, we have so far no such notation for algebraic direct sums: we therefore write $\oplus^{\textnormal{alg}}$ when we feel the need to stress an algebraic sum, whereas $\oplus$ denotes a completed direct sum.

\begin{lem}\label{regsum} Let $D_{i}:\Dom (D_{i})\rightarrow \mathpzc{F}_{i},$ $i\in I$, be a countable collection of regular operators. Then 
\[\oplus_{i\in I}^{\textnormal{alg}}D_{i}:\bigoplus_{i\in I}^{\textnormal{alg}}\Dom (D_{i})\rightarrow \bigoplus_{i\in I}\mathpzc{F}_{i}\]
is closable and its closure is a  regular operator $D:=\oplus_{i\in I}D_{i}$ between $\mathpzc{E}:=\bigoplus_{i\in I}\mathpzc{E}_{i}$ and $\mathpzc{F}:=\bigoplus_{i\in I}\mathpzc{F}_{i}$. The map
\[ u_{D}:\bigoplus_{i\in I}\mathfrak{G}(D_{i}) \rightarrow \mathfrak{G}(D)\subset \mathpzc{E}\oplus\mathpzc{F},\qquad 
\begin{pmatrix} e_{i} \\ D_{i}e_{i}\end{pmatrix}_{i\in I}  \mapsto \begin{pmatrix}  (e_{i})_{i\in I} \\ (D_{i}e_{i})_{i\in I} \end{pmatrix},\]
is a unitary isomorphism of $C^{*}$-modules and $D^{*}=\oplus_{i\in I} D_{i}^{*}$. If each $D_{i}$ is self-adjoint, then the direct sum extends to an essentially self-adjoint regular operator. If $D_{i}\in\End^{*}_{B}(\mathpzc{E}_{i},\mathpzc{F}_{i})$ and $\sup_{i\in I}\|D_{i}\|<\infty$, then $\bigoplus_{i\in I} D_{i}\in \End_{B}^*(\mathpzc{E},\mathpzc{F})$.
\end{lem}

\begin{proof} By definition of the direct sum completion, it is immediate that $\oplus^{\textnormal{alg}}D_{i}$ is closable with the indicated graph isomorphism. We \emph{define} $D^{*}$ to be 
$\oplus_{i\in I} D_{i}^{*}$. To prove regularity we argue as follows. For each $i$ separately, regularity of $D_{i}$ gives an isomorphism
\[\begin{split}\mathfrak{G}(D_{i})\oplus v\mathfrak{G}(D_{i}^{*})&\xrightarrow{+} \mathpzc{E}_{i}\oplus \mathpzc{F}_{i},\\
\left(\begin{pmatrix} x \\ D_{i}x\end{pmatrix},\begin{pmatrix} -D_i^{*} y \\ y\end{pmatrix} \right)&\mapsto \begin{pmatrix} x-D_{i}^{*}y \\ D_{i}x + y\end{pmatrix}.\end{split}\]
Therefore we find that
\[ u_{D}^{*}\oplus u^{*}_{D^{*}}\left(\mathfrak{G}(D)\oplus v\mathfrak{G}(D^{*})\right)\supset \bigoplus_{i\in I}^{\textnormal{alg}}\left(\mathfrak{G}(D_{i})\oplus v\mathfrak{G}(D_{i}^{*})\right)\xrightarrow{+}\bigoplus_{i\in I}^{\textnormal{alg}}\left( \mathpzc{E}_{i}\oplus \mathpzc{F}_{i}\right),\]
and therefore $\mathfrak{G}(D)\oplus v\mathfrak{G}(D^{*})$ is dense in $\left(\bigoplus_{i\in I} \mathpzc{E}_{i}\oplus \bigoplus_{i\in I}\mathpzc{F}_{i}\right)$, which implies that $D$ is regular. The statement about self-adjointness follows by a similar argument and the statement about uniform bounded sequences is immediate.
\end{proof}

Finally we arrive at the main result of this section. Suppose that we are given: 
\begin{enumerate}[\hspace{0.5cm} (i)]
\item a Lipschitz module $\mathcal{E}\cong\bigoplus_{i\in I}\E_{i}\lrh\mathcal{B}$ with Hermitian connection $\nabla:\mathcal{E}\rightarrow \mathcal{E}\,\tilde{\otimes}_{\mathcal{B}}\,\Omega^{1}(B,\mathcal{B})$; 
\item a $C^{*}$-module $\mathpzc{F}\lrh C$ equipped with a self-adjoint regular operator $T:\Dom (T)\rightarrow \mathpzc{F}$; 
\item a $*$-homomorphism $\pi:\mathcal{B}\rightarrow \End^{*}_{C}(\mathpzc{F})$ such that $b\mapsto [T,\pi(b)]$ is a cb-derivation on $\mathcal{B}$. 
\end{enumerate}
Then we define a linear operator $1\otimes_{\nabla}T:\E\otimes_{\B}\Dom (T)\to\E\,\hotimes_{\B}\mathpzc{F}=\mathpzc{E}\hotimes_{B}\mathpzc{F}$ on the algebraic tensor product, by
\[ (1\otimes_{\nabla}T) (e\otimes f):=\gamma(e)\otimes Tf + \nabla_{T} (\gamma(e))f, \qquad e\in\E,~f\in \mathpzc{F},\]
with $\nabla_{T}:\mathcal{E}\rightarrow \mathcal{E}\,\hotimes_{\mathcal{B}}\,\Omega^{1}_{T}(\mathcal{B})$ the connection on $\E$ induced by $\nabla$.
 
\begin{lem} The operator $1\otimes_{\nabla}T$ is symmetric on its domain.
\end{lem}
\begin{proof} Since $\nabla$ is Hermitian we have
\[ \langle e_{1},\nabla(e_{2})\rangle \otimes 1 -1\otimes\langle \nabla(e_{1}), e_{2}\rangle =1\otimes \d\langle e_{1},e_{2}\rangle\otimes 1 \]
on the level of universal forms. After applying the map
\[\begin{split} 1\otimes j_{T}\otimes 1: B\otimes_{\B}\Omega^{1}(\B)\otimes_{\B}1 &\rightarrow \End^{*}_{C}(\mathpzc{F}), \\
a\otimes \d b\otimes c &\mapsto \pi(a)[\pi(1)T\pi(1),\pi(b)]\pi (c),\end{split}
\]
this amounts to
\[\langle e_{1},\nabla_{T}(e_{2})\rangle\pi(1) -\pi(1)\langle \nabla_{T}(e_{1}), e_{2}\rangle =\pi(1)[\pi(1)T\pi(1),\pi(\langle e_{1},e_{2}\rangle)]\pi(1)=\pi(1)[T,\langle e_{1},e_{2}\rangle ] \pi(1).\]
Replacing $\mathpzc{F}$ with $\pi(1)\mathpzc{F}$ if necessary, this in turn gives
\[\langle e_{1}\otimes f_{1},\nabla_{T}(e_{2})f_{2}\rangle -\langle \nabla_{T}(e_{1})f_{1}, e_{2}\otimes f_{2}\rangle =\langle f_{1},[T,\langle e_{1},e_{2}\rangle ]
f_{2}\rangle\]
and we can compute
\[\begin{split} \langle e_{1}\otimes f_{1}, e_{2}\otimes Tf_{2}+\nabla_{T}(e_{2})f_{2}\rangle & =\langle f_{1},\langle e_{1},e_{2}\rangle Tf_{2}\rangle +\langle f_{1}, \langle e_{1}, \nabla_{T} (e_{2})\rangle f_{2}\rangle \\
&=\langle f_{1},\langle e_{1},e_{2}\rangle Tf_{2}\rangle + \langle f_{1}, \langle  \nabla_{T}(e_{1}), e_{2}\rangle f_{2}\rangle + \langle f_{1},[T,\langle e_{1},e_{2}\rangle ]f_{2}\rangle \\
&=\langle Tf_{1},\langle e_{1}, e_{2}\rangle f_{2}\rangle +  \langle f_{1}, \langle  \nabla_{T}(e_{1}), e_{2}\rangle f_{2}\rangle \\
&=\langle e_{1}\otimes Tf_{1} + \nabla_{T}(e_{1})f_{1}, e_{2}\otimes f_{2}\rangle \end{split}\]
on the algebraic tensor product.
\end{proof}


\begin{thm}\label{indgraph} The operator $1\otimes_{\nabla}T$ is essentially self-adjoint and regular on $\mathpzc{E}\,\hotimes_{B}\mathpzc{F}$. The map
 \[ g:\left(\bigoplus_{i\in I}^{\textnormal{alg}}\E_{i}\right)\,\otimes_{\mathcal{B}}\,\mathfrak{G}(T) \rightarrow\mathfrak{G}(1\otimes_{\nabla}T), \qquad e\otimes \begin{pmatrix}f \\Tf \end{pmatrix} \mapsto \begin{pmatrix} e\otimes f \\ (1\otimes_{\nabla}T)(e\otimes f) \end{pmatrix},\] defined on the algebraic direct sum, extends to an everywhere defined  adjointable operator with dense range
\[g:\mathcal{E}\hotimes_{\B}\mathfrak{G}(T)\rightarrow \mathfrak{G}(1\otimes_{\nabla}T),\]
given by the same formula.
\end{thm}

\begin{proof} First we observe that, since $\mathcal{E}$ is isomorphic to $\im (p)$ for a direct sum projection operator $p=\bigoplus_{i\in I}p_{i}$ on $\bigoplus_{i\in I}\mathcal{H}_{\mathcal{B}}\cong \H_{\B}$, the difference
\[\nabla-p\d p:\E\rightarrow \mathpzc{E}\hotimes_{B}\Omega^{1}(B,\B)\]
is completely bounded (\textit{cf}. Definition~\ref{connection}). The Haagerup tensor product is functorial for completely bounded module maps, so
\[(\nabla-p\d p)\otimes 1:\E\hotimes_{\B}\mathpzc{F}\rightarrow \mathpzc{E}\hotimes_{B}\Omega^{1}(B,\B)\hotimes_{\B}\mathpzc{F}\]
is completely bounded. Composition with the natural map induced by operator multiplication
\[\mathpzc{E}\hotimes_{B}\Omega^{1}(B,\B)\hotimes_{\B}\mathpzc{F}\rightarrow \mathpzc{E}\hotimes_{\B}\mathpzc{F}\]
thus gives a skew-adjoint operator $R:\mathpzc{E}\hotimes_{B}\mathpzc{F}\rightarrow \mathpzc{E}\hotimes_{B}\mathpzc{F}$.
Therefore, it suffices to prove the statement for the connection $p\d p$.

For each $\mathcal{E}_{i}$ separately one shows just as in \cite{Mes09b} ({\em cf}. \cite{KaLe}) that, with $\nabla_{i}$ the Grassmann connection on $\mathcal{E}_{i}$, the operator $t_{i}:=1\otimes_{\nabla_{i}}T$ is self-adjoint and regular. In this case the map  
\[g_{i}:\mathcal{E}_{i}\,\hotimes_{\mathcal{B}}\,\mathfrak{G}(T)\rightarrow\mathfrak{G}(1\otimes_{\nabla_{i}}T)\]
is a topological isomorphism. The operator $t:=1\otimes_{\nabla}T$ can be identified with the algebraic direct sum operator
\[\oplus_{i\in I}^{\textnormal{alg}} t_{i}:\bigoplus_{i\in I}^{\textnormal{alg}} \mathcal{E}_{i}\otimes_{ \mathcal{B}}\Dom (T)\rightarrow \mathpzc{E}_{i}\hotimes_B\mathpzc{F}.\]
Since each $t_{i}$ is self-adjoint in $\mathpzc{E}_{i}\hotimes_{B}\mathpzc{F}$, the direct sum is essentially self-adjoint by Lemma~\ref{regsum}. 

In fact we claim that $g_{i}$ is contractive. The this end, denote by $\d$ the Grassmann connection on $\mathcal{H}_{\B}$. 
The map
\[\begin{split}  u:\mathcal{H}_{\B}\hotimes_{\B}\mathfrak{G}(T)&\rightarrow \mathfrak{G}(1\otimes_{\d}T) \\
e\otimes\begin{pmatrix} f\\ Tf \end{pmatrix} & \mapsto \begin{pmatrix} e\otimes f \\ (1\otimes_{\d}T)e\otimes f \end{pmatrix}, \end{split}\]
is unitary, \textit{cf}. \cite[Theorem 5.4.1]{Mes09b}. Since $\mathcal{E}_{i}=p_{i}\mathcal{H}_{\B}$, the operator $1\otimes_{\nabla_{i}}T$ equals the operator $p_{i}(1\otimes_{\d}T)p_{i}$ on its domain, so we can write
\[\begin{split} g_{i}\left(\sum_{k=1}^{n}e_{k}\otimes\begin{pmatrix} f_{k} \\ T f_{k}\end{pmatrix}\right) &=\sum_{k=1}^{n} \begin{pmatrix} e_{k}\otimes f_{k} \\  p_{i} (1\otimes_{\d} T) e_{k}\otimes f_{k} \end{pmatrix}\\
&=\begin{pmatrix} p_{i} & 0 \\ 0 & p_{i} \end{pmatrix} \sum_{k=1}^{n}\begin{pmatrix} e_{k}\otimes f_{k} \\ (1\otimes_{\d} T) e_{k}\otimes f_{k} \end{pmatrix}\\
&=\begin{pmatrix} p_{i} & 0 \\ 0 & p_{i} \end{pmatrix} u \left (\sum_{k=1}^{n}e_{k}\otimes\begin{pmatrix} f_{k} \\ T f_{k}\end{pmatrix}\right),\end{split}\]
and thus $g_{i}$ factors as $p_{i}u$, which shows that $\|g_{i}\|\leq 1$.

The direct sum of the $g_{i}$ defines a map
\[\bigoplus_{i\in I}^{\textnormal{alg}} g_{i}:\bigoplus_{i\in I}^{\textnormal{alg}}\E_{i}\hotimes_{ \mathcal{B}}\, \mathfrak{G}(T_{i})\rightarrow \bigoplus_{i\in I}^{\textnormal{alg}}\mathfrak{G}(t_{i})\subset\mathfrak{G}(t),\]
since each $g_{i}$ extends to the Haagerup tensor product $\mathcal{E}_{i}\hotimes_{\B}\mathfrak{G}(T)$.
Since $\sup_{i}\|g_{i}\|\leq 1$ by the above discussion, the direct sum extends to an everywhere defined  adjointable operator
\[g:\mathcal{E}\hotimes_{\B}\mathfrak{G}(T)\rightarrow \mathfrak{G}(1\otimes_{\nabla}T).\]
Since each $g_{i}$ is invertible, the direct sum has dense range, which proves the statement on the graph map.
\end{proof}

\subsection{The unbounded Kasparov product} Having now obtained a theory of connections on Lipschitz modules, we are now ready to describe the product of KK-cycles in the unbounded setting. The key ingredient in doing so will be the following

\begin{defn}\label{de:lipcyc}
A {\bf Lipschitz cycle} between spectral triples $(A,\mathcal{H}_{1},D_{1})$ and $(B,\mathcal{H}_{2},D_{2})$ is a triple $(\mathcal{E},S,\nabla)$ consisting of:
\begin{enumerate}[\hspace{0.5cm} (i)]
\item a Lipschitz $(\mathcal{A},\mathcal{B})$-bimodule $\mathcal{E}$;
\item an odd self-adjoint regular operator $S$ in $\mathcal{E}$ such that $(S\pm i)^{-1}\in\K_{\mathcal{B}}(\mathcal{E})$;
\item the map $a\mapsto [S,a] \in \End^{*}_{B}(\mathpzc{E})$ is a cb-derivation $\mathcal{A}\rightarrow \End^{*}_{B}(\mathpzc{E})$. In particular, the commutators $[S,a]$ extend to bounded operators on the enveloping $C^{*}$-module $\mathpzc{E}$.
\item an even, completely bounded connection $\nabla:\mathcal{E}\rightarrow\mathcal{E}\tilde{\otimes}_{\mathcal{B}}\,\Omega^{1}(B,\mathcal{B})$ such that $[\nabla, S]=0$.
\end{enumerate}
Given a pair of Lipschitz algebras $\A,\B$, we denote by $\Psi_{0}^{\ell}(\mathcal{A},\mathcal{B})$ the set of unitary equivalence classes of $(\A,\B)$ Lipschitz cycles.
\end{defn}

As in \cite{Mes09b, KaLe}, the commutator condition on the connection in the latter definition can be weakened, requiring more intricate self-adjointness proofs. This will be dealt with elsewhere: the only examples we shall encounter in the present paper are commuting connections.
Given a Lipschitz cycle, the operator $S$ extends to the $C^{*}$-completion $\mathpzc{E}\cong \mathcal{E}\,\hotimes_{\mathcal{B}}\,B$ as $S\otimes 1$. The pair $(\mathpzc{E},S)$ is an unbounded KK-cycle for $(A,B)$. 

\begin{defn} Two self-adjoint regular operators $s$ and $t$ on a Lipschitz module $\E\lrh\B$ are said to {\bf anti-commute} if $\im \left((s+ i)^{-1}(t+i)^{-1}\right)=\im \left((t+ i)^{-1}(s+i)^{-1}\right)$ and $st+ts=0$ on this submodule.
\end{defn}

For an anti-commuting pair, the resolvent $(s\pm i)^{-1}$ preserves the domain of $t$ and the resolvent $(t\pm i)^{-1}$ preserves the domain of $s$. We have the following relations:
\begin{equation}\label{com1}
t(s+i)^{-1}+(s-i)^{-1}t=(s-i)^{-1}[s,t](s+i)^{-1}=0 \qquad\textnormal{on }~\Dom (t);\end{equation}
\begin{equation}\label{com2}s(t+i)^{-1}+(t-i)^{-1}s=(t-i)^{-1}[s,t](t+i)^{-1}=0 \qquad\textnormal{on }~\Dom (s).\end{equation}
From this it follows readily that $s$ commutes with $(1+t^{2})^{-1}$ on $\Dom (s)$ and $t$ commutes with $(1+s^{2})^{-1}$ on $\Dom (t)$. We therefore have equalities
\begin{equation}\label{rescom}(s\pm i)^{-1}(1+t^{2})^{-1}=(1+t^{2})^{-1}(s\pm i)^{-1},\qquad (t\pm i)^{-1}(1+s^{2})^{-1}=(1+s^{2})^{-1}(t\pm i)^{-1},
\end{equation}
of bounded operators on the Lipschitz module $\E\lrh\B$.

The proof of the following theorem is essentially contained in \cite{Mes09b} yet, since in our case of a commuting connection it simplifies greatly, we include it here for the sake of  completeness.

\begin{thm} The sum of anti-commuting operators $s$ and $t$ on a $C^{*}$-module $\mathpzc{E}$ is self-adjoint and regular on $\Dom (s)\cap\Dom (t)$ with core $\im(s+ i)^{-1}(t- i)^{-1}$.
\end{thm}

\begin{proof} The sum is closed and symmetric by a standard argument. It is self-adjoint and regular by the following argument. The operator $x=(s+i)^{-1}(t+i)^{-1}$ maps $\mathpzc{E}$ into $\Dom (s)\cap\Dom (t)$ and by \eqref{rescom} we have
\[xx^{*}=(s+i)^{-1}(1+t^{2})^{-1}(s-i)^{-1}=(1+s^{2})^{-1}(1+t^{2})^{-1}=(1+t^{2})^{-1}(1+s^{2})^{-1}.\]
The operators $(s+t)x$, $(s+t)x^{*}$ are bounded by \eqref{com1} and \eqref{com2}; moreover we find that $(s+t)xx^{*}=xx^{*}(s+t)$. Therefore the operator
\[g:=\begin{pmatrix}x & -(s+t)x \\ (s+t)x & x\end{pmatrix},\]
satisfies
\[gg^{*}=\begin{pmatrix} xx^{*}+(s+t)xx^{*}(s+t) & 0\\0 & xx^{*}+(s+t)xx^{*}(s+t)\end{pmatrix},\]
which is strictly positive since $xx^{*}$ is so. Thus $g$ has dense range in $\mathpzc{E}\oplus\mathpzc{E}$. It maps $\mathpzc{E}\oplus\mathpzc{E}$ into $\mathfrak{G}(D)\oplus v\,\mathfrak{G}(D)$, which must therefore be all of $\mathpzc{E}\oplus\mathpzc{E}$. It follows that $s+t$ is self-adjoint and regular.
\end{proof}

Let us fix some notation. Let $B$ be a $C^{*}$-algebra with a fixed Lipschitz subalgebra $\mathcal{B}$. 

\begin{defn}
We denote by $\Psi_{0}(\mathcal{B},C)$ the set of unitary equivalence classes of $(B,C)$ KK-cycles $(\mathpzc{F},T)$ with the property that the map
$$
\mathcal{B}\rightarrow\End^{*}_{C}(\mathpzc{F}),\qquad b\mapsto [T,b],
$$
is a cb-derivation. 
\end{defn}

Most importantly, there is a natural action of Lipschitz cycles upon such KK-cycles, which in turn induces the bounded Kasparov product in the following way.

\begin{thm} 
\label{thm:KK-prod}
Let $(\mathcal{E},S,\nabla)$ be a Lipschitz cycle for $(\mathcal{A},\mathcal{B})$ and let $(\mathpzc{F},T)$ be a $(\mathcal{B},C)$ KK-cycle. Then the pair 
$$
(\mathpzc{E}\,\hotimes_{B}\,\mathpzc{F}, S\otimes 1 +1\otimes_{\nabla}T)
$$ 
is an $(\mathcal{A},C)$ KK-cycle representing the Kasparov product of $(\mathpzc{E},S)$ and $(\mathpzc{F},T)$.
\end{thm}

\begin{proof} We shall prove here that the operators $s:=S\otimes 1$ and $t=1\otimes_{\nabla}T$ anti-commute, so that their sum $s+t$ is self-adjoint on the intersection of the domains, and we shall check that $\mathcal{A}$ preserves the domain of the sum and has bounded commutators. The sum has compact resolvent by the same considerations as \cite{KaLe, Mes09b}: all of this is enough to deduce that we do indeed have an $(\A,C)$ KK-cycle.  Just as in \cite{KaLe, Mes09b}, one can then check Kucerovsky's conditions \cite[Thm~13]{Kuc97} to verify that the Kasparov product is indeed represented in this way.

We first show that resolvents $(s \pm i)^{-1}$ preserve the domain of $t$. 
The submodule
\[X:= \left(1\otimes \textnormal{pr}_{1}\right) \left(\E \hotimes_{\B}\mathfrak{G}(T)\right)=\textnormal{pr}_{1}( g(\mathcal{E}\hotimes_{\mathcal{B}}\mathfrak{G}(T)))\]
is a core for $t$, since $g$ has dense range in the graph $\mathfrak{G}(1\otimes_{\nabla}T)$ by Theorem~\ref{indgraph}. The operators $(s\pm i)^{-1}$ map this core into the domain of $t$, because $(s\pm i)^{-1}=(S\pm i)^{-1}\otimes 1$ and $(S\pm i)^{-1}\in \End^{*}_{\B}(\mathcal{E})$ by assumption. Thus, on $X$ we can write
\[(s\pm i)^{-1} X=(s\pm i)^{-1}(1\otimes\textnormal{pr}_{1})(\mathcal{E}\hotimes_{\B}\mathfrak{G}(T))=(1\otimes \textnormal{pr}_{1})( (S\pm i)^{-1}\mathcal{E}\hotimes_{\B}\mathfrak{G}(T))\subset X,\]
and thus $(s\pm i)^{-1}$ preserve this core. 
Note that, since $t$ is an odd operator, the graded commutator $[t,a]$ is computed via
\[ [t,a]=ta-(-1)^{|t||a|}at=ta-(-1)^{|a|}at=ta-\gamma(a)t.\]
Also note that, since $s$ is odd, we have $\gamma(s\pm i)^{-1}=-(s\mp i)^{-1}$.
For each elementary tensor $e\otimes f \in X$ it holds that 
\[(s\pm i)^{-1}(e\otimes f)=(S\pm i)^{-1}e\otimes f \in X,\] and thus we can write 
\[\begin{split} [t& ,(s\pm i)^{-1}]e\otimes f \\& =\gamma ((S\pm i)^{-1} e)\otimes Tf +\nabla(\gamma(S\pm i)^{-1}e))f-\gamma((s\pm i)^{-1})(\gamma(e)\otimes Tf+\nabla(\gamma(e))f),\\
&= -((S\mp i)^{-1} \gamma(e))\otimes Tf -\nabla((S\mp i)^{-1}\gamma(e))f+((S\mp i)^{-1}\gamma(e))\otimes Tf+(s\mp i)^{-1}\nabla(\gamma(e))f)\\
&=(s\mp i)^{-1}\nabla(\gamma(e))f-\nabla((S\mp i)^{-1}\gamma(e))f\\
&=[\nabla,(S\mp i)^{-1}]\gamma(e)f\\
&=0. \end{split}\]
It follows that for $e\otimes f\in X$ we have
\begin{equation}\label{anticomm2} (s\pm i)^{-1} e\otimes f \in \Dom (t),\quad t(s\pm i)^{-1}e\otimes f=-(s\mp i)^{-1}t(e\otimes f).\end{equation}
Since $X$ is a core for $t$, for any $x\in \Dom (t)$ there is a sequence $x_{n}\in X$ converging to $x$ in the graph norm of $t$. Then by \eqref{anticomm2}
\[t(s\pm i)^{-1}x_{n}=-(s\mp i)^{-1}tx_{n} \rightarrow- (s\mp i)^{-1}tx,\]
and therefore $(s\pm i)^{-1}x\in \Dom (t)$. So the resolvents preserve the domain of $t$ and \begin{equation}\label{comm}t(s\pm i)^{-1}=(-s\pm i)^{-1}t,\quad t(1+s^{2})^{-1}=(1+s^{2})^{-1}t,\qquad\textnormal{on}~ \Dom (t)\end{equation}
(this is in fact a standard argument, see for example \cite[Prop.~2.1]{FMR13} for details).
From this we obtain the identities
\[\begin{split}((s+i)^{-1}(t+i)^{-1}+(t-i)^{-1}(s-i)^{-1})&=(t-i)^{-1}\left((s-i)^{-1}(t+i)+(t-i)(s+i)^{-1}\right)(t+i)^{-1}\\
&=(t-i)^{-1}\left(i(s-i)^{-1}-i(s+i)^{-1}\right)(t+i)^{-1}\\
&=2(t-i)^{-1}(1+s^{2})^{-1}(t+i)^{-1}\\
&=2(1+t^{2})^{-1}(1+s^{2})^{-1}\\
&=2(s+i)^{-1}(1+t^{2})^{-1}(s-i)^{-1},\end{split}\]
 which implies that
\[(s+i)^{-1}(t+i)^{-1}=(t-i)^{-1}(s-i)^{-1}( 2(s+i)^{-1}(t+i)^{-1}-1),\]
and
\[(t-i)^{-1}(s-i)^{-1}=(s+i)^{-1}(t+i)^{-1}( 2(t-i)^{-1}(s-i)^{-1}-1).\]
From this it is immediate that
\[\im \left((s+ i)^{-1}(t+i)^{-1}\right)\subset \im \left((t-i)^{-1}(s-i)^{-1}\right)\subset \im \left((s+i)^{-1}(t+i)^{-1}\right),
\] 
and that $[s,t]=0$ on this set. That is to say that $s$ and $t$ anti-commute and so the self-adjointness proof is complete.

To show that the algebra $\A$ preserves the domain of the sum operator, we first show that 
\[Y:=1\otimes\textnormal{pr}_{1}( (S+ i)^{-1}\mathcal{E}\hotimes_{\B}\mathfrak{G}(T))=(s+i)^{-1}X\subset \im \left((s+i)^{-1}(t+i)^{-1}\right)\]
is a core for $s+t$. This follows because $X$ is a core for $t$ so there is a dense submodule $Z\subset \mathpzc{E}\hotimes_{B}\mathpzc{F}$ such that $X=(t+i)^{-1}Z$. Fix an arbitrary element $w\in \mathpzc{E}\hotimes_{B} \mathpzc{F}$ and choose a sequence $z_{n}\in Z$ converging to $w$. Then
\[(s+t)(s+i)^{-1}(t+i)^{-1}z_{n}= (1-i(s+i)^{-1})(t+i)^{-1}z_{n} -(s-i)^{-1}(1-i(t+i)^{-1})z_{n},\]
which is convergent because $z_{n}$ is convergent. Therefore, for all $w\in \mathpzc{E}\hotimes_{B}\mathpzc{F}$ we have \[(s+i)^{-1}(t+i)^{-1}w\in \overline{Y}_{s+t},\] the closure of $Y$ in the graph norm of $s+t$. Therefore $\overline{Y}_{s+t}$ contains the core $\im (s+i)^{-1}(t+i)^{-1}$, and therefore $\overline{Y}_{s+t}$ is dense in the graph of $s+t$. Since it is also closed in the graph norm, this shows that $\overline{Y}_{s+t}=\Dom (s)\cap\Dom (t)$.

Now since $a\in \A$ preserves the domain of $s$, it maps $Y$ into the domain of $s$. Since both $(s+i)^{-1}$ and $a$ preserve the domain of $t$, $a$ maps $Y$ into the domain of $t$. So $a$ is a map 
\[a:Y\rightarrow \Dom (s)\cap\Dom (t)=\Dom (s+t).\] The commutator $[s,a]=[S,a]\otimes 1$ and therefore is bounded. Because $Y=(s+i)^{-1}X\subset X$, the commutator $[t,a]$ can be computed via
\[ [t,a]e\otimes f=[1\otimes_{\nabla}T,a]e\otimes f=[\nabla,a]e\otimes f,\]
which is bounded because by assumption both $\nabla$ and $a$ define completely bounded operators on $\E$. Thus, the elements of $a\in\mathcal{A}$ map the core $Y$ into the domain of $s+t$ and the commutators are bounded on $Y$. Therefore $a\in\mathcal{A}$ actually preserves the domain of $s+t$ and has bounded commutators there (once again see \cite[Prop.~2.1]{FMR13}).  This completes the proof.
\end{proof}

The situation of the previous theorem is captured in the following diagram:
\begin{diagram}\Psi_{0}^{\ell}(\mathcal{A},\mathcal{B}) & \times & \Psi_{0}(\mathcal{B},C) & \rTo & \Psi_{0}(\mathcal{A},C)\\
\dTo & & \dTo & & \dTo \\
\KK_{0}(A,B) & \times & \KK_{0}(B,C) & \rTo & \KK_{0}(A,C).
\end{diagram}

\begin{rem}\textup{
Of course, this method only gives a recipe for seeking the internal product of {\em even} unbounded KK-cycles. In this paper we will also be interested in taking the internal product of {\em odd} cycles. In the remainder of this section, we explain how to adapt the above construction in order to achieve this in the various possible cases. \\
We emphasise that all of the examples below are consequences of the theory developed in this section for products of even cycles. The formul{\ae} for the product operators below are very convenient expressions, however they do not give short-cuts for checking the above analysis (for example for checking Kucerovsky's conditions).
}
\end{rem}

An odd cycle for a pair of ungraded $C^{*}$-algebras consists of an $(A,B)$-bimodule $\mathpzc{E}$ and a self-adjoint regular  operator $D$ with compact resolvent such that $[D,a]$ extends to an operator in $\End^{*}_{B}(\mathpzc{E})$ for all $a$ in a dense subalgebra of $A$. The set of unitary equivalence classes of such cycles is denoted $\Psi_{-1}(A,B)$. Odd spectral triples constitute examples of elements in $\Psi_{-1}(A,\C)$. Odd Lipschitz cycles are similarly defined and denoted $\Psi_{-1}^{\ell}(\mathcal{A},\mathcal{B})$.

Given a pair of $C^*$-algebras $A,B$, it also makes sense to speak of unbounded $(A,B\otimes\C_i)$-cycles for each Clifford algebra $\C_i$, $i=0,1,2,\ldots$. We thus define
$$
\Psi_i(A,B):=\Psi_0(A,B\otimes\C_i),\qquad i=0,1,2,\ldots.
$$
By construction, these unbounded cycles enjoy the periodicity property 
$$
\Psi_{i-2}(A,B)\cong \Psi_{i}(A,B)
$$ 
for each $i=1,2,\ldots$. The same applies to Lipschitz cycles.

\begin{rem}\label{re:double}\textup{
In particular, this use of Clifford algebras means that every odd unbounded KK-cycle may be viewed as an even cycle by equipping $A,B$ with trivial gradings and then using the identifications
$$
\Psi_{-1}(\A,B)\cong \Psi_1(\A,B)\cong \Psi_0(\A,B\otimes\C_1).
$$
Indeed, each odd cycle $(\mathcal{E},D)$ in $\Psi_{-1}(\A,B)$ determines an element $(\widetilde{\mathcal{E}},\widetilde{D})$ of $\Psi_0(\A,B\otimes\C_1)$ with grading $\Gamma:\widetilde{\mathcal{E}}\to\widetilde{\mathcal{E}}$ by setting
\begin{equation}\label{doubled}
\widetilde{\mathcal{E}}:=\mathcal{E}\otimes\C^2, \qquad \widetilde{D}:=\begin{pmatrix}0&D\\D&0\end{pmatrix},\qquad \Gamma:=\begin{pmatrix}0&-i\\i&0\end{pmatrix}.
\end{equation}
The original cycle $(\mathcal{E},D)$ is recovered from $(\widetilde{\mathcal{E}},\widetilde{D})$ by viewing $\mathcal{E}\subset \widetilde{\mathcal{E}}$ as the diagonal submodule.
}
\end{rem}

The latter construction also applies to Lipschitz cycles $(\mathcal{E},D,\nabla)$ in $\Psi_{-1}^{\ell}(\mathcal{A},\mathcal{B})$ by doubling the connection as
\[\widetilde{\nabla}:=\begin{pmatrix}\nabla &0 \\ 0 & \nabla\end{pmatrix}.\]
These observations now make the various combinations of products of unbounded KK-cycles rather easy to describe. We sketch in turn how to form the product of even-with-odd, odd-with-even and odd-with-odd unbounded KK-cyles.

\begin{example}
Let $(\mathcal{E},S,\nabla)\in \Psi_0^{\ell}(\mathcal{A},\mathcal{B})$ and $(\mathpzc{F},T)\in\Psi_{-1}(\B,C)$. Then using the `doubling' construction described in Remark~\ref{re:double}, we may identify $(\mathpzc{F},T)$ with an element $(\widetilde{\mathpzc{F}},\widetilde T)$ of $\Psi_0(\B,C\otimes\C_1)$. The internal product
$$
\Psi_0^{\ell}(\mathcal{A},\mathcal{B})\times\Psi_0(\mathcal{B},C\otimes\C_1)\to\Psi_0(\mathcal{A},C\otimes\C_1)\cong\Psi_{-1}(\A,C)
$$
of the resulting elements is now well defined; 
the final step passing from even to odd cycles as described in Remark~\ref{re:double} yields the cycle $(\mathcal{E}\hotimes_\B\mathpzc{F},S\otimes 1+1\otimes_\n T)\in \Psi_{-1}(\mathcal{A},C)$.
\end{example}

\begin{example}
\label{ex:odd-ev}
Let $(\mathcal{E},S,\nabla)\in \Psi_{-1}^{\ell}(\mathcal{A},\mathcal{B})$ and $(\mathpzc{F},T)\in\Psi_{0}(\mathcal{B},C)$. Again using the doubling construction, we may identify $(\mathcal{E},S,\nabla)$ with an element $(\widetilde{\mathcal{E}},\widetilde S,\widetilde\nabla)$ of $\Psi_0^{\ell}(\mathcal{A},\mathcal{B}\otimes\C_1)$. Similarly, we may take the {\em external} product of $(\mathpzc{F},T)$ in $\Psi_{0}(\mathcal{B},C)$ with the trivial cycle $(\C^2,0)$ in $\Psi_0(\C_1,\C_1)$ to obtain the cycle
$$
(\mathpzc{F}\otimes\C^2,T\otimes1)\in \Psi_0(\mathcal{B}\otimes\C_1,C\otimes\C_1),
$$
graded by the tensor product of the grading $\Gamma_\F:\mathpzc{F}\to\mathpzc{F}$ with the grading on $\C_1$. The internal product
$$
\Psi_0^{\ell}(\mathcal{A},\mathcal{B}\otimes\C_1)\times\Psi_0(\mathcal{B}\otimes\C_1,C\otimes\C_1)\to\Psi_0(\mathcal{A},C\otimes\C_1)\cong\Psi_{-1}(\mathcal{A},C)
$$
is now well defined. 
The final step in passing from even to odd cycles is made using Remark~\ref{re:double}, resulting in the cycle $$(\mathcal{E}\hotimes_{\B}\mathpzc{F},S\otimes \Gamma_{\F}+1\otimes_\n T)$$ as an element of $\Psi_{-1}(\A,C)$. If we decompose the module $\mathpzc{F}$ into its graded components,
$$
\mathpzc{F}=\mathpzc{F}_+\oplus\mathpzc{F}_-,\qquad T=\begin{pmatrix} 0 & T_+ \\ T_- & 0 \end{pmatrix},\qquad \Gamma_{\mathpzc{F}}=\begin{pmatrix} 1 & 0 \\ 0 & -1 \end{pmatrix},
$$
the product operator has the explicit form
\begin{equation}\label{eq:ev-odd-op}
\begin{pmatrix} S\otimes 1 & 1\otimes_\n T_+ \\ 1\otimes_\n T_- & -S\otimes 1 \end{pmatrix}
\end{equation}
as an operator from $\left(\Dom(S\otimes 1)\cap \Dom(1\otimes_\n T)\right)$ to  $\mathcal{E}\hotimes_{\B}\mathpzc{F}$. 
\end{example}

\begin{example}\label{ex:KK-odd}
Finally we consider the product of odd KK-cycles $(\mathcal{E},S,\nabla)\in \Psi_{-1}^{\ell}(\mathcal{A},\mathcal{B})$ and $(\mathpzc{F},T)\in \Psi_{-1}(\mathcal{B},C)$. In this case we apply the doubling construction to each of these to obtain unbounded cycles $(\widetilde{\mathcal{E}},\widetilde{S},\widetilde{\nabla})\in \Psi_0^{\ell}(\mathcal{A},\mathcal{B}\otimes\C_1)$ and $(\widetilde{\mathpzc{F}},\widetilde{T})\in \Psi_0(\mathcal{B},C\otimes\C_1)$. 
Following this, we take the external product of $(\widetilde{\mathpzc{F}},\widetilde{T})\in \Psi_0(\mathcal{B},C\otimes\C_1)$ with $(\C_1,0)\in\Psi_0(\C_1,\C_1)$ to obtain the cycle $(\widetilde{\mathpzc{F}}\otimes\C_1,\widetilde{T}\otimes 1)\in \Psi_0(\mathcal{B}\otimes\C_1,C\otimes \textup{M}_2(\C))$. We now have a well defined internal product
$$
\Psi_0^{\ell}(\mathcal{A},\mathcal{B}\otimes\C_1)\times\Psi_0(\mathcal{B}\otimes\C_1,C\otimes\textup{M}_2(\C))\to\Psi_0(\mathcal{A},C\otimes\textup{M}_2(\C))\cong\Psi_{0}(A,C).
$$
One finds that the resulting even $(\mathcal{A},C)$ cycle is given by $$(\mathcal{E}\hotimes_{\B}\mathpzc{F}\otimes\C^2, S\otimes 1\otimes \gamma^1+1\otimes_\n T\otimes\gamma^2),
$$ 
where $\gamma^1$, $\gamma^2$ are (real) generators of the Clifford algebra $\C_1$. Upon making explicit choices of representatives for the gamma matrices, the product operator has the form
\begin{equation}\label{odd-odd-op}
\begin{pmatrix} 0 & S\otimes 1 -i 1\otimes_\n T \\ S\otimes 1 +i 1\otimes_\n T & 0 \end{pmatrix}
\end{equation}
as an operator from $\left(\Dom(S\otimes 1)\cap \Dom(1\otimes_\n T)\right)\otimes\C^2$ to  $\mathcal{E}\hotimes_{\B}\mathpzc{F}\otimes\C^2$ ({\em cf}. \cite{KaLe}).
\end{example}

\section{Gauge Theories from Noncommutative Manifolds and KK-Factorization}
\label{sect:gauge-KK}

In this section, we will show how spectral triples naturally give rise to (generalized) gauge theories. Starting from a given spectral triple, we first recall from \cite{C94,C96} how to associate to it a  gauge group together with a set of gauge potentials equipped with a natural action of the former. This can be done in two ways: one in terms of the unitary endomorphisms of a noncommutative vector bundle, which generalizes the classical notion of gauge theory ({\em cf.} \cite{C94} and references therein), the other using Connes' notion of `inner fluctuations' of the spectral triple arising via Morita self-equivalences \cite{C96}.

As already mentioned, however, these two approaches are in general mutually incompatible. The goal of this section is to put these notions of gauge theory into a commutative geometric context, starting with a factorization in unbounded KK-theory of a Hilbert bundle over a commutative base manifold. The above gauge groups and gauge potentials can then be naturally described in terms of the Hilbert bundle in a unified way.

\subsection{Gauge transformations and inner fluctuations}
The motivation for our proposal for gauge theories in noncommutative geometry comes from symmetries of spectral triples. The first candidate for our attention is unitary equivalence. Let $B_1$ and $B_2$ be unital $C^*$-algebras.

\begin{defn}
A pair of spectral triples $(B_1,\H_1,D_1)$ and $(B_2,\H_2,D_2)$ 
are said to be {\bf unitary equivalent} if $B_1 \simeq B_2$ and there exists a unitary 
operator $U: \H_1 \to \H_2$ such that
\begin{align*}
U D_1 U^* = D_2,\qquad U \pi_1(b) U^*  = \pi_2(b),
\end{align*}
for all $b\in \B_1$.
\end{defn}

In this situation we immediately find that $U$ implements an isomorphism between the corresponding Lipschitz algebras $\B_{1}$ and $\B_{2}$ which is an isometry for the norms induced by the Lipschitz representations \eqref{Liprep}, as the following result shows.

\begin{prop}\label{pr:ueq}
Let $(B_1,\H_1,D_1)$ and $(B_2,\H_2,D_2)$ be unitarily equivalent spectral triples. Then the corresponding Lipschitz algebras $\B_1$ and $\B_2$ are isometrically isomorphic.
\end{prop}

\proof We simply compute that
\[\begin{split}\pi_{D_2}(b)&=\begin{pmatrix} U\pi_{1}(b)U^{*} & 0 \\ [D_{2},U\pi_{1}(b)U^{*}] & U\pi_{1}(b)U^{*}\end{pmatrix}\\
&=\begin{pmatrix} U& 0\\ 0 & U\end{pmatrix}\begin{pmatrix} \pi_{1}(b) & 0 \\ [D_{1},\pi_{1}(b)] & \pi_{1}(a)\end{pmatrix}\begin{pmatrix} U^{*}& 0\\ 0 & U^{*}\end{pmatrix}\\
&=U\pi_{D_1}(b)U^{*},\end{split}\]
where the operators $\pi_{D_{i}}(b)$, $i=1,2$, are as in eq.~\eqref{Liprep}. \endproof

As a special case, we consider the situation where $B=B_1=B_2$ and the unitary equivalence is implemented by a unitary element of the Lipschitz algebra $\B$, leading to the following definition.

\begin{defn}\label{de:int-g}
The {\bf gauge group} of the spectral triple $(B,\H,D)$ is defined to be the group $\U(\B)$ of unitary elements of the Lipschitz algebra $\B$ equipped with the multiplication induced by the algebra structure of $\B$, and the topology it inherits as a subspace of $\mathcal{B}$.
\end{defn}

Each element $u\in \U(\B)$ of the internal gauge group induces a perturbation of the Dirac operator according to the transformation rule
\begin{equation}\label{pert}
D \mapsto D^u:=u D u^* = D + u [D,u^*].
\end{equation}
This in turn implements a unitary equivalence between the spectral triples $(B,\H,D)$ and  $(B,\H, D^u)$ ({\em cf}. \cite{C96}).

\begin{rem}
\textup{
Perturbing the Dirac operator according to the rule $D\mapsto D^u$ is equivalent to acting upon the algebra $\B$ by the automorphism 
$$
\alpha_u:\B\to\B,\qquad \alpha_u(b):=ubu^*.
$$
The set of automorphisms of $\B$ of this type form a group under the operation of composition.  The elements of this group,  which we denote by $\textup{Inn}(\B)$, are called {\bf inner automorphisms}, in contrast to the group $\textup{Out}(\B)$ of outer automorphisms, defined to be the quotient $\textup{Out}(\B):=\Aut(\B)/\textup{Inn}(\B)$. This is nicely summarized by the short exact sequence
$$
1 \to \textup{Inn}(\B) \to \Aut(\B) \to \textup{Out}(\B) \to 1.
$$ 
Note that if $\A = \Lip(M)\supset C^\infty(M)$ is the (commutative) algebra of Lipschitz functions on a classical smooth manifold $M$, there are no non-trivial inner automorphisms and so $\textup{Out}(\A)$ is the group of bi-Lipschitz homeomorphisms $M\rightarrow M$. In particular, there is an inclusion $\textup{Diff}(M)\subset \textup{Out}(\A)$.
} 
\end{rem}

In this way, we see that a non-Abelian gauge group appears naturally whenever $\A$ is a noncommutative algebra. Yet it turns out that we can do better than this. Noncommutative algebras allow for a more general and much more natural notion of equivalence than that afforded by inner automorphisms. Indeed, the most natural notion of an invertible morphism between  noncommutative $C^*$-algebras is given by {\em Morita equivalence}.
 Let us see if we can lift Morita equivalence to the level of spectral triples. 

Given a $C^*$-algebra $B$, any Morita equivalent $C^*$-algebra $A$ is necessarily isomorphic to the algebra of adjointable endomorphisms of some finitely generated (right) Hilbert module $\EE\leftrightharpoons B$, 
$$
A = \End^{*}_{B}(\EE).
$$
In this situation, let $(B,\H,D)$ be a spectral triple over $B$ and let $\E\leftrightharpoons\B$ be a right Lipschitz module over the Lipschitz algebra $\B$ whose $C^*$-envelope is isomorphic to $\EE$.

As already mentioned, the Lipschitz module $\E$ always admits a connection
$$
\nabla: \E \to \E \hotimes_{\B} \Omega^1_D(\B).
$$
Let us choose one. Then writing $\H_\E:=\E\hotimes_\B \H$, we construct the operator 
$$
D_\n:\mathfrak{Dom}(D_\n)\to \H_\E,\qquad D_\n:=1\otimes_\n D,
$$
and define 
$$
\A:=\{a\in A~|~[D_\n,a]\in \End^*_{B}(\EE)\}\cong \End^{*}_{\B}(\E).
$$
The last isomorphism follows from \cite[Thm~5.5.1]{Mes09b}. It follows immediately from the definition that the datum $(\E, 0 , \nabla)$ determines an element of the set of Lipschitz cycles $\Psi^\ell_0(\A,\B)$. Upon choosing such a connection, we find the following result. The construction first appeared in \cite{C96}, here we recast it in terms of our KK-theoretic language.

\begin{prop}\label{pr:gauge}
The Kasparov product of the Lipschitz cycle $(\E,0,\nabla) \in \Psi^\ell_0(\A,\B)$ with the spectral triple $(B,\H,D) \in \Psi_0(\B,\C)$, given by the formula
\begin{equation}
\label{eq:st-morita}
(\H_\E,  D_\n)=(\E\otimes_\A \H,1\otimes_\n D) \in \Psi_0(\A,\C),
\end{equation}
yields a spectral triple over the $C^*$-algebra $A$, with Lipschitz algebra cb-isomorphic to $\A$.
\end{prop}

\proof This is an immediate consequence of Theorem~\ref{thm:KK-prod}.  \endproof

Let us now focus upon \emph{Morita self-equivalences}, i.e. the situation in which $A =\EE = B$. Let $(A,\H,D)$ be a spectral triple over $A$ as above. In this setting we look at Hermitian connections 
$$
\nabla: \A \to \Omega^1_D(\A).
$$
By the Leibniz rule we automatically have $\nabla = d +\omega$, where $\omega \equiv \nabla(1) = \sum_j a_j[D,b_j]$ is a generic self-adjoint one-form in $\Omega^1_D(\A)$. Under the identification $\E\otimes_\A \H \simeq \H$ we have
$$
1 \otimes_\nabla D \equiv  D + \omega.
$$
In other words, upon choosing a connection $\n$ on $\E$, the Dirac operator $D$ is `internally perturbed' to $D_\omega := D+\omega$. The one-form $\omega^* =\omega \in \Omega^1_D(\A)$ will be interpreted as {\bf gauge field}. 

\begin{rem}
\textup{
The passage from the spectral triple $(A,\H,D)$ to the spectral triple $(A,\H,D_\omega)$ is called an {\bf inner fluctuation} of the Dirac operator $D$, since it is the algebra $\A$ that generates the gauge fields $\omega$ through Morita self-equivalences. In terms of the unbounded Kasparov product described in Proposition~\ref{pr:gauge}, the internal gauge fields $\omega$ are generated by taking the internal product of the spectral triple $(A,\H,D) \in \Psi_0(\A,\C)$ with the Lipschitz cycle $(\A,0,\nabla) \in \Psi^\ell_0(\A,\A)$. 
}
\end{rem}

Combining the two main ideas presented in this section is now easy.  The fluctuated spectral triple $(\A, \H, D_\omega)$ also carries an action of the internal gauge group group $\U(\A)$ by unitary equivalences as in eq.~\eqref{pert}, that is to say
$$
D_\omega \mapsto u D_\omega u^*, \qquad u\in \U(\A),
$$
or equivalently
\begin{equation}
\label{eq:gauge-action}
\omega \mapsto u \omega u^* + u [D,u^*], \qquad u\in\U(\A),~\omega\in\Omega^1_D(\A),
\end{equation}
which is the usual rule for the transformation of a gauge field. 

\subsection{KK-factorization and a proposal for gauge theories}
In this way, we have two different possibilities for perturbing a given spectral triple: the first via unitary equivalences and the second via Morita self-equivalences. However, in place of this noncommutative-geometric interpretation of these constructions, we would like to make contact with the classical world, by finding a unifying description of these two possibilities in terms of classical geometric objects, similar to the usual formulation of gauge theory on a vector bundle over a classical manifold. 

To this end, given a spectral triple $(A,\H,D)$, let us assume that we can {\it factorize} it in unbounded KK-theory over a {\it classical spin manifold}. That is to say, we assume that there exist a {\it commutative} $C^*$-algebra $B$ equipped with a spectral triple $(B,\H_0,D_0)$, together with a Lipschitz cycle $(\E,T,\nabla)$ for $(\A,\B)$ and such that $(A,\H,D)$ factors as an internal Kasparov product:
\begin{equation}
\label{eq:KK-fact}
(A,\H,D) \simeq (
\E\, \tilde\otimes_{\B}\, \H_0, T \otimes 1 + 1 \otimes_\nabla D_0 ) \in \Psi_0(\A,\C),
\end{equation}
{\em cf}. Theorem \ref{thm:KK-prod}. Since $B \simeq C(X)$ for some compact Hausdorff space $X$, the right $B$-module $\EE:= \E \hotimes_{\B} B$ consists of continuous sections of some Hilbert bundle $V\to X$ \cite{Tak79}. Our proposal is to consider this Hilbert bundle as the natural geometrical object on which to define a gauge theory, as we will now describe.

\begin{defn}
\label{defn:KK-gauge}
In the above notation, we define the {\bf Lipschitz gauge group} associated to the factorization \eqref{eq:KK-fact} to be
$$
\G(\E) :=  \left \{   U \in \End^{*}_{\B}(\E) : UU^* = 1_{\E} = U^*U,~ U \A U^* = \A, ~[T,U] \in \End^{*}_{B}(\mathpzc{E}) \right\}.
$$
The continuous gauge group is given similarly by $\mathpzc{G}(\EE)$, where we allow for $U\in \End^{*}_{B}(\mathpzc{E})$ and drop the bounded commutator condition in the definition above. The group $\mathpzc{G}(\mathpzc{E})$ is the $C^{*}$-closure of $\G(\mathcal{E})$.
\end{defn}

For practical reasons which will become apparent,  we would like this group to be realized as the group of unitary elements of some $C^*$-algebra $\tilde A$ which contains $A$, with $\tilde A$ the smallest possible $C^*$-algebra having this property. Inspired by the definition of reduced group $C^*$-algebras, we make the following definitions.

\begin{defn}
We define $\tilde \A$ to be the closure of the linear span of $\mathcal{G}(\mathcal{E})$ in the operator space topology given by the representation
\begin{equation}\label{eq: gaugerep} U\mapsto \begin{pmatrix} U & 0 \\ [T,U] & U\end{pmatrix}\in\End^{*}_{\mathcal{B}\oplus B}(\mathcal{E}\oplus\mathpzc{E}),\end{equation}
where $\mathcal{B}\oplus B$ denotes the matrix diagonal direct sum of involutive operator algebras. We define $\tilde A$ to be the $C^*$-closure of the algebra $\tilde A$.
\end{defn}

\begin{prop}
The $C^*$-algebra $\tilde A$ is the minimal $C^*$-algebra (ordered by inclusion) with the property that $\U(\tilde A)$ contains $\GG(\EE)$. It is isomorphic to the closure in $\End_{B}^{*}(\EE)$ of the complex linear span of $\GG(\EE)$:
\begin{equation}
\label{eq:tildeA}
\tilde A \simeq \overline{\text{\textup{Span}}}_\C \GG(\EE). 
\end{equation}
The $C^*$-algebra $\tilde A$ contains $A$ as a $C^*$-subalgebra. 
\end{prop}

\proof
Let $\tilde B$ be the minimal $C^*$-algebra such that $\GG(\EE)\subset\U(\tilde B)$ and let $B$ be an arbitrary $C^*$-algebra with the property that $\GG(\EE) \subset \U(B)$. Then $\GG(\EE) \subset \U(B) \subset B$, which by continuity implies that $\tilde B \into B$. Since clearly $\GG(\EE) \subset \U( \overline{\text{Span}}_\C \GG(\EE) )$, it follows that $\tilde B$ is the minimal $C^*$-algebra with this property.  It follows immediately that $\tilde B$ is isomorphic to the $C^*$-closure $\tilde A$ of the algebra $\tilde\A$.
\endproof

Alternatively, $\tilde{\mathcal{A}}$ can be described as those elements $a \in \tilde A$ for which $a\in\End^{*}_{\mathcal{B}}(\mathcal{E})$ and $[T,a] \in\End^{*}_{B}(\mathpzc{E})$. This in turn contains $\A$ and by construction we have $(\E,T,\nabla) \in \Psi^{\ell}_{0}(\tilde \A,\B)$, so that the following definition makes sense.

\begin{defn}
\label{defn:KK-fields}
With respect to the factorization \eqref{eq:KK-fact}, we define the space of {\bf scalar fields} $\cC_s$ to be 
$$
\cC_s:=\Omega^1_T(\tilde \A) = \{ \sum_j a_j [T,b_j]: a_j, b_j \in \tilde \A\},
$$ 
a subset of $\End_{B}^{*}(\mathpzc{E})$, and the space of {\bf gauge fields} $\cC_g$ as the affine space of connections $\nabla: \E \to \E \hotimes_{\B} \Omega^1_{D_0}(\B)$, i.e.
$$
\cC_g := \Hom_{\B} (\E, \E \hotimes_{\B} \Omega^1_{D_0}(\B)).
$$
\end{defn}

Note that such a use of unitary endomorphisms to generate the gauge fields  (via the algebra $\tilde \A$) has already been exploited in the context of physical applications of noncommutative geometry (e.g. \cite{LS01}), albeit in a topologically trivial context.

\begin{lem}
The following defines an action of $\G(\E)$ on $\cC(\E)=\cC_s \oplus \cC_g$:
\begin{align*}
\G(\E) \times \cC(\E) &\to \cC(\E),\qquad
(U,\Phi,\nabla)\mapsto ( \Phi^U,\nabla^U):= (U \Phi U^*, U \nabla U^*).
\end{align*}
\end{lem}

\proof
Let $\Phi = a[T,c]$ be a scalar field with $a,c \in \tilde \A$. Then
$$
U a [T,c] U^* = (U a) [T,c U^*] - U a c [T,U^*].
$$
Since $U a, c U^*$, $U a c$, and $U^*$ are elements in $\tilde \A$, the element $\Phi^U$ is in $\Omega^1_T(\tilde \A)$. 

Also, if $\nabla: \E \to \E \hotimes_{\B} \Omega^1_{D_0}(B,\B)$ is a connection on $\E$, then we check that
$$
U \nabla U^* (e b) = U \nabla U^* (e) b + e \otimes [D_0,b]
$$
for all $e \in \E, b \in \B$. In other words, $\nabla^U$ satisfies the Leibniz rule and is a connection on $\E$. 
\endproof

\begin{lem}
\label{lem:inner-decomp}
In the above notation we have:
\begin{enumerate}[(i)]
\item the internal gauge group $\U(\A)$ is a normal subgroup of $\G(\E)$;
\item the space $\Omega^1_D(\A)$ of internal gauge fields is a subspace of $\cC_s\oplus \cC_g$.
\end{enumerate}
\end{lem}

\proof
Indeed, for all $u \in \U(\A)$ and $U \in \G(\E)$ we have
$U u U^* \in \A$, which is at the same time a unitary element. By Definition~\ref{de:lipcyc}, the element $u\in\A$ has the property that $[T,u]$ is a bounded operator in $\End^{*}_{B}(\mathpzc{E})$. 
%
For the second claim, write $D = T \otimes 1 + 1 \otimes_{\nabla} D_0$. Then for each element in $\Omega^1_T(\tilde \A)$ we have
$$
a [D,c] = a [T,c] + a[\nabla,c]; \qquad (a,c \in \tilde \A).
$$
The first term is an element in $\Omega^1_T(\tilde \A)$ whereas for the second we show that it is a $\B$-linear map from $\E \to \E \hotimes_{\B} \Omega^1_{D_0}(B,\B)$:
\begin{align*}
a[\nabla,c](e b) &= a \nabla (c(e) b) - ac \nabla(e b) \\
&= (a \nabla (c(e))) b + a c e \otimes [D_0,b] - (ac \nabla(e )) b - ac e \otimes [D_0,b] \\
&= (a [\nabla, c](e)) b 
\end{align*}
for all $e \in \E$ and $b \in \B$. This concludes the proof.
\endproof

We deduce that our proposal for a gauge theory associated to a factorization of the form \eqref{eq:KK-fact} encompasses the previous {\it internal} gauge theory determined by the original spectral triple $(\A,\H,D)$. However, it does more than this: it allows us to distinguish within $\Omega^1_D(\A)$ between scalar fields acting fibrewise upon the Hilbert bundle $V$ and gauge fields as connections thereon. Similarly, the gauge group $\GG(\E)$ (containing the unitary group $\U(\A)$) acts fibrewise upon the space of Lipschitz sections $\Gamma^{\ell}(X,V)$. This action extends to an action of the $C^{*}$-gauge group $\mathcal{G}(\EE)$ on the space of continuous sections $\Gamma(X,V)$.

\section{Yang--Mills Theory} 
\label{sect:YM}

Next we turn to studying the implications and consequences of our new setting for gauge theory in unbounded KK-theory. To set the scene for the more general noncommutative case, in this section we recall how to use unbounded KK-theory to describe Yang--Mills gauge theories over classical manifolds \cite{CC,BoeS10}. We connect to the usual theory of principal bundles and connections thereon. 

\subsection{The Yang--Mills spectral triple} 
Let $M$ be a closed Riemannian spin manifold with dimension $m$ and spinor bundle $\cS$ and let $(C(M),L^2(M,\cS),\dirac_M)$ be the canonical spectral triple of Definition~\ref{de:can-st}. 

Let $E$ be a Hermitian vector bundle over $M$. We write $A:=\Gamma(M,\Xi)$ for the unital $C^*$-algebra consisting of continuous sections of the endomorphism bundle $\Xi:=\textup{End}(E)$.  From the Serre-Swan theorem for $*$-algebra bundles \cite{BoeS10} and holomorphic stability of the inclusion $\Lip(M)\subset C(M)$ it follows that $\Xi$ admits a Lipschitz structure and we write $\A=\Gamma^\ell(M,\Xi)$ for the involutive operator algebra of Lipschitz sections of $\textup{End}(E)$, which sits densely inside $A$. 
The Lipschitz subalgebra $\A$ acts as bounded endomorphisms on the $\Lip(M)$-module 
$$
\Gamma^\ell(M,E\otimes\cS) \simeq \Gamma^\ell(M,E) \hotimes_{\Lip(M)} \Gamma^\ell(M,\cS).
$$
Combining the Hermitian structure on $E$ with the usual $L^2$-inner product on the spinor bundle $\cS$ induces a natural inner product on the module $\Gamma^{\ell}(M,E\otimes\cS)$, giving a Hilbert space $\H_E:=L^2(M,E\otimes\cS)$ of square-integrable sections of the vector bundle $E\otimes\cS$. 

This construction gives us the first two ingredients of a spectral triple: a pre-$C^*$-algebra and a Hilbert space. In order to define a Dirac-type operator on $L^2(M,E \otimes \cS)$ we twist $\dirac_M$ by a Hermitian connection on $E$,
$$\n:\Gamma^\ell(M,E)\to\Gamma^\ell(M,E)\hotimes_{\Lip(M)}\Omega^1_D(M),$$
writing $\n_\Xi$ for its lift to the endomorphism bundle $\Xi$. The associated Dirac operator with coefficients in $E$ is the unbounded operator $D_E$ on $\h_E$ defined by
\begin{equation}\label{tw-dir}
D_E:\Dom(D_E)\to \H_E,\qquad D_E:=c\circ(1\otimes\n_\cS+\n\otimes 1),
\end{equation}
where $\n_\cS$ denotes the canonical spin connection on $\cS$ and $c$ denotes ordinary Clifford action of differential forms upon spinors. The following is a well-known result, essentially already contained in \cite{C94} and \cite{CC} ({\it cf.} \cite[Thm~3.10]{BoeS10}).

\begin{prop}\label{pr:bdl}
The datum $(A,\H_E,D_E)$ constitutes an $m^+$-summable spectral triple over the $C^*$-algebra $A$.
\end{prop}

\proof
The action of the pre-$C^*$-algebra $\A$ on $\Gamma^\ell(M,E \otimes \cS)$ extends to an action by bounded operators on $\H_E$. Since $D_E$ is a first order differential operator, the commutators $[D_E,a]$ are bounded for all $a\in \A$. Since $M$ is compact, the twisted Dirac operator $D_E$ is elliptic, from which the compact resolvent and summability conditions follow.\endproof

We shall refer to the spectral triple $(A,\H_E,D_E)$ as the {\bf Yang--Mills spectral triple} over $M$ determined by the vector bundle $E$ and the connection $\n$. To see why this terminology is appropriate, it is useful to recall the following alternative description of the geometry appearing in this section in terms of the parallel theory of principal bundles. 

As already mentioned, from the $*$-algebra bundle $\Xi=\textrm{End}(E)$ we obtain the $*$-algebra $\A=\Gamma^\ell(M,\Xi)$, which is finitely generated and projective as a $\Lip(M)$-module. Conversely, using the Serre-Swan theorem for $*$-algebra bundles \cite{BoeS10}, from $A$ we can reconstruct the original bundle $\Xi$ (up to isomorphism) as a locally trivial $*$-algebra bundle over $M$. For simplicity, we assume that $\Xi$ has typical fibre $\textup{M}_N(\C)$. 

\begin{lem}
There exists a principal bundle $P$ over $M$ with structure group $\textup{PSU}(N)$ such that 
\begin{equation}\label{assoc}
\Xi\simeq P\times_{\textup{ad}}\mathrm{M}_N(\C) 
\end{equation}
is the vector bundle associated to the adjoint representation $\textup{ad}:\textup{PSU}(N)\to \mathrm{M}_N(\C)$. Moreover, under this identification, Hermitian connections $\n_\Xi$ on $\Xi$ correspond bijectively to $\mathfrak{su}(N)$-valued connection one-forms $\omega$ on the principal bundle $P$.
\end{lem}

\proof 
Since all $*$-automorphisms of $\mathrm{M}_N(\C)$ are inner, i.e. they are obtained by conjugation by a unitary matrix $u\in\mathrm{M}_N(\C)$, the transition functions of the vector bundle $\Xi$ take values in the adjoint representation $\mathrm{Ad}\,\mathrm{U}(N)=\textup{U}(N)/\textup{Z}(\textup{U(N)})\cong \mathrm{PSU}(N)$. From these transition functions we construct a principal bundle $P$ over $M$ with structure group $\mathrm{PSU}(N)$ to which $\Xi=\textrm{End}(E)$ is the associated vector bundle as stated. It is not difficult to see that every such pair $(P,\omega)$ arises in this way from the datum of a Yang--Mills spectral triple $(\A,\H_E,D_E)$. We refer to \cite{BoeS10} for full details.
\endproof

Now we are ready to recast the spectral triple description of Yang--Mills theory in terms of the unbounded Kasparov product, beginning with the following easy result.

\begin{lem}
The datum $(\Gamma^\ell(M,E),0,\nabla)$ defines an element of the set $\Psi_0^\ell(\A,\Lip(M))$ of classes of even Lipschitz cycles.
\end{lem}

\proof The right Hilbert $C(M)$-module $\Gamma(M,E)$ is the completion of the vector space $\Gamma^\ell(M,E)$ in the topology defined by the Hermitian structure on $E$. By construction, it carries a bounded action of the $C^*$-algebra $A$ by left multiplication. The $C^*$-algebras $A$ and $C(M)$ and the Hilbert module $\Gamma(M,E)$ are all trivially graded, whence the zero operator is indeed odd. The remaining conditions are automatically satisfied. \endproof

For simplicity, let us assume that the manifold $M$ is even-dimensional, so that the canonical spectral triple $(C(M),\H,\dirac_M)$ is even and defines an unbounded KK-cycle in $\Psi_0(\Lip(M),\C)$. It is not difficult to see how to extend the following results to the case where $M$ is odd-dimensional.

\begin{prop}
As an element of $\Psi_0(\A,\C)$, the KK-cycle $(\H_E,D_E)$ is a Kasparov product of unbounded KK-cycles, namely
$$
(\H_E,D_E)\simeq (\Gamma^{\ell}(M,E),0,\nabla)\otimes_{\Lip(M)}(\H,\dirac_M),
$$
where $(\Gamma^\ell(M,E),0,\nabla)\in\Psi_0^\ell(\A,\Lip(M))$ and $(\H,\dirac_M)\in\Psi_0(\Lip(M),\C)$.
\end{prop}

\proof 
This result is a slight strengthening of \cite[Thm~3.21]{BoeS10}. We have an isomorphism of Hilbert spaces
$$
\H_E=L^2(M,E\otimes\cS)\simeq \Gamma(M,E)\hotimes_{C(M)}L^2(M,\cS).
$$
Under this identification, we clearly have that
$$
D_E\simeq 1\otimes_\n \dirac_M.
$$
The result now follows from Theorem \ref{thm:KK-prod}.
\endproof

With the aid of this result, we are now able to describe the Yang--Mills gauge theory of the previous section in terms of unbounded KK-cycles. The gauge group of the above KK-factorization is given by
$$
\G(\Gamma^\ell(M,E)) =  \U (\End_{\Lip(M)}^{*} (\Gamma^\ell(M,E))) \simeq \U(\A).
$$
In other words, in this case there is no difference between the internal gauge group $\U(\A)$ and the gauge group as defined in Definition \ref{defn:KK-gauge}. Recall that we can identify $\A$ with the space of Lipschitz sections of the adjoint algebra bundle $P \times_{\ad} \M_N(\C)$ associated to a $\textup{PSU}(N)$-principal bundle $P$. Consequently, 
$$
\G(\Gamma^\ell(M,E)) \simeq \Gamma^\ell(M,\textup{Ad}\,P),\qquad \text{where}\quad \textup{Ad}\,P:=P\times_\textup{Ad} \U(N).
$$
Let us now determine the gauge fields determined by the above KK-factorization, following our prescription of Definition~\ref{defn:KK-fields}. 

\begin{prop}
The gauge fields of the above KK-factorization are given by the affine space of connections $\nabla$ on $\Gamma^\ell(M,E)$; the scalar fields of the above KK-factorization all vanish. 
\end{prop}

\proof
First, since any $C^*$-algebra is generated by its unitary group, we have that $\tilde A = A$ and similarly that $\tilde \A = \A$. Moreover, since $T = 0$ we find that $\cC_s = 0$, leaving only $\cC_g$, the affine space of connections on $E$. 
\endproof

This justifies our use of the term \emph{gauge field}: the connections correspond bijectively to connections on the principal bundle $P$. Since $\Omega^1_D(\A) = \cC(\E)$, they also coincide with the internal gauge fields of Section~\ref{sect:gauge-KK}.

\subsection{Almost-commutative manifolds} 

An immediate generalization of the Yang--Mills spectral triple is given by the class of so-called almost-commutative spin manifolds. Roughly speaking, these are spin geometries described by a spectral triple whose function algebra is not commutative but rather consists of continuous sections of a finite-rank algebra bundle over some classical spin manifold. A very special example of such a manifold was constructed in \cite{CCM07} and subsequently applied to the Standard Model of particle physics. In this section we examine the structure of this class of (topologically trivial) almost-commutative manifolds from the point of view of unbounded KK-theory and their associated gauge theory.

As in the previous section, let $M$ be a closed Riemannian spin manifold with dimension $m$ and spinor bundle $\mathcal{S}$. Let $(C(M),L^2(M,\mathcal{S}),\dirac_M)$ be the canonical spectral triple over $M$. Moreover, let $(A_F,\H_F,D_F)$ be a {\em finite} spectral triple, that is to say a spectral triple in the sense of Definition~\ref{de:st} for which $\H_F$ is a finite-dimensional Hilbert space and so $A_F\subseteq \BB(\H_F)$ is a sum of matrix algebras.

We shall further assume for simplicity that $M$ is even-dimensional with grading operator $\gamma_M:L^2(M,\mathcal{S})\to L^2(M,\mathcal{S})$.  Let us fix a linear transformation $T:\H_F\to\H_F$ and define
$$
A:= A_F\otimes C(M), \qquad \H:= \H_F\otimes L^2(M,\mathcal{S}),\qquad D_T:=T\otimes \gamma_M
+ 1\otimes \dirac_M.
$$
Then we have a trivial $*$-algebra bundle $\Xi:=M\times A_F$ over $M$ with typical fibre $A_F$ such that $A\simeq \Gamma(M,\Xi)$. We write $\A=\Gamma^\ell(M,\Xi)$ for the Lipschitz algebra of the resulting spectral triple $(A,\H,D_T)$, that is to say the $*$-subalgebra of $A$ consisting of elements $a\in A$ such that $[D_T,a]$ extends to a bounded operator on $\H$. Since the algebra $A_{F}$ is finite and the operator $T$ is bounded, this algebra is just $\A=A_{F}\otimes \Lip(M)$.

\begin{lem}
The datum $(\Gamma^\ell(M, M\times \H_F),T,\d)$ determines a Lipschitz cycle in $\Psi^\ell_{-1}(\A,\Lip(M))$.
\end{lem}

\proof In writing $(\Gamma^\ell(M, M\times \H_F),T,\d)$, we are equipping the bundle $M\times \H_F$ with the trivial connection. It is obvious by definition that the space $\Gamma^\ell(M, M\times \H_F)$ is a Lipschitz $\A$-$\Lip(M)$-bimodule, equipped with the trivial gradings. Since $T:\H_F\to\H_F$ is a linear transformation of a finite-dimensional vector space, the conditions of Definition~\ref{de:lipcyc} are vacuously satisfied.\endproof

\begin{thm}
The spectral triple $(A,\H,D_T)$ defines an element of $\Psi_{-1}(\A,\C)$ which factors as an unbounded Kasparov product of the KK-cycles 
$$
(\Gamma^\ell(M, M\times \H_F),T,0)\in\Psi^\ell_{-1}(\A,\Lip(M)),\qquad  (L^2(M,\cS), \dirac_M)\in\Psi_0(\Lip(M),\C).
$$
\end{thm}

\proof This is an immediate consequence of Theorem~\ref{thm:KK-prod} by using Example~\ref{ex:odd-ev} to take the product of an odd cycle with an even cycle. We find rather easily that
$$
\H=\H_F\otimes L^2(M,\mathcal{S})\simeq\Gamma(M,M\times\H_F)\otimes_{C(M)}L^2(M,\cS)
$$
and that, since the connection on the bundle $M\times\H_F$ is trivial, the product operator $1\otimes_\d \dirac_M$ coincides with $1\otimes \dirac_M$ under this identification.\endproof

In order to see how our proposal for gauge theories indeed captures physical models, let us compute the gauge group of the above situation in a special case. The gauge theory of more general (topologically non-trivial) examples can be found in \cite{Cac11,BD13}.

\begin{example}[Glashow--Weinberg--Salam electroweak model]
We consider the special case where $A_F := \C \oplus \M_2(\C)$, acting upon the Hilbert space $H_F = \C \oplus \C^2$. Together with the matrix
$$
T = \begin{pmatrix} 0 & z_1 & z_2 \\ 
\bar z_1 & 0 & 0 \\
\bar z_2 & 0 & 0\end{pmatrix},
$$
this gives rise to one of the spectral triples studied in \cite{CCM07} ({\em cf}.  \cite{DS11}),
$$
(A,\H,D_T) = (A_F \otimes C(M),  H_F \otimes L^2(M,\cS), T \otimes \gamma_M+ 1 \otimes D).
$$
In fact, it is an external product of spectral triples \cite[Section VI.3]{C94} ({\em cf}. \cite{Lnd97}). It is of the above almost-commutative type, being an internal Kasparov product of the KK-cycle
$$
(\Gamma(M,M \times H_F), T,\d ) \in \Psi_{-1}^\ell(\A,\Lip(M))
$$
with the canonical spectral triple $(L^2(M,\cS), D) \in \Psi_0(\Lip(M),\C)$.

Let us then compute the gauge group and the scalar and gauge fields for this KK-factorization. 
The algebra $\tilde \A$ is the algebra of endomorphisms of the bundle $M \times H_F$ which multiply the action of $\A$, i.e.
$$
\tilde \A = \Gamma^\ell(M,\C \oplus \M_2(\C)) = \A.
$$
This implies that the gauge group $\U(\tilde \A) = \U(\A)\simeq \Lip(M,\U(1)\times\U(2))$ coincides with the internal gauge group of the algebra $\A$. ({\em cf}. Section~\ref{sect:gauge-KK}). The gauge fields are given by connections $\nabla$ on the $\U(1) \times \U(2)$-bundle $M \times (\C \oplus \C^2)$, which take the form
$$
\nabla = d+ \omega^{(1)} + \omega^{(2)}
$$
for connection one-forms $\omega^{(1)} \in \Omega^1(M) \otimes \mathfrak{u}(1)$ and $\omega^{(2)} \in \Omega^1(M) \otimes \mathfrak{u}(2)$ taking values in the Lie algebras $\mathfrak{u}(1)$ and $\mathfrak{u}(2)$ respectively. They transform under gauge transformations $(u^{(1)},u^{(2)}) \in \Lip(M,\U(1) \times \U(2))$ according to the familiar rules
\begin{align*}
\omega^{(1)} &\mapsto \omega^{(1)} + u^{(1)} \d u^{(1)*},\\
\omega^{(2)} &\mapsto u^{(2)}\omega^{(2)}u^{(2)*} + u^{(2)} \d u^{(2)*}.
\end{align*}
In physics these two `gauge potentials' respectively describe the $B$ and $W$ bosons (ignoring for the moment the still required reduction in structure group from $\U(2)$ to $\SU(2)$; this can be handled by replacing $\M_2(\C)$ by the quaternions as in \cite{CCM07}).

Similarly, the scalar fields are given by elements in $\Omega^1_T(\tilde \A)$. We compute for $\lambda, \lambda' \in \C$ and $m,m' \in \M_2(\C)$ that
$$
\begin{pmatrix} \lambda(x) & 0 \\ 0 & m (x) \end{pmatrix} \left[ T, 
\begin{pmatrix} \lambda'(x) & 0 \\ 0 & m'(x) \end{pmatrix} \right]
=: \begin{pmatrix} 0 & \phi_1(x) & \phi_2(x) \\ \overline{\phi_1(x)} & 0 & 0 \\ \overline{\phi_2(x)} & 0 &  0 \end{pmatrix}.
$$
The components of the pair $(\phi_1(x),\phi_2(x))$ of complex fields respectively transform in the defining representations of $\U(1)$ and $\U(2)$. In physics, they describe the Higgs boson.
\end{example}

\section{The Noncommutative Torus}
\label{sect:nc-torus}
This section is dedicated to obtaining an understanding of the spin geometry of the noncommutative torus $\TT^2_\theta$ in our context. Following \cite{C80} (see also \cite{C94}), we illustrate its geometry in terms of a canonical spectral triple and then demonstrate how to factorize this geometry as a noncommutative principal bundle with a classical base space. Although we concentrate on the special case of the noncommutative two-torus, it is not difficult to imagine how one might extend the construction to noncommutative tori of higher rank.

\subsection{Isospectral deformations} One of the best known ways of obtaining interesting examples of noncommutative spin manifolds is to `isospectrally deform' a classical spin manifold along the isometric action of a two-torus $\TT^2$ \cite{CL01}. This extends the $C^*$-algebraic deformation of \cite{Rie93a} to the smooth (spin manifold) case. The noncommutative manifolds (often called {\em Connes-Landi deformations}) that we shall consider in the remainder of the present paper will be of this form, so here we briefly recall the details of the construction,  mainly in order to establish notation.

\begin{defn}\label{de:eqst}
A spectral triple $(A,\H,D)$ over a $C^*$-algebra $A$ is said to be \textbf{torus-equivariant} if there exists 
a unitary group action $U:\TT^2\to \BB(\H)$ and an isometric action $\alpha:\TT^2\to \textup{Aut}(A)$ such that
$$
U(t)D=DU(t), \qquad U(t) \pi(a) U(t)^*=\pi(\alpha_{t}(a)),
$$
for all $t\in \TT^2$ and all $a\in A$.
\end{defn}

In particular, let $(C(M),\H,D)$ be the canonical spectral triple over a closed Riemannian spin manifold $M$ with representation $\pi:C(M)\to\BB(\H)$. Write $\Lip(M)$ for the corresponding Lipschitz algebra. 

\begin{example}
Suppose that we are given a unitary action of $\TT^2$ upon the spinor bundle over $M$. We denote the corresponding action on sections by
\begin{equation}\label{toract1}
U:\TT^2\to \BB(\H),\qquad t\mapsto U(t).
\end{equation}
This action descends to a group homomorphism $\TT^2\rightarrow \textup{Isom}(M)$. We write
\begin{equation}\label{toract2}
\alpha:\TT^2\to \Aut(C(M)),\qquad t\mapsto \alpha_t
\end{equation}
for the resulting isometric action of $\TT^2$ upon continuous functions obtained by pull-back of the action on $M$. Then by construction it follows that $(C(M),\H,D)$ is a torus-equivariant spectral triple over $C(M)$. As already shown in Proposition~\ref{pr:ueq}, the action \eqref{toract2} induces a completely isometric action upon the Lipschitz algebra $\Lip(M)$.
\end{example}


Let $\delta=(\delta_1,\delta_2)$ be the infinitesimal generator of the action \eqref{toract1}, meaning that for each $t\in \TT^2$ we have
$$
U(t)=\exp(i(t_1\delta_1+t_2\delta_2))
$$
for some real numbers $0\leq t_1, t_2\leq 2\pi$. Using this, one obtains a grading of the $C^*$-algebra $\BB(\H)$ by the Pontrjagin dual group $\ZZ^2$ by declaring an operator $T\in\BB(\H)$ to be of degree $(n_1,n_2)\in \ZZ^2$ if and only if
\begin{equation}\label{grad}
\alpha_t(T)=\exp(t_1n_1+t_2n_2)T \qquad \text{for all}~t\in\TT^2,
\end{equation}
where $\alpha_t(T)=U(t)TU(t)^*$ as above. As in \cite{Mes09b}, denote by $\textup{Sob}_{\delta}(\mathcal{H})$ the algebra of operators $T\in \BB(\H)$ that preserve the domain of $\delta$ and for which $[\delta,T]$ extends to an operator in $\BB(\H)$. Any such operator may be written as a norm convergent sum of homogeneous elements
$$
T=\sum_{(n_1,n_2)\in\ZZ^2} T_{n_1,n_2},
$$
where the operators $T_{n_1,n_2}$ are of degree $(n_1,n_2)$.


Let $\lambda:=\exp(2\pi i\theta)$, where $\theta$ is a real deformation parameter. Then for each operator $T\in \textup{Sob}_{\delta}(\mathcal{H})$ we define its \textbf{left twist} $\LL(T)$ to be the operator
$$
\LL(T):=\sum_{(n_1,n_2)\in\ZZ^2} T_{n_1,n_2}\lambda^{n_2p_1}.
$$
In particular, for homogeneous operators $x,y$ of respective degrees $(n_1,n_2)$ and $(m_1,m_2)$, if we define
$$
x\star y:= \lambda^{m_1n_2}xy,
$$
then a simple computation \cite{CL01} shows that $\LL(x)\LL(y)=\LL(x\star y)$. This new product $\star$ extends by linearity to an associative product on the vector space $\textup{Sob}_{\delta}(\mathcal{H})$.

In particular, 
by defining $\Lip(M_\theta)$ to be the vector space $\Lip(M)$ equipped with the twisted product $\star$, we obtain a new $*$-algebra equipped with a representation
$$
\pi_\theta: \Lip(M_\theta)\to\BB(\H),\qquad \pi_\theta(a)=\LL(a)
$$
by bounded operators upon the same Hilbert space $\H$ as before. Upon completion in this representation, we obtain the $C^{*}$-algebra $C(M_{\theta})$. 

Since the Dirac operator commutes with the torus action \eqref{toract1}, one immediately finds by straightforward computation that
$$
[D,\LL(a)]=\LL ([D,a])
$$
for all $a\in\Lip(M)$. This is enough to deduce as in \cite{CL01} that the commutators $[D,\pi_\theta(a)]$ are bounded operators for all $a\in \Lip(M)$ and then, since $\Lip(M)$ is dense in $C(M_\theta)$, that the datum $(C(M_\theta),\H,D)$ constitutes a torus-equivariant spectral triple over $C(M_\theta)$. We interpret this spectral triple as defining a spin geometry over the noncommutative manifold $M_\theta$, obtained by isospectral deformation of the spin geometry of the classical manifold $M$. The Lipschitz algebra of $(C(M_\theta),\H,D)$ is $\Lip(M_\theta)$.

\subsection{Spin geometry of the noncommutative torus} Following the construction of the previous section, here we recall how to obtain the spin geometry of the noncommutative two-torus $\TT^2_\theta$ from that of its classical counterpart
$$
\TT^2:=\{(u_1,u_2)\in \C^2~|~u_1^*u_1=u_2^*u_2=1\}.
$$

\begin{defn}
The $C^*$-algebra $C(\TT^2)$ of continuous functions on the two-torus $\TT^2$ is the universal unital commutative $C^*$-algebra generated by the unitary elements $U_1, U_2$.
\end{defn}

The $C^*$-algebra $C(\TT^2)$ carries an action of the two-torus $\bT^2$ by $*$-automorphisms, determined by the formul{\ae}
\begin{equation}\label{ta}
\alpha:\TT^2\to \textrm{Aut}(C(\TT^2)),\qquad U_i \mapsto \alpha_t(U_j):= e^{i t_j} U_j,\qquad i=1,2,
\end{equation}
where for $0\leq t_1,t_2\leq2\pi$ we use the notation $t=( e^{i t_1}, e^{i t_2})$ to describe a general element $t\in\TT^2$. This action is generated by a pair $\delta_1,\delta_2$ of fundamental vector fields, obeying
\begin{equation}\label{dU}
\delta_j(U_k)=\delta_{jk},\qquad j,k=1,2,
\end{equation}
where $\delta_{jk}$ denotes the Kronecker delta symbol. The spinor bundle $\cS$ over $\TT^2$ is trivializable and of rank two, so $\H:=L^2(\TT^2,\cS)\simeq L^2(\TT^2)\otimes \C^2$. The action \eqref{dU} lifts to the diagonal action on Hilbert space $\H$, upon viewing $C(\TT^2)$ as a dense subspace of $L^2(\TT^2)$.

Then, in terms of the (real) $2\times 2$ matrix generators $\gamma^k$, $k=1,2$, of the Clifford algebra $\C_1$,
the Dirac operator $D$ on the Hilbert space $\H$
is defined to be
\begin{equation}\label{tor-D}
D :=  i\delta_1\otimes \gamma^1+i\delta_2\otimes\gamma^2=\begin{pmatrix} 0& \delta_1+i\delta_2 \\ -\delta_1+i\delta_2 & 0 \end{pmatrix},
\end{equation}
the second equality following after choosing a pair of representatives for the gamma matrices. It is straightforward to check that the action \eqref{ta} makes the datum $(C(\TT^2),\H,D)$ into a torus-equivariant spectral triple.

It is immediate that the generators $U_1,U_2$ of the algebra $C(\TT^2)$ have respective degrees $(1,0)$ and $(0,1)$ with respect to the action \eqref{ta}. Given a real deformation parameter $\theta$ and setting $\lambda:=\exp(2\pi i \theta)$, we obtain the twisted operators
$$
\LL(U_1):=\begin{pmatrix}U_1 & 0 \\ 0 & U_1\end{pmatrix},\qquad \LL(U_2):=\begin{pmatrix}U_2 & 0 \\ 0 & U_2\end{pmatrix} \lambda^{\delta_1}.
$$
One finds that these twisted operators satisfy the commutation relations $$\LL(U_1)\LL(U_2)=\lambda\,\LL(\U_2)\LL(U_1)$$ as elements of $\BB(\H)$ and so, dropping the symbol $\LL$, we arrive at the following definition of the noncommutative two-torus.

\begin{defn}
With $\lambda:=\exp(2 \pi {i} \theta)$, the $C^*$-algebra of continuous functions on the noncommutative torus $\TT^2_\theta$ is the universal unital $C^*$-algebra generated by the unitary elements $U_1,U_2$ obeying the relation
$$
U_1 U_2= \lambda U_2 U_1.
$$
We denote this $C^*$-algebra by $C(\TT^2_\theta)$.
\end{defn}

As we did in the previous section, by defining the Hilbert space $L^2(\bT^2_\theta):=L^2(\bT^2)$, the operator $D$ of eq.~\eqref{tor-D} defines a Dirac operator on the Hilbert space 
$$
\H_\theta:=L^2(\TT^2_\theta,\cS)\simeq L^2(\TT^2_\theta)\otimes\C^2.
$$ 
Thus we recover the well known fact that the datum $(C(\TT^2_\theta), \H_\theta, D)$ constitutes a $2^+$-summable  torus-equivariant spectral triple over the $C^*$-algebra $C(\TT^2_\theta)$ \cite{C80}. We denote the Lipschitz subalgebra of this spectral triple by $\Lip(\TT_{\theta}^{2})$. It coincides with the algebra of elements $a\in C(\TT_{\theta}^{2})$ for which the function 
$$
\TT^{2}\rightarrow C(\TT^{2}_{\theta}),\qquad t\mapsto \alpha_{t}(a),
$$ 
is a Lipschitz function. 

The $C^*$-algebra $C(\TT^2_\theta)$ continues to carry an action of the classical two-torus $\bT^2$ by $*$-automorphisms, determined by the formul{\ae}
\begin{equation}\label{tor-act}
\alpha:\TT^2\to \textrm{Aut}(C(\TT^2_\theta)),\qquad U_i \mapsto \alpha_t(U_j):= e^{i t_j} U_j,\qquad i=1,2,
\end{equation}
where we write $t=( e^{i t_1}, e^{i t_2})$, $0\leq t_1,t_2\leq2\pi$, to denote a general element $t\in\TT^2$. 


\subsection{The noncommutative torus as a fibration over the circle}
In order to describe the noncommutative torus $\TT^2_\theta$ as the total space of a fibration over a classical base space, let us now search for a commutative subalgebra of $C(\TT^2_\theta)$.  Consider the action $U:\TT^2\to\BB(\H_\theta)$  and extend it to $\mathcal{H}_{\theta}\oplus\mathcal{H}_{\theta}$ diagonally. By definition of $\Lip(\TT^{2}_{\theta})$, this action satisfies the equality
\begin{equation}\label{Lipact}\begin{pmatrix} U(t) & 0\\ 0 & U(t)\end{pmatrix}\begin{pmatrix} a & 0 \\ [D,a] & a\end{pmatrix}\begin{pmatrix} U(t^{-1}) & 0 \\ 0& U(t^{-1})\end{pmatrix}=\begin{pmatrix}\alpha_{t}(a) & 0\\ [D,\alpha_{t}(a)] & \alpha_{t}(a)\end{pmatrix},\end{equation}
so that $\alpha_{t}$ acts on $\Lip(\TT^{2}_{\theta})$ by isometries.
We write $\TT^2=\TT\times\TT'$, where 
\begin{equation}\label{tor-dec}
\TT:=\{(e^{i t},0)\}\in\TT^2~|~t\in\RR\}, \qquad \TT':=\{(0,e^{{i} t})\in\TT^2~|~t\in\RR\},
\end{equation}
and look for the fixed point subalgebra of $\Lip(\TT^2_\theta)$ under the action of $\TT$ induced by eq.~\eqref{Lipact}, which we continue to denote by $\alpha:\TT\to\textup{Aut}(\Lip(\TT^2_\theta))$.  

In terms of Definition~\ref{de:can-st}, the canonical spectral triple on the classical manifold $\bS^1$ is the $1^+$-summable datum $(C(\bS^1),\h,\dirac)$, where $\h=L^2(\bS^1)$ is the Hilbert space of square-integrable functions on $\bS^1$ and $\dirac=i \textup{d}/\textup{d}x$ is the Dirac operator on $\h$. The Lipschitz algebra of this spectral triple is the usual algebra $\Lip(\bS^{1})$ of Lipschitz functions on the circle.  The pair $(\h,\dirac)$ defines an odd unbounded KK-cycle and hence an element of $\Psi_{-1}(\Lip(\bS^1),\C)$. 

\begin{lem}\label{le:inv} The $\TT$-invariant part $(C(\TT^{2}_{\theta})_{0},L^2(\TT^2_\theta,\mathcal{S})_0,D_{0})$ of the datum $(C(\TT^2_\theta), \H_\theta, D)$  is a spectral triple that is unitarily equivalent to $(C(\bS^1),\h,\dirac)$.
\end{lem}

\begin{proof}It is clear by inspection of eq.~\eqref{Lipact} that the invariant subspace of the action of $U(t)$ on $\mathcal{H}_{\theta}$ is spanned by (the image in $\H_\theta$ of) the vectors $\{U_{2}^{j}~|~ j\in\ZZ\}$ and that this subspace is isomorphic to $L^{2}(\bS^{1})$. Similarly, the fixed-point subalgebra of $C(\TT^{2}_{\theta})$ under this action is the unital $C^*$-algebra generated by $U_{2}$, i.e. it is isomorphic to $C(\bS^{1})$. The invariant part of $D$ is the operator 
$$
D_0:\Dom(D_0)\to L^2(\TT^2_\theta,\mathcal{S})_0, \qquad D_{0}=\begin{pmatrix} 0& i\delta_{2} \\ i\delta_{2} & 0\end{pmatrix},
$$ 
so the invariant spectral triple is exactly the spectral triple one obtains by applying the doubling construction from Remark~\ref{re:double} to the datum $(C(\bS^1),\h,\dirac)$.
\end{proof}

\begin{prop}\label{pr:inv}
The fixed-point subalgebra $\Lip(\TT^{2}_{\theta})_{0}$ of $\Lip(\TT^2_\theta)$ under the action of $\TT$ induced by \eqref{Lipact} is isometrically isomorphic to the commutative operator algebra $\Lip(\bS^1)$ of Lipschitz functions on the circle.
\end{prop}

\begin{proof}  This follows from the above lemma by observing that for $a\in C(\TT^{2}_{\theta})_{0}\cap\Lip(\TT^2_\theta)$, we have that
\[\bigg[\begin{pmatrix} 0 & \delta_{1}+i\delta_{2} \\ -\delta_{1}+i\delta_{2} & 0\end{pmatrix},\begin{pmatrix}
\pi(a) & 0 \\ 0 & \pi(a)\end{pmatrix}\bigg] =\begin{pmatrix} 0 & i[\delta_{2},\pi(a)] \\ i[\delta_{2},\pi(a)] & 0\end{pmatrix}.\] From this it follows immediately that for such $a$ we have $\|\pi_{D}(a)\|=\|\pi_{\dirac}(a)\|$.
\end{proof}

We can integrate elements of $\mathbb{B}(\mathcal{H}_\theta)$ along the orbits of the torus $\bT \subseteq \bT^2$, this time defining a completely positive map $\tau_{0}:\mathbb{B}(\mathcal{H}_{\theta})\rightarrow \mathbb{B}(\mathcal{H}_{\theta})$. This map restricts to $\tau_0 : C(\TT^2_\theta) \to C(\bS^1)$ with values in the invariant subalgebra defined in Lemma~\ref{le:inv}:
$$
\tau_0:C(\TT^2_\theta)\to C(\bS^1),\qquad \tau_0(a):=\int_{\bT} \alpha_{t} (a) dt.
$$
\begin{lem}\label{lem:contractive} The map $\Lip(\TT_{\theta}^{2})\rightarrow \M_{2}(\mathbb{B}(\mathcal{H}_\theta))$ defined by $\begin
{pmatrix} a & 0\\ [D,a] & a\end{pmatrix}\mapsto \begin{pmatrix} a &0 \\ [D_{0},a] & a\end{pmatrix}$ is completely contractive.
\end{lem}

\begin{proof} Consider the odd self-adjoint unitary operator $u\in \mathbb{B}(\mathcal{H}_{\theta})$ that interchanges the two copies of $L^{2}(\TT^{2}_{\theta})$. Since $a$ acts by even endomorphisms, we have $u\pi(a)u=\pi (a)$, and also $uD_{0}u=D_{0}$. Since $u\begin{pmatrix}0&\delta_{1} \\ -\delta_{1} &0\end{pmatrix} u=\begin{pmatrix}0&-\delta_{1} \\ \delta_{1} &0\end{pmatrix}$, we have that
\[\pi_{D}(a)+u\pi_{D}(a)u=2\pi_{D_{0}}(a).\]
Therefore $\|\pi_{D_{0}}(a)\|\leq \|\pi_{D}(a)\|$. This clearly extends to matrices.
\end{proof}

Since $\tau_{0}$ is completely positive, its extension to $\mathbb{B}(\mathcal{H}_{\theta}\oplus\mathcal{H}_{\theta})\cong \M_{2}(\mathbb{B}(\mathcal{H}_{\theta}))$ is positive, and it is immediate that $\tau_{0}:\Lip(\TT^{2}_{\theta})\rightarrow\Lip(\bS^{1})$. We denote by $\mathcal{E}_{\theta}$ the completion of $\Lip(\TT_{\theta}^{2})$ in the norm
\begin{equation}\label{Lipnorm}\|a\|_{\mathcal{E}_{\theta}}:=\|\tau_{0}(\pi_{D_{0}}(a)^{*}\pi_{D_{0}}(a))\|^{1/2}.\end{equation}
The map $\tau_0$ induces a Hermitian structure on $C(\TT^2_\theta)$ with values in the invariant subalgebra $C(\bS^1)$. To prove this, we denote by $\mathpzc{E}_\theta$ the completion of $C(\TT^2_\theta)$ in the norm 
$$
\| a \|_{\mathpzc{E}_\theta} := \| \langle a, a \rangle \|_{C(S^1)}^{1/2},
$$ 
where the $C(S^1)$-valued inner product $\langle\cdot,\cdot\rangle$ is defined by \begin{equation}
\label{inp}\langle a, b \rangle := \tau_0 (a^* b) 
\end{equation}
for each $a,b \in C(\bT^2_\theta)$. It is clear that $\mathcal{E}_{\theta}\subset\mathpzc{E}_{\theta}$ densely and that $\mathpzc{E}_{\theta}$ is a $C^{*}$-module. Moreover, we find the following result.

\begin{prop}
\label{prop:hilbert-module-torus}
The completion $\mathcal{E}_\theta$ is a right Lipschitz module over $\Lip(\bS^1)$ isomorphic to $L^{2}(\bS^{1})\,\hotimes\, \Lip(\bS^{1})$. Multiplication in $\Lip(\TT_{\theta}^{2})$ induces a completely contractive $*$-homomorphism $\Lip(\TT_{\theta}^{2})\rightarrow \End^{*}_{\Lip(\bS^{1})}(\mathcal{E}_{\theta})$. Consequently, $C(\TT_{\theta}^{2})$ is represented upon $\mathpzc{E}_{\theta}$ by a $*$-homomorphism.
\end{prop}

\begin{proof} It is clear that, on the dense right $\Lip(\bS^{1})$-submodule of $\Lip(\TT^{2}_\theta)$ spanned by $\{U_{1}^{k}~|~ k\in\Z\}$, the norm \eqref{Lipnorm} coincides with the norm 
\[\|\sum_{k=-n}^{n}U_{1}^{k} g_{k}\|^{2}=\|\sum_{k=-n}^{n}\pi_{D_{0}}(g_{k})^{*}\pi_{D_{0}}(g_{k})\|,\]
and therefore its completion is isomorphic to $L^{2}(\bS^{1})\hotimes\Lip(\bS^{1})$, which is evidently a Lipschitz module. From positivity of the map $\tau_{0}$, and using Lemma \ref{lem:contractive}, we find that for $a,b\in\Lip(\TT_{\theta}^{2})$ we have
\[\begin{split}\tau_{0}\left(\pi_{D_{0}}(ba)^{*}\pi_{D_{0}}(ba)\right)&=\tau_{0}(\pi_{D_{0}}(a)^{*}\pi_{D_{0}}(b)^{*}\pi_{D_{0}}(b)\pi_{D_{0}}(a))\\
&\leq\|\pi_{D_{0}}(b)^{*}\pi_{D_{0}}(b)\|\tau_{0}(\pi_{D_{0}}(a)^{*}\pi_{D_{0}}(a))\\
&\leq\|\pi_{D}(b)^{*}\pi_{D}(b)\|\tau_{0}(\pi_{D_{0}}(a)^{*}\pi_{D_{0}}(a)),\end{split}\]
and therefore
\[\begin{split} \|ba\|_{\mathcal{E}_{\theta}}^{2}&=\|\tau_{0}(\pi_{D_{0}}(ba)^{*}\pi_{D_{0}}(ba))\| \\
&=\|\tau_{0}(\pi_{D_{0}}(a)^{*}\pi_{D_{0}}(b)^{*}\pi_{D_{0}}(b)\pi_{D_{0}}(a)\| \\
&\leq \|\pi_{D}(b)^{*}\pi_{D}(b)\| \|\tau_{0}(\pi_{D_{0}}(a)^{*}\pi_{D_{0}}(a))\| \\
&=\|b\|_{ \Lip(\TT_{\theta}^{2})}^{2}\|a\|^{2}_{\mathcal{E}_{\theta}}, 
\end{split}\]
proving that we have a contractive representation $\Lip(\TT_{\theta}^{2})\rightarrow \End^{*}_{ \Lip(S^{1}) }(\mathcal{E}_{\theta})$. The second statement now follows by taking $C^{*}$-envelopes.
\end{proof}

The generator of the $\TT$-action on $\Lip(\TT^{2}_{\theta})$ defines an unbounded operator $T=\delta_{1}$ on the Lipschitz module $\mathcal{E}_{\theta}$, and the isomorphism 
$$
\mathcal{E}_{\theta}\cong L^{2}(\bS^{1})\hotimes\Lip(\bS^{1})\cong \ell^{2}(\ZZ)\hotimes\Lip(\bS^{1})
$$ 
gives us a canonical Grassmann connection $\nabla:\mathcal{E}_{\theta}\rightarrow\mathcal{E}_{\theta}\hotimes_{\Lip(\bS^{1})}\Omega^{1}(C(\bS^{1}),\Lip(\bS^{1}))$.

\begin{thm}The triple $(\mathcal{E}_{\theta},T,\nabla)$ defines a Lipschitz cycle in $\Psi^{\ell}_{-1}(\Lip(\TT^{2}_{\theta}),\Lip(\bS^{1}))$.
\end{thm}

\begin{proof} The operator $T$ is self-adjoint and regular with compact resolvent in $\mathcal{E}_{\theta}$ because, under the isomorphism
\begin{equation}\label{eq:triv}\mathcal{E}_{\theta}\cong L^{2}(\bS^{1})\,\hotimes\,\Lip(\bS^{1})\cong \ell^{2}(\Z)\,\hotimes\,\Lip(\bS^{1}) \end{equation}
determined by  $U_{1}^{k}\otimes f\mapsto e_{k}\otimes f$, it corresponds to the number operator $$t\otimes 1: e_{k}\otimes f\mapsto ke_{k}\otimes f.$$ 
This operator is closed since $\mathfrak{G}(t\otimes 1)\cong\mathfrak{G}(t)\hotimes\Lip(\bS^{1})$. Therefore $T\pm i$ have dense range in $\mathcal{E}_{\theta}$ and $(T\pm i)^{-1}$ is bounded for the Lipschitz operator norm. Thus, by Lemma \ref{pmi}, $T$ is self-adjoint and regular. The compact resolvent property follows from the fact that $$\K_{\Lip(\bS^{1})}(\ell^{2}(\Z)\,\hotimes\,\Lip(\bS^{1}))\cong\K(\ell^2(\ZZ))\,\hotimes\,\Lip(\bS^{1})$$ and $(t\pm i)^{-1}\in\K(\ell^2(\ZZ))$. For $a\in\Lip(\TT^{2}_{\theta})$, the commutators $[T,a]$ are by definition bounded for the norm in $\End^{*}_{C(\bS^{1})}(\EE_{\theta})$. Under the isomorphism \eqref{eq:triv}, the connection $\nabla$ corresponds to the connection $e_{k}\otimes f\mapsto e_{f}\otimes \d f$ and is thus completely bounded and satisfies $[\nabla,T]=0$.
\end{proof} 
On the level of KK-theory, this theorem is a special case of the construction of Kasparov modules from circle actions in \cite{CNNR08}. Next we consider the Hilbert space $L^2(\TT^2_\theta)$ and its relation to the space $L^2(\bS^1)$ of square-integrable functions on the base space $\bS^1$, whose inner product we denote by $(\cdot,\cdot)_{\bS^1}$. For this we consider the tensor product of Hilbert modules $\mathpzc{E}_\theta \otimes_{C(\bS^1)} L^2(\bS^1)$, which we equip with the inner product
\begin{equation}
\label{eq:inner-tensor-torus}
\left ( e \otimes h, e' \otimes h' \right ) := (h, \langle e,e'\rangle h')_{\bS^1}
\end{equation}
for each $e\otimes h,~ e' \otimes h' \in \mathpzc{E}_\theta\, \otimes_{C(\bS^1)} L^2(\bS^1)$.

\begin{prop}\label{pr:fac1}
The Hilbert space $L^2(\TT^2_\theta)$ is isomorphic to the completion $\mathcal{E}_{\theta}\hotimes_{\Lip(\bS^{1})} L^{2}(\bS^{1})\cong \mathpzc{E}_\theta\, \hotimes_{C(\bS^1)} L^2(\bS^1)$ of the tensor product $\mathpzc{E}_\theta \otimes_{C(\bS^1)} L^2(\bS^1)$
with respect to the inner product \eqref{eq:inner-tensor-torus}.
\end{prop}

\proof  The first isomorphism follows directly from Corollary~\ref{co:*hom}, while the second follows from proposition \ref{prop:hilbert-module-torus} and the corresponding isomorphism in the classical case.
\endproof

Together with the spectral triple on  $\TT^2_\theta$, which in turn defines an even unbounded cycle in $\Psi_0(\Lip(\TT^2_\theta),\C)$, these considerations lead to the following theorem. The spin geometry of the noncommutative torus has of course remained a fundamental example since the beginning of the theory \cite{C80}.

\begin{thm}
As an element of $\Psi_0(\Lip(\bT^2_\theta),\C)$, the Riemannian spin geometry of $\TT^2_\theta$ factorizes as a Kasparov product of unbounded KK-cycles, namely 
$$
(\H_\theta, D)\simeq (\mathcal{E}_\theta, T,\nabla) \otimes_{\Lip(\bS^1)} (\H, \dirac),
$$
where $(\mathcal{E}_\theta, T,\nabla) \in \Psi_{-1}^{\ell}(\Lip(\bT^2_\theta), \Lip(\bS^1))$ and $(\H,\dirac) \in \Psi_{-1}(\Lip(\bS^1), \C)$. 
\end{thm}

\begin{proof}
In order to compute the Kasparov product of these two odd cycles, we follow the method of Example~\ref{ex:KK-odd}. Proposition~\ref{pr:fac1} gives us the necessary isomorphism at the level of modules. All of the necessary analytic details have been verified, so it remains to check that the product operator agrees with the Dirac operator on $\TT^2_\theta$. This follows immediately by inspection of the formula \eqref{odd-odd-op}.\end{proof}

This KK-factorization allows for the following gauge theoretical interpretation. The $C^*$-module $\EE_\theta$ is the space of continuous sections of some Hilbert bundle $V \to \bS^1$. Essentially, the fibers of $V$ are copies of $L^2(\bS^1)$. The internal gauge group $\U(\Lip(\T^2))$ is a normal subgroup of the group $\GG(\E_\theta)$ of Definition~\ref{defn:KK-gauge}, which acts fibrewise on $V$. 
The internal gauge fields in $\Omega^1_D(\Lip(\T^2))$ decompose according to Lemma \ref{lem:inner-decomp}: the scalar fields act vertically on the Hilbert bundle $V$, whilst the gauge fields are given by connections thereon. 

We have thus cast the gauge theory as described by the spin geometry of the noncommutative torus $\T^2$ into a geometrical setting consisting of a Hilbert bundle over the circle $\bS^1$, equipped with a connection and a fibrewise endomorphism. Interestingly, in passing from $\U(\Lip(\T^2))$ to $\GG(\E_\theta)$ we allow for more gauge degrees of freedom, and in particular those of a type encountered in noncommutative instanton searches \cite{BL09a,bvs,B13} (and for an early appearance \cite{lvs:tcs}). Namely, the Pontrjagin dual group $\Z$ of $\TT$ acts on $\E_\theta$ through the bounded operators $e^{2 \pi i n \theta T} \equiv \lambda^{n \theta T}$ for any $n \in \Z$. One easily checks that this $\Z$ is a subgroup of $\GG(\E_\theta)$ and the relevant extension of $\U(\Lip(\T^2))$ to consider is the semidirect product $\U(\Lip(\T^2)) \rtimes \Z < \GG(\E_\theta)$.

\section{The Noncommutative Hopf Fibration}
\label{sect:nc-hopf}
Next we come to investigate the spin geometry of the toric noncommutative Hopf fibration $\bS^3_\theta\to \bS^2$. In contrast with the example of the fibration $\TT^2_\theta\to \bS^1$ given in the previous section, this `quantum principal bundle' gives us an example of a topologically non-trivial fibration of noncommutative manifolds. 

This quantum fibration will be described at the topological level in terms of an algebra inclusion $C(\bS^2)\hookrightarrow C(\bS^3_\theta)$, as a noncommutative analogue of the familiar classical Hopf fibration $\bS^3 \to \bS^2$ with structure group $\U(1)$. The aim of this section is to spell out the noncommutative spin geometry of this fibration in full detail, using the language of Kasparov products in unbounded KK-theory.

\subsection{The noncommutative three-sphere} We begin by describing the spin geometry of the noncommutative three-sphere $\bS^3_\theta$, obtained as an isospectral deformation of the Riemannian spin geometry of the classical sphere
$$
\bS^3=\{ (v_1 ,v_2) \in \C^2~|~ v_1^*v_1 + v_2^*v_2=1\}.
$$
The latter has a natural parametrization in terms of polar coordinate functions
\begin{gather*}
v_1 = e^{i t_1} \cos \chi, \qquad 
v_2 = e^{i t_2} \sin \chi.
\end{gather*}
where the toroidal coordinates $0 \leq t_1,t_2 \leq 2\pi$ together parametrize a two-torus and $0 \leq \chi \leq \pi/2$ is the polar angle.
%
%

\begin{defn}
The $C^*$-algebra $C(\bS^3)$ of continuous functions on $\bS^3$ is the universal commutative unital $C^*$-algebra generated by the elements
$$
V_1:=U_1 \cos\chi,\qquad V_2:=U_2\sin\chi
$$
for unitary elements $U_1,U_2$ and $0\leq \chi\leq \pi/2$ the polar angle.
\end{defn}

The $C^*$-algebra $C(\bS^3)$ carries a natural action of the two-torus $\bT^2$ by $*$-automorphisms, defined on generators by 
\begin{equation}\label{sph-act}
\alpha:\TT^2\to\textup{Aut}(C(\bS^3)),\qquad \alpha_t(V_j):=e^{i t_j}V_j,\qquad j=1,2,
\end{equation}
where we denote a general element of the torus $\TT^2$ by $t=(e^{i t_1},e^{i t_2})$ for $0\leq t_1,t_2\leq 2\pi$.

Since the classical three-sphere $\bS^3\simeq \SU(2)$ is in particular a group, the rank two spinor bundle over $\bS^3$ is trivializable. Thus we find that
\begin{equation}\label{sp-triv}
\H:=L^2(\bS^3,\cS)\simeq L^2(\bS^3)\otimes \C^2
\end{equation}
is the Hilbert space of square-integrable sections of the spinor bundle $\cS$ over $\bS^3$. Immediately we obtain a continuous representation $\pi:C(\bS^3)\to \mathbb{B}(\H)$ of $C(\bS^3)$ on $\H$ acting as diagonal operators.

The Dirac operator $D$ for the round metric on the classical sphere $\bS^3$ is an unbounded self-adjoint operator on the Hilbert space $\H$. With $\gamma^2$, $\gamma^3$ the generators of the Clifford algebra $\C_1$ and $\gamma^1$ its $\ZZ_2$-grading, the Dirac operator has the explicit form
\begin{equation}
\label{dirac-S3}
D = i Z_1\otimes\gamma^1 + iZ_2\otimes\gamma^2 + iZ_3\otimes\gamma^3 + \frac{3}{2},
\end{equation}
where $Z_k$, $k=1,2,3$, denote the corresponding (right) fundamental vector fields on the group manifold $\SU(2)$ ({\em cf}. \cite{Hom00}).

Upon making an explicit choice of representatives for the gamma matrices, we may write 
$$
D=\begin{pmatrix} iZ_1 & Z_2+iZ_3 \\ -Z_2+i Z_3 & -iZ_1\end{pmatrix} +\frac{3}{2}.
$$
For later convenience we introduce the `laddering' and `counting' operators $Z_{\pm}:=\pm Z_2+i Z_3$ and $T:=iZ_1$. These satisfy $T^* = T$, $Z_\pm^* =Z_\mp$ and the crucial property
\begin{equation}\label{ladder}
[T,Z_\pm] = \pm 2 Z_\pm,
\end{equation}
explaining their respective names. From now on we shall omit the tensor product symbols from expressions such as eq.~\eqref{dirac-S3}

The spin lift of the torus action \eqref{sph-act} on $C(\bS^3)$ defines a unitary representation $U:\bT^2\to \mathbb{B}(\H)$ generated by the pair of commuting derivations
\begin{equation}\label{torgen} 
H_{1}=-\gamma_{1} -iZ_{1},\quad  H_{2}=i\tilde{Z}_{1},
\end{equation}
the latter coming from the {\it left} fundamental vector field $\tilde Z_1$ on the classical sphere $\bS^3\simeq\SU(2)$. Again we refer to \cite{Hom00} for further details (noting that the torus acting on $\H$ is in fact a double cover of the torus acting on $\bS^3$; we take the liberty of being notationally sloppy about this point).

In terms of the resulting $\ZZ^2$-grading \eqref{grad} of the algebra $\BB(\H)$, one easily checks that the generators $V_1,V_2$ of the $C^*$-algebra $C(\bS^3)$ have respective degrees $(1,1)$ and $(1,-1)$. It will be convenient to simplify our notation, often writing
\begin{equation}\label{relab}
a:=V_1,\qquad b:=V_2.
\end{equation}
This time taking $\lambda:=\exp(\pi i \theta)$ (due to the aforementioned issues regarding double covers), the corresponding twisted operators are therefore given by the formul{\ae}
\[\LL(a)=\begin{pmatrix}a & 0 \\ 0 & a\end{pmatrix}\lambda^{H_{1}},\qquad \LL(b)=\begin{pmatrix} b & 0 \\ 0 & b\end{pmatrix}\lambda^{-H_{1}}.\]
Immediately one verifies the commutation relation $\LL(a)\LL(b)=\lambda^2 \LL(b)\LL(a)$, leading to the following definition ({\em cf}. \cite{Mat91}).

\begin{defn}
With $\lambda:=\exp(\pi i \theta)$, the $C^*$-algebra of continuous functions on the noncommutative three-sphere $\bS^3_\theta$ is the universal unital $C^*$-algebra generated by the elements
\begin{gather*} V_1 = U_1\cos \chi, \qquad 
V_2 = U_2\sin \chi,
\end{gather*}
where $U_1,U_2$ are unitary elements obeying the relation $U_1U_2=\lambda^2 U_2U_1$ and $0 \leq \chi \leq \pi/2$ is the polar angle. We denote this $C^*$-algebra by $C(\bS^3_\theta)$.
\end{defn}

Following the isospectral deformation procedure described in the previous section, we take
\begin{equation}\label{sp-triv2}
\H_\theta=L^2(\bS^3_\theta,\cS):= L^2(\bS^3)\otimes \C^2
\end{equation}
for the Hilbert space of square-integrable sections of the spinor bundle $\cS$ over the noncommutative sphere $\bS^3_\theta$. Immediately we obtain a continuous representation $\pi:C(\bS^3)\to \mathbb{B}(\H_\theta)$ of $C(\bS^3)$ on $\H_\theta$ acting as diagonal operators The formula \eqref{tor-D} continues to define a Dirac operator on the Hilbert space $\H_\theta$.

\begin{prop}\label{eq triple}
The datum $(C(\bS^3_\theta), \H_\theta, D)$ constitutes a $3^+$-summable torus-equivariant spectral triple over the $C^*$-algebra $C(\bS^3_\theta)$.
\end{prop}

\proof By computing directly that $-\gamma^1 - iZ_1$ commutes with $D$, as obviously does $\tilde Z_1$, the result follows from checking that 
\begin{align}\label{eq-tr}
U(t) D U(t)^{-1} = D,\qquad U(t) \pi(a) U(t)^* = \pi(\alpha_{t}(a)),
\end{align}
for all $a \in C(\bS^3_\theta)$ and all $t\in\bT^2$.\endproof

The associated Lipschitz algebra is denoted $\Lip(\bS^{3}_{\theta})$. It consists of those elements $a\in C(\bS^{3}_{\theta})$ for which the map $t\mapsto \alpha_{t}(a)$ is a Lipschitz function $\TT^{2}\rightarrow C(\bS^{3}_{\theta})$.

\begin{cor}
The datum $(\H_\theta, D)$ constitutes a cycle in the set $\Psi_{-1}(\Lip(\bS^3_\theta), \C)$ of odd unbounded KK-cycles.
\end{cor}

\proof This follows immediately from the fact that the datum $(C(\bS^3_\theta), \H_\theta, D)$ constitutes an odd spectral triple over $C(\bS^3_\theta)$ with Lipschitz algebra $\Lip(\bS^{3}_{\theta})$.\endproof

\subsection{The classical two-sphere}  
This time we write $\TT^2=\TT\times\TT'$, where 
\begin{equation}\label{tor-dec2}
\TT:=\{(e^{i t},e^{i t})\in\TT^2~|~t\in\RR\}, \qquad \TT':=\{(e^{i t},e^{-i t})\in\TT^2~|~t\in\RR\},
\end{equation}
and look for the fixed point subalgebra of $C(\bS^3_\theta)$ under the action of $\TT$ induced by eq.~\eqref{sph-act}, which we continue to denote by $\alpha:\TT\to\textup{Aut}(C(\bS^3_\theta))$. Starting with the spectral geometry of $\bS^3_\theta$ described by the above spectral triple, we compute its $\bT$-invariant part and show that the resulting datum is unitarily equivalent to the canonical spectral triple on the base space $\bS^2$ of the Hopf fibration.

\begin{prop}\label{pr:inv2}
The fixed-point subalgebra of $C(\bS^3_\theta)$ under the action of $\TT$ induced by \eqref{sph-act} is isomorphic to the commutative $C^*$-algebra $C(\bS^2)$ of continuous functions on the classical two-sphere.
\end{prop}

\proof It is clear by inspection that the fixed-point subalgebra of $C(\bS^3_\theta)$ under the action of $\TT$ is the universal unital $C^*$-algebra generated by the complex element $W:=U_1^*U_2\sin\chi \cos\chi$ and the real element $x:=\cos^2 \chi$, with $0\leq\chi\leq \pi/2$ the polar angle. One readily checks that 
$$
W^*W=WW^*=x(1-x),
$$
so this commutative $C^*$-algebra is nothing other than the algebra $C(\bS^2)$ of continuous functions on the classical two-sphere of radius $1/2$.\endproof

Next we describe the geometry of the base space $\bS^2$ of the Hopf fibration.
Denote by  $\Lip(\bS^3_\theta)_0$ and $L^2(\bS^3_\theta,\cS)_0$ the $\bT$-invariant subspaces of $\Lip(\bS^3_\theta)$ and $L^2(\bS^3_\theta,\cS)$, respectively. 
The space $\Lip(\bS^3_\theta)$ decomposes into homogeneous spaces of weight $n\in\ZZ$ under the action of the operator $T:={i}Z_1$, 
$$
\mathcal{L}_n:=\{a\in \Lip(\bS^3_\theta)~|~T(a)=na\}.
$$
Our choice of notation is deliberately suggestive of line bundles, where $\L_n$ is thought of  as the space of Lipschitz sections of the line bundle over $\bS^2$ with first Chern number $n$. This yields the familiar Peter-Weyl decomposition of $\Lip(\bS^3_\theta)$ into weight spaces
\begin{equation}\label{decomp-ln}
\Lip(\bS^3_\theta) \simeq \bigoplus_{n \in \Z} \L_n
\end{equation}
and from eq.~\eqref{ladder} we find that $Z_\pm : \L_n \to \L_{n\pm 2}$. 
In what follows, we shall say that each element $x\in\L_n$ is \textbf{homogeneous of degree $n\in\ZZ$} with respect to the decomposition \eqref{decomp-ln}. 

The operator $T$ is the infinitesimal generator of the $\TT$-action on $\Lip(\bS^3_\theta)$ induced by eq.~\eqref{sph-act} and so $\L_0$ is the algebra of invariant Lipschitz functions on $\bS^3_\theta$ which identifies with a dense $*$-subalgebra of $C(\bS^{2})$. The product and involution in $\Lip(\bS^3_\theta)$ therefore induce the identifications of (sections of) line bundles
\begin{equation}
\label{eq:line-product}
\L_m \otimes_{\L_0} \L_n \simeq \L_{m+n}, \qquad\mathcal{L}_n^*\simeq \mathcal{L}_{-n},
\end{equation}
as one should expect from the classical case. In particular, the generators $a,b$ of eq.~\eqref{relab} are elements of the $\L_0$-bimodule $\mathcal{L}_1$, whereas their conjugates $a^*,b^*$ are elements of $\L_{-1}$.

There are two combinations of these line bundles which are of particular interest, respectively forming the Lipschitz sections of the spinor bundle $\cS$ and the cotangent bundle $\Lambda^1$ on $\bS^2$, as the following result shows. These will prove useful in the final section of the paper.

\begin{lem}\label{le:idents}
There are explicit isomorphisms of $\mathcal{L}_{0}$-bimodules of Lipschitz sections
$$
\Gamma^\ell(\bS^2,\Lambda^1) \simeq \L_2 \oplus \L_{-2},\qquad \Gamma^\ell(\bS^2,\cS) \simeq \L_1 \oplus \L_{-1}.
$$
\end{lem}

\proof By definition, the decomposition \eqref{decomp-ln} is equivariant under the $\mathbb{T}$-action on $\Lip(\bS^3_\theta)$, whence the $\mathcal{L}_{0}$-bimodules $\mathcal{L}_n$ are isomorphic as vector spaces to their classical counterparts. The result now follows from the corresponding classical isomorphisms.
\endproof

More precisely, the latter result means that every one-form $\omega\in\Gamma^\ell(\bS^2,\Lambda^1)$ may be written uniquely as
\begin{equation}\label{oneforms}
\omega=f_+\omega_++f_-\omega_-,\qquad \text{for some}~f_\pm\in \mathcal{L}_{\pm 2},
\end{equation}
with $\omega_\pm$ the left-invariant one-forms on $\bS^3_\theta$ which are dual to the vector fields $Z_\pm$. The relationship between the vector bundles $\cS$ and $\Lambda^1$ is expressed through the usual Clifford multiplication 
\begin{align*}
c:\Gamma^\ell(\bS^2,\Lambda^1) \to\textup{End}_{\L_{0}}^{*}(\Gamma^\ell(\bS^2,\cS)).
\end{align*}
Recall our choice of representatives for the gamma matrices,
$$
\gamma^1= \begin{pmatrix}1&0\\0&-1\end{pmatrix},\qquad \gamma^2= \begin{pmatrix}0&-i\\i&0\end{pmatrix},\qquad  \gamma^3= \begin{pmatrix}0&1\\1&0\end{pmatrix},
$$
written in terms of the basis that decomposes $\Gamma^\ell(\bS^2,\cS)$ into a direct sum. 
Using these, we introduce the matrices
$$
\sigma^\pm:=\tfrac{1}{2}(\pm i\gamma^2+\gamma^3);\quad\qquad\sigma^+ = \begin{pmatrix} 0 & 1 \\ 0 & 0 \end{pmatrix}, \qquad 
\sigma^- = \begin{pmatrix} 0 & 0 \\ 1 & 0 \end{pmatrix}.
$$
The Clifford multiplication on $\bS^2$ is then conveniently expressed in the form
\begin{align*}
c(\omega)(s)=  f_+ c(\omega_+) s + f_- c(\omega_-) s 
=f_+ \sigma_+ s + f_- \sigma_- s  = (f_+ s_-,   f_- s_+),
\end{align*}
for each $s=(s_+,s_-)\in\mathcal{L}_1\oplus\mathcal{L}_{-1}$. The grading on $\Gamma^\ell(\bS^2,\cS)$ determined by the matrix $\gamma^1$ makes the Clifford action into an odd representation of $\Gamma^\ell(\bS^2,\Lambda^1)$ on $\Gamma^\ell(\bS^2,\cS)$.

From Proposition~\ref{eq triple} we know that the Dirac operator $D$ of eq.~\eqref{dirac-S3} is $\bT^2$-equivariant, whence it induces an unbounded self-adjoint operator $D_0$ on $L^2(\bS^3,\cS)_0$. Since $(\gamma^1)^2=1$, we may rewrite eq.~\eqref{dirac-S3} as
$$
D=-(-\gamma^1-iZ_1)\gamma^1 + i Z_2\gamma^2 +i Z_3\gamma^3 + \frac{1}{2}.
$$
The $\U(1)$-invariant part $L^2(\bS^3,\cS)_0$ is the closed subspace of the spinor space $L^2(\bS^3)\otimes\C^2$ that is annihilated by the operator $H_1=-\gamma^1-iZ_1$, the latter being the spin lift of the infinitesimal $\U(1)$ generator $T=iZ_1$. We deduce that we may write
$$
D_0 = Z_2\sigma^2 + Z_3\sigma^3 + \frac{1}{2}.
$$
as an operator on the invariant spinors $L^2(\bS^3_\theta,\cS)_0$. Since the matrices $\sigma_\pm$ correspond to the Clifford multiplication of the one-forms $\omega_\pm$, we deduce that the exterior derivative on the two-sphere has the form
\begin{equation}\label{extder}
\d:\Lip(\bS^2)\to\Omega^1_{D_0}(C(\bS^2),\Lip(\bS^2)),\qquad \d f=[D_0,f]=Z_+(f)\omega_+ + Z_-(f)\omega_-,
\end{equation}
where $Z_\pm:=\pm Z_2+iZ_3$ are the vector fields appearing in the above form of the Dirac operator and $\omega_\pm$ are the one-forms from eq.~\eqref{oneforms}. 

\begin{prop}
The datum $(C(\bS^3_\theta)_0, L^2(\bS^3_\theta,\cS)_0, D_0)$ constitutes a $2^+$-summable spectral triple.
\end{prop}

\begin{proof}
The only non-trivial condition to check concerns the summability. Since the Dirac operator on $\bS^3_\theta$ coincides with that of the classical sphere $\bS^3$ acting on the Hilbert space $L^2(\bS^3,\cS)$, we may as well carry out the computation there. 

Recall \cite{Hom00} that the eigenvalues of the round Dirac operator on the three sphere are labeled by integers as 
$$
\lambda_{\pm k} = \pm (k+3/2), \qquad (k \geq 0),
$$
with multiplicities $(k+1)(k+2)$. Moreover, the $\lambda_{\pm k}$-eigenspaces are precisely the highest weight representations of $\Spin(4)=\SU(2) \times \SU(2)$ with highest weight $(k+1,k)$ and $(k,k+1)$ for the positive and negative eigenvalues, respectively. 

As such, there is a unique eigenvector in each eigenspace that is invariant under the action of one of the two copies of $\textup{U}(1)$ in $\Spin(4)$. It follows that, in passing from $L^2(\bS^3_\theta,\cS)$ to $L^2(\bS^3_\theta,\cS)_0$, one simply removes part of the degeneracies of the eigenvalues, so we have
\begin{equation}
\label{dec-S3-L2-0}
L^2(\bS^3_\theta,\cS)_0 \simeq { \bigoplus_{n \geq 0} V_{2n+1} \oplus V_{2n+1} }.
\end{equation}
where $V_{2n+1}$ denotes the $\textup{U}(1)$-invariant part of the highest weight $(2n+1)$-representation space for the relevant copy of $\SU(2)$ in $\Spin(4)$. 

In this way, each of the spaces $V_{2n+1}$ is an eigenspace for $D_0$. Since $D_0$ has eigenvalues $2n+5/2$ (in the first summand) and $-(2n+3/2)$ (in the second summand), each with multiplicity $2n+2$, we conclude that $D_0$ is $2^+$-summable.
\end{proof}

As a consequence we see that the base manifold for the Hopf fibration is two-dimensional. Let us make the geometric structure of this manifold more explicit by relating the operator $D_0$ on $L^2(\bS^3_\theta,\cS)_0$ to the Dirac operator $\dirac$ on the round two-sphere $\bS^2$. 

\begin{thm}
\label{thm:equiv-S30-S2}
Under the isomorphism $C(\bS^3_\theta)_0 \simeq C(\bS^2)$, there is a unitary equivalence between the spectral triples $(C(\bS^3_\theta)_0, L^2(\bS^3_\theta,\cS)_0, D_0-\half)$ and
$$
\left( C(\bS^2), L^2(\bS^2,\cS),2 \dirac \right).
$$
In particular, there is a completely bounded isomorphism $\mathcal{L}_{0}\cong\Lip(\bS^{2})$. 
\end{thm}

\proof Recall that the spectrum of the Dirac operator $\dirac$ on the round two-sphere $\bS^2$ is $\Z-\{ 0 \}$ with multiplicities $2|\ell|$ for each $\ell \in \Z-\{ 0 \}$. The corresponding eigenfunctions in $L^2(\bS^2,\cS)$ are the well known harmonic spinors on $\bS^2$: 
$$
\dirac Y_{jm}^\pm = \pm (j+\half) Y_{jm}^\pm, \qquad (j \in \bN + \half,~ m= -j,-j+1, \ldots, j-1,j). 
$$
For each fixed half-integer $j$, the functions $Y_{jm}^\pm$ for $m=-j,-j+1,\ldots,j-1,j$ span the highest weight representation space $V_{2j}$ for $\SU(2)$. Then, upon writing $j=n+\half$ we can identify this representation space (for $\pm$) with the spaces $V_{2n+1}$ in the above decomposition \eqref{dec-S3-L2-0} of $L^2(\bS^3,\cS)_0$. Identifying the Hilbert spaces $L^2(\bS^3,\cS)$ and $L^2(\bS^3_\theta,\cS)$ yields the result.\endproof

\begin{rem}\textup{
Observe that the canonical spectral triple and, in particular, the Dirac operator $\dirac$ on $\bS^2$ are written in terms of a sphere of radius one. On the other hand, we have just seen that the torus-invariant part of the spectral triple on $\bS^3_\theta$ has a Dirac operator which is equivalent to $2\dirac+\half$. This is indeed consistent with the aforementioned fact that the base space of the Hopf fibration is in fact a two-sphere of radius $1/2$. Since the constant operator $1/2$ is not odd for the natural grading, the invariant part is not an even spectral triple in the strict sense and we have to consider $D_{0}-\half$. Later on, this will be important to obtain the correct commutation relation with the vertical KK-cycle. 
}
\end{rem}

It now makes sense for us to write $(C(\bS^2),\H_0,D_0-\half)$ unambiguously for the spectral triple on the base space $\bS^2$ of the noncommutative Hopf fibration.

\begin{cor}
The datum $(\mathcal{H}_0,D_0-\half)$  constitutes an element of the set of unbounded even KK-cycles $\Psi_0(\Lip(\bS^2),\C)$.
\end{cor}

\proof The spinor bundle $\cS$ on the two-sphere is trivializable and of rank two, giving an obvious grading $\Gamma:\H_0\to\H_0$ of the Hilbert space $\H_0=L^2(\bS^2,\cS)$. The result is now immediate from the fact that $\left( C(\bS^2), L^2(\bS^2,\cS),2 \dirac \right)$ is a spectral triple which is even with respect to this grading and has Lipschitz algebra cb-isomorphic to $\Lip(\bS^{2})$.
\endproof

\subsection{The Lipschitz module of the Hopf fibration} The previous section described the horizontal part of the geometry of the noncommutative Hopf fibration. Next we come to describe its vertical geometry. To this end we use the completely positive map
\[\tau_{0}:\mathbb{B}(\mathcal{H})\rightarrow \mathbb{B}(\mathcal{H}),\qquad
\tau_0(a):= \int_{\bT}\alpha_{t}(a)dt, \]
where $\alpha_{t}$ is defined as in eq. \eqref{eq-tr}. We saw in Proposition \ref{eq triple} that this restricts to the $C^{*}$-algebra valued map
\begin{equation}\label{cpm}
\tau_0:C(\bS^3_\theta)\to C(\bS^2),\qquad \tau_0(a):=\int_{\bT} \alpha_{t} (a) dt.
\end{equation}
However, yet more is true:
just as it did for the noncommutative torus, the latter induces a right Hermitian structure $\langle\cdot,\cdot\rangle$ defined by $\langle a, b \rangle := \tau_0 (a^* b)$
on $\Lip(\bS^3_\theta)$ with values in the invariant subalgebra $\Lip(\bS^2)$. 

The corresponding $C^{*}$-norm is 
$$
\| a \|_{\mathpzc{E}_\theta} := \| \langle a, a \rangle \|_{C(\bS^3_\theta)}^{1/2},
$$
and we wish to identify the appropriate Lipschitz submodule $\mathcal{E}_{\theta}$ of $\mathpzc{E}_{\theta}$. 
Writing $(\cdot)^{\textup{tr}}$ to denote ordinary matrix transpose, we introduce for each $n\geq 1$ the partial isometries
\begin{align}
\Psi_{n}=(\Psi_{n,k})&:=
\begin{pmatrix}
a^{n} &  c_{1} a^{n-1} b & \cdots & c_{n-1} a b^{n-1} & b^{n}
\end{pmatrix}^{\textup{tr}},
\\
\Psi_{-n}=(\Psi_{-n,k})&:=
\begin{pmatrix}
a^{*n} &  c_{1} a^{*n-1} b^{*} & \cdots & c_{n-1} a^{*} b^{*n-1} & b^{*n} 
\end{pmatrix}^\textup{tr},
\end{align}
where $a,b$ are the generators of the $C^*$-algebra $C(\bS^3_\theta)$ and $c_{k}^2: = {n\choose k}$, so that each $\Psi_{n}$ is normalized in the sense that
\begin{equation}
\Psi^*_{n} \Psi_{n}= (|a|^2 + |b|^2 )^{|n|}= 1.
\end{equation}
Note that, with the convention $\Psi_{0}=1$, the matrices $\Psi_{n}$ are defined for all $n\in\Z$. We write $\Psi_{n,k}$ for the $k$th entry of the column vector $\Psi_n$, where $k=0,1,\ldots,|n|$. We will often employ the shorthand notation $[D,\Psi_n]$ to denote the column vector with entries $[D,\Psi_{n,k}]$, $k=0,1,\ldots,|n|$.

The following result describes the structure of the $\Lip(\bS^{2})$-submodules $\mathcal{L}_{n}$ appearing in the Peter-Weyl decomposition of $\Lip(\bS^{3}_{\theta})$. It will turn out to be a key step in describing the Lipschitz cycle for the Hopf fibration.

\begin{prop} \label{pr:projprop}
For each $n\in\Z$ the operator $p_{n}:=\Psi_{n}\Psi_{n}^{*}$ is a projection in $M_{|n|+1}(\Lip(\bS^{2}))$ and the map
\[ \mathcal{L}_{-n} \to \Lip(\bS^{2})^{|n|+1}, \qquad
x\mapsto (\Psi_{n}x),\]
implements a cb-isomorphism of finitely generated Lipschitz modules $\mathcal{L}_{-n} \simeq p_n\Lip(\bS^{2})^{|n|+1}$.
\end{prop}

\begin{proof} Since the generators $a$ and $b$ are in $\Lip(\bS^{3}_{\theta})$, it is immediate that $[D,\Psi_{n}]$ is bounded. For $x\in\L_{-n}$ we have
\[\pi_{D}(\Psi_{n})\pi_{D}(x)=\pi_{D}(\Psi_{n}x)=\pi_{D_{0}}(\Psi_{n}x),\]
since $(\Psi_{n}x)$ is a column vector consisting of elements of degree zero, whence \[[D,\Psi_{n}x]=[T+D_{0},\Psi_{n}x]=[D_{0},\Psi_{n}x].\] Hence we have
\[\pi_{D_{0}}(\Psi_{n}x)^{*}\pi_{D_{0}}(\Psi_{n}x)\leq \|\pi_{D}(\Psi_{n})\|^{2}\pi_{D}(x)^{*}\pi_{D}(x),\]
showing that $\Psi_{n}$ is completely bounded as a map $\mathcal{L}_{-n}\rightarrow p_{n}\Lip(\bS^{2})^{|n|+1}$. Its inverse $\Psi_{n}^{*}$ is completely bounded since, for each $v\in p_{n}\Lip(\bS^{2})^{|n|+1}$ we have
\[\pi_{D}(\Psi_{n}^{*}v)=\pi_{D}(\Psi_{n}^{*})\pi_{D}(v)=\pi_{D}(\Psi_{n}^{*})\pi_{D_{0}}(v),\]
which holds because $v$ is a column vector made up of elements of degree zero. Hence
\[\pi_{D}(\Psi_{n}^{*}v)^{*}\pi_{D}(\Psi_{n}^{*}v)\leq \|\pi_{D}(\Psi_{n})\|^{2}\pi_{D_{0}}(v)^{*}\pi_{D_{0}}(v),\]
implying complete boundedness.
\end{proof}

\begin{cor} There is a $C^{*}$-module isomorphism \[\mathpzc{E}_{\theta}\cong\bigoplus_{n\in\Z}p_{n}C(\bS^{2})^{|n|+1}\subset\mathcal{H}_{C(\bS^{2})}\] which, on the dense subspace $C(\bS^{3}_{\theta})\subset\mathpzc{E}_{\theta}$, is defined by $x\mapsto(\tau_{0}(\Psi_{n}x))_{n\in\Z}$.
\end{cor}

\begin{proof} This follows by taking $C^{*}$-completions of the $\mathcal{L}_{n}$ in the previous proposition and observing that $\mathpzc{E}_{\theta}$ is the $C^{*}$-module direct sum of these completions.
\end{proof}

We define the Lipschitz module $\mathcal{E}_{\theta}$ to be the direct sum of the Lipschitz modules $\mathcal{L}_{n}$ in the sense of Proposition~\ref{directsum}. By definition,  $\mathcal{E}_{\theta}$ is the completion of the dense subalgebra of finite sums of homogeneous elements $x\in\Lip(\bS^{3}_{\theta})$ in the norm
\begin{equation}\label{stabnorm}\|x\|_{\mathcal{E}_{\theta}}^{2}:=\|\sum_{n\in\Z} \pi_{D_{0}}(\tau(\Psi_{n}x))^{*}\pi_{D_{0}}(\tau(\Psi_{n}x))\|=\|\sum_{n\in\Z}\pi_{D_{0}}(\Psi_{n}x_{-n})^{*}\pi_{D_{0}}(\Psi_{n}x_{-n})\|\end{equation}
for $x\in \E_\theta$, where we denote by $x_{-n}\in\L_{-n}$ the component of $x$ of homogeneous degree $-n$. 
We will now analyze this norm in order to prove that multiplication in $\Lip(\bS^{3}_{\theta})$ induces a cb-homomorphism $\Lip(\bS^{3}_{\theta})\rightarrow \End^{*}_{\Lip(\bS^{2})}(\mathcal{E}_{\theta})$. 

\begin{lem}\label{derab} 
The derivatives of the generators $a,b$ of the $C^*$-algebra $C(\bS^3_\theta)$ with respect to the operator $D_{0}$ are
\[[D_{0},\pi_\theta(a)]=\begin{pmatrix} 0 & 0 \\ 2b^{*} & 0\end{pmatrix}\lambda^{H_{1}},\qquad [D_{0},\pi_\theta(b)]=\begin{pmatrix} 0 & 0 \\ -2a^{*} & 0\end{pmatrix}\lambda^{-H_{1}}.\] 
In particular they satisfy
\[\pi_\theta(a)[D_{0},\pi_\theta(a)]=[D_{0},\pi_\theta(a)]\pi_\theta(a),\qquad \pi_\theta(b)[D_{0},\pi_\theta(b)]=[D_{0},\pi_\theta(b)]\pi_\theta(b),\]
and the elements
\[\pi_\theta(b^{*})[D_{0},\pi_\theta(b)],\qquad \pi_\theta(a^{*})[D_{0},\pi_\theta(a)],\]
belong to the commutant of $\pi_\theta(C(\bS^{3}_{\theta}))$ in $\BB(\H_\theta)$.
\end{lem}

\begin{proof} For the generators of the classical algebra $C(\bS^3)$ one computes rather easily that
$$ 
[D_{0},\pi(a)]=\begin{pmatrix} 0 & 0 \\ 2b^{*} & 0\end{pmatrix}, \qquad [D_{0},\pi(b)]=\begin{pmatrix} 0 & 0 \\ -2a^{*} & 0\end{pmatrix}.
$$
The first claim follows directly from the fact that $\LL([D_{0},\pi(a)])=[D_{0},\LL(\pi(a))]$ and similarly so for the generator $b$. It follows immediately that the elements
\[\pi_\theta(b^{*})[D_{0},\pi_\theta(b)],\qquad \pi_\theta(a^{*})[D_{0},\pi_\theta(a)],\]
have bidegree $(0,0)$ and therefore commute with all of $\pi_\theta(C(\bS^{3}_{\theta})).$
\end{proof}

Since no confusion will arise, from now on we shall omit the subscript $\theta$ from the representation $\pi_\theta:C(\bS^3_\theta)\to \BB(\H_\theta)$. From the latter result it follows that, for each $k=0,1,\ldots,|n|$, the commutator $ [D_{0},\pi(\Psi_{n,k})] $ is a lower triangular matrix when $n$ is positive and upper triangular when $n$ is negative. For positive $n$ we have
\[ [D_{0},\pi(\Psi_{n,k})]=\begin{pmatrix} n \\ k\end{pmatrix}^{\frac{1}{2}}\left( (n-k)[D_{0},\pi(a)]\pi(a^{n-k-1}b^{k})+k\pi(a^{n-k}b^{k-1})[D_{0},\pi(b)]\right)\]
for each $k=0,1,\ldots,n$, with a similar formula for negative $n$. For positive $n$, each component of the matrix $[D,\pi(\Psi_{n})]^{*}[D,\pi(\Psi_{n})]$ is non-zero only in the upper left corner, whereas for negative $n$ it is non-zero only in the lower right corner. 

From now on we assume $n>1$, as the calculations are similar for $n<-1$. The cases $n=-1,0,1$ are trivial and can be done by hand. A straightforward calculation using the above lemma shows that the only non-zero entry of the $k$th component of $[D_{0},\pi(\Psi_{n})]^{*}[D_{0},\pi(\Psi_{n})]$ is equal to
\begin{align}\label{derivative}
4\sum_{k=0}^{n}\begin{pmatrix} n \\ k\end{pmatrix}&\left( (n-k)^{2}|a|^{2(n-k-1)}|b|^{2(k+1)}+\right.
\\
&\left.+k^{2}|a|^{2(n-k+1)}|b|^{2(k-1)}-2k(n-k)|a|^{2(n-k)}|b|^{2k}\right) \nonumber 
\\
=4\sum_{k=0}^{n}\begin{pmatrix} n \\ k\end{pmatrix}&\left( \left((n-k)+((n-k)^{2}-(n-k))\right)|a|^{2(n-k-1)}|b|^{2(k+1)}+\right. \nonumber
\\
&\left.+\left(k+(k^{2}-k)\right)|a|^{2(n-k+1)}|b|^{2(k-1)}-2k(n-k)|a|^{2(n-k)}|b|^{2k}\right). \nonumber 
\end{align}

\begin{lem}\label{binom} For $n,k\geq 1$, the binomial coefficients $ \begin{pmatrix} n \\ k\end{pmatrix}$ satisfy the identities
$$
k\begin{pmatrix} n \\ k\end{pmatrix}=n\begin{pmatrix} n-1 \\ k-1\end{pmatrix}, \qquad
(n-k)\begin{pmatrix} n \\ k\end{pmatrix}=n\begin{pmatrix} n-1 \\ k\end{pmatrix}.$$
\end{lem}
 
\begin{proof}These are verified by direct computation. For the first identity one uses the fact that $\begin{pmatrix} n \\ k\end{pmatrix}=\frac{n}{k}\begin{pmatrix} n-1 \\ k-1\end{pmatrix}$.
For the second claim one combines this with the fact that $\begin{pmatrix} n \\ k\end{pmatrix}=\begin{pmatrix} n \\ n-k\end{pmatrix}$.
\end{proof}

\begin{cor}\label{co:comb}For $n\geq 2$, we have
\begin{enumerate}[(i)]\item for $2\leq k\leq n$ : $k(k-1)\begin{pmatrix} n \\ k\end{pmatrix}=n(n-1)\begin{pmatrix} n-2 \\ k-2\end{pmatrix};$
\item for $1\leq k\leq n-1$ :  $k(n-k)\begin{pmatrix} n \\ k\end{pmatrix}=n(n-1)\begin{pmatrix} n-2 \\ k-1\end{pmatrix};$
\item for $0\leq k \leq n-2$ : $(n-k)(n-k-1)\begin{pmatrix} n \\ k\end{pmatrix}=n(n-1)\begin{pmatrix} n-2 \\ k\end{pmatrix}$.
\end{enumerate}
\end{cor}
\begin{proof}These identities follow from applying the previous lemma twice.
\end{proof}

We will now split the expression \eqref{derivative} into four parts:
\begin{equation}\label{split1a} 
4(n|a|^{2}|b|^{2(n-1)}+n|a|^{2(n-1)}|b|^{2});
\end{equation}
\begin{equation}\label{split1b}
4\sum_{k=1}^{n-1}n\begin{pmatrix} n-1 \\ k\end{pmatrix}|a|^{2(n-k-1)}|b|^{2(k+1)}+n\begin{pmatrix} n-1 \\ k-1\end{pmatrix}|a|^{2(n-k+1)}|b|^{2(k-1)};
\end{equation}
\begin{equation}\label{split2}
4\sum_{k=0}^{n}\begin{pmatrix} n \\ k\end{pmatrix}\left(((n-k)^{2}-(n-k))|a|^{2(n-k-1)}|b|^{2(k+1)}+(k^{2}-k)|a|^{2(n-k+1)}|b|^{2(k-1)}\right);
\end{equation}
\begin{equation}\label{split3}
4\sum_{k=1}^{n-1}-2k(n-k)\begin{pmatrix} n \\ k\end{pmatrix}|a|^{2(n-k)}|b|^{2k}.\end{equation}
By applying Lemma~\ref{binom} to the expression \eqref{split1b}, it is straightforward to check that \eqref{split1a} and \eqref{split1b} precisely cancel some of the terms in eq.~\eqref{split2} and so together these four terms add up to give the expression \eqref{derivative}. We claim that equations \eqref{split1a} and \eqref{split1b} add up to $4n$, whereas \eqref{split2} and \eqref{split3} add up to zero.

\begin{lem} \label{n}Equations \eqref{split1a} and \eqref{split1b} add up to $4n$.
\end{lem}

\begin{proof} Omitting the constant $4$, we observe that
\[\begin{split}n|a|^{2}|b|^{2(n-1)}+\sum_{k=1}^{n-1}n\begin{pmatrix} n-1 \\ k-1\end{pmatrix}|a|^{2(n-k+1)}|b|^{2(k-1)}&=n|a|^{2}\sum_{k=1}^{n}\begin{pmatrix} n-1 \\ k-1\end{pmatrix}|a|^{2(n-k)}|b|^{2(k-1)};\\
n|a|^{2(n-1)}|b|^{2}+\sum_{k=1}^{n-1}n\begin{pmatrix} n-1 \\ k\end{pmatrix}|a|^{2(n-k-1)}|b|^{2(k+1)}&=n|b|^{2}\sum_{k=1}^{n}\begin{pmatrix} n-1 \\ k-1\end{pmatrix}|a|^{2(n-k)}|b|^{2(k-1)}.
\end{split}\]
Adding these equations together yields
\[n(|a|^{2}+|b|^{2})\sum_{k=1}^{n}\begin{pmatrix} n-1 \\ k-1\end{pmatrix}|a|^{2(n-k)}|b|^{2(k-1)}=n(|a|^{2}+|b|^{2})^{n}=n,\]
as was required.\end{proof} 

\begin{lem} \label{Rn}
Equations \eqref{split2} and \eqref{split3} add up to zero. 
\end{lem}

\begin{proof} Again omitting the constant $4$, eq.~\eqref{split3} equals 
\[\sum_{k=1}^{n-1}-2k(n-k)\begin{pmatrix} n \\ k\end{pmatrix}|a|^{2(n-k)}|b|^{2k}=\sum_{k=1}^{n-1}-2n(n-1)\begin{pmatrix} n-2 \\ k-1\end{pmatrix}|a|^{2(n-k)}|b|^{2k},\]
whereas \eqref{split2} may be rewritten using Corollary~\ref{co:comb} as 
\begin{equation}\label{cancel} n(n-1)\left(\sum_{k=0}^{n-2}\begin{pmatrix} n-2 \\ k\end{pmatrix}|a|^{2(n-k-1)}|b|^{2(k+1)}+\sum_{k=2}^{n}\begin{pmatrix} n-2 \\ k-2\end{pmatrix}|a|^{2(n-k+1)}|b|^{2(k-1)}\right).\end{equation}
Now 
\[\sum_{k=2}^{n}\begin{pmatrix} n-2 \\ k-2\end{pmatrix}|a|^{2(n-k+1)}|b|^{2(k-1)}=\sum_{k=1}^{n-1}\begin{pmatrix} n-2 \\ k-1\end{pmatrix}|a|^{2(n-k)}|b|^{2k},\]
and
\[\sum_{k=0}^{n-2}\begin{pmatrix} n-2 \\ k\end{pmatrix}|a|^{2(n-k-1)}|b|^{2(k+1)}=\sum_{k=1}^{n-1}\begin{pmatrix} n-2 \\ k-1\end{pmatrix}|a|^{2(n-k)}|b|^{2k},\]
so \eqref{cancel} and \eqref{split3} cancel one another.\end{proof}

\begin{prop}\label{diffes} We have the following operator identities:
\begin{align*} 
[D_{0},\pi(\Psi_{n})]^{*}[D_{0},\pi(\Psi_{n})]&=\begin{pmatrix}4n &0 \\ 0 & 0\end{pmatrix} \qquad (n>0); \\
[D_{0},\pi(\Psi_{n})]^{*}[D_{0},\pi(\Psi_{n})]&=\begin{pmatrix}0 &0 \\ 0 & 4|n|\end{pmatrix} \qquad (n<0).
\end{align*}
\end{prop}

\begin{proof} For $n=1$ one checks directly that $[D_{0},\pi(\Psi_{n})]^{*}[D_{0},\pi(\Psi_{n})]=\textup{diag}(4,0)$. 
For $n\geq 2$ the required equality follows from eq.~\eqref{derivative} as a direct consequence of Lemmata~\ref{n} and \ref{Rn}. The case where $n<0$ is similar. 
\end{proof}

\begin{lem} \label{miracle}
 We have $\pi(\Psi_{n}^{*})[D_{0},\pi(\Psi_{n})]=0$.
\end{lem}

\begin{proof}From Lemma~\ref{derab} we already know that the matrix $\pi(\Psi_{n}^{*})[D_{0},\pi(\Psi_{n})]$ is lower triangular. Using Lemmata~\ref{binom} and \ref{derab} we directly compute the lower left entry of this matrix to be
\[\begin{split}
\left(\pi(\Psi_{n}^{*})[D_{0},\pi(\Psi_{n})]\right)_{21} &=2a^{*}b^{*}\sum_{k=0}^{n}\begin{pmatrix}n \\ k\end{pmatrix}k|a|^{2(k-1)}|b|^{2(n-k)}a^{*}b^{*}-(n-k)|a|^{2k}|b|^{2(n-k-1)} \\
&=2a^{*}b^{*}\sum_{k=1}^{n}k\begin{pmatrix}n \\ k\end{pmatrix}|a|^{2(k-1)}|b|^{2(n-k)} \\
& \qquad\qquad\qquad\qquad\qquad-\sum_{k=0}^{n-1}(n-k)\begin{pmatrix}n \\ k\end{pmatrix}|a|^{2k}|b|^{2(n-k-1)}\\
&=2a^{*}b^{*}\sum_{k=0}^{n-1}n\begin{pmatrix}n-1 \\ k\end{pmatrix}|a|^{2k}|b|^{2(n-1-k)} \\
& \qquad\qquad\qquad\qquad\qquad -\sum_{k=0}^{n-1}n\begin{pmatrix}n-1 \\ k\end{pmatrix}|a|^{2k}|b|^{2(n-k-1)}\\
&=2n(|a|^{2}+|b|^{2})^{n-1}a^{*}b^{*}-2n(|a|^{2}+|b|^{2})^{n-1}a^{*}b^{*}=0,
\end{split}\]
which is the result we were looking for.\end{proof}

\begin{lem}\label{lem:contractive2} The maps $\Lip(\bS^{3}_{\theta})\rightarrow\M_2(\mathbb{B}(\mathcal{H}_{\theta}))$ defined by 
$$
\pi_{D}(a)\mapsto \pi_{D_{0}}(a),\qquad \pi_{D}(a)\mapsto \pi_{T}(a),
$$ 
are completely bounded.
\end{lem}

\begin{proof} This is proved in the same spirit as Lemma~\ref{lem:contractive}. Consider the unitary operator $u$ interchanging the two copies of $L^{2}(\bS^{3}_{\theta})$ in $\mathcal{H}_{\theta}$, together with the unitary $v$ from eq.~\eqref{v}. We have the identities
\[uDu^{*}= \begin{pmatrix} -iZ_1 & -Z_2+iZ_3 \\ Z_2+i Z_3 & 
iZ_1\end{pmatrix} +\frac{3}{2}, \qquad vDv^{*}=\begin{pmatrix} -iZ_1 & Z_2-iZ_3 \\ -Z_2-i Z_3 & iZ_1\end{pmatrix} +\frac{3}{2},\]
and clearly $u\pi(a)u^{*}=v\pi(a)v^{*}=\pi(a)$. Therefore
\[2\pi_{D}(a)+u\pi_{D}(a)u^{*}+v\pi_{D}(a)v^{*}=4\pi_{\frac{1}{2}D_{0}}(a).\]
Now since
\[\pi_{D_{0}}(a)=g\pi_{\frac{1}{2}D_{0}}(a)g^{-1},\quad\textnormal{with}\quad g=\begin{pmatrix} 1 & 0 \\ 0 & 2\end{pmatrix},\]
we can write
\[\|\pi_{D_{0}}(a)\|=\|g\pi_{\frac{1}{2}D_{0}}(a)g^{-1}\|\leq \|g\|\|g^{-1}\|\|\pi_{D}(a)\|=2\|\pi_{D}(a)\|,\]
a fact which clearly extends to matrices. Using the same unitaries, we have
\[uvDv^{*}u^{*}=\begin{pmatrix} iZ_{1} & -Z_{2}-iZ_{3} \\ Z_{2}-iZ_{3} & -iZ_{1}\end{pmatrix} + \frac{3}{2},\]
from which it follows that
\[\pi_{D}(a)+uv\pi_{D}(a)v^{*}u^{*}=2\pi_{T}(a),\]
so that $\|\pi_{T}(a)\|\leq\|\pi_{D}(a)\|$ as desired.\end{proof}

\begin{prop} \label{prop:id} For any homogeneous element $x\in \Lip(\bS^{3}_{\theta})_{-n}$ we have
\[\pi_{D_{0}}(\Psi_{n}x)^{*}\pi_{D_{0}}(\Psi_{n}x)=\begin{pmatrix}(4|n|+1)x^{*}x+[D_{0},x]^{*}[D_{0},x] & [D_{0},x]^{*}x \\
x^{*}[D_{0},x] & x^{*}x\end{pmatrix}.
\]
\end{prop}

\begin{proof} From Lemma \ref{lem:contractive2} it follows that $[D_{0},x]$ is bounded whenever $[D,x]$ is bounded. By Lemma \ref{miracle} we have
\[\begin{split}[D_{0},\Psi_{n}x]^{*}[D_{0},\Psi_{n}x] &=x^{*}[D_{0},\Psi_{n}][D_{0},\Psi_{n}]x+x^{*}[D_{0},\Psi_{n}]^{*}\Psi_{n}[D_{0},x] \\ &\qquad\qquad +[D_{0},x]\Psi_{n}^{*}[D_{0},\Psi_{n}]x+[D_{0},x]^{*}[D_{0},x] \\
&=x^{*}[D_{0},\Psi_{n}][D_{0},\Psi_{n}]x+ [D_{0},x]^{*}[D_{0},x]\end{split}\]
and similarly that
\[x^{*}\Psi_{n}^{*}[D_{0},\Psi_{n}x]= x^{*}\Psi_{n}^{*}[D_{0},\Psi_{n}]x+x^{*}\Psi_{n}^{*}\Psi_{n}[D_{0},x]=x^{*}[D_{0},x].\]
Therefore
\[ \pi_{D_{0}}(\Psi_{n}x)^{*}\pi_{D_{0}}(\Psi_{n}x)=\begin{pmatrix}x^{*}x+x^{*}[D_{0},\Psi_{n}]^{*}[D_{0},\Psi_{n}]x+[D_{0},x]^{*}[D_{0},x] & [D_{0},x]^{*}x \\
x^{*}[D_{0},x]& x^{*}x\end{pmatrix},\]
from which the desired equality follows by using Proposition~\ref{diffes}.
\end{proof}

\begin{cor}\label{eqnorm} The norm \eqref{stabnorm} on the Lipschitz module $\E_\theta$ is cb-isometric to the norm
\begin{equation}\label{newnorm}
\|x\|^{2}=\|\tau_0(\pi_{D_{0}}(x)^{*}\pi_{D_{0}}(x)) + \sum_{n\in\Z} 4|n|x_{n}^{*}x_{n}\|,
\end{equation}
where $\tau_0:\B(\mathcal{H}_{\theta}\oplus \mathcal{H}_{\theta})\to \B(\mathcal{H}_{\theta}\oplus \mathcal{H}_{\theta})$ is the map \eqref{cpm}. Consequently, $\mathcal{E}_{\theta}$ is the completion of $\Lip(\bS^3_\theta)$ in this norm.
\end{cor}

\proof This is now an easy computation using Proposition~\ref{prop:id}, extended to linear combinations of homogeneous elements in $\E_\theta$. To see that this norm is defined on all of $\Lip(\bS^{3}_{\theta})$, write 
\[p=\begin{pmatrix}1 & 0 \\ 0 & 0\end{pmatrix},\]
and estimate
\[\sum_{n\in\Z}4|n|x_{n}^{*}x_{n}\leq \sum_{n\in\Z}4n^{2}x_{n}^{*}x_{n}\leq\tau_{0}(p\pi_{T}(2x)^{*}\pi_{T}(2x)p)\leq \|\pi_{T}(2x)\|^{2}\tau_{0}(p)=4 \|\pi_{T}(x)\|^{2}.\]
Then apply Lemma \ref{lem:contractive2} to obtain that
\[\|x\|^{2}=\|\tau_0(\pi_{D_{0}}(x)^{*}\pi_{D_{0}}(x)) + \sum_{n\in\Z} 4|n|x_{n}^{*}x_{n}\|\leq 8\|\pi_{D}(x)\|^{2},\]
so $\mathcal{E}_{\theta}$ is a completion of $\Lip(\bS^{3}_{\theta})$. \endproof

\subsection{The Hopf fibration as a Lipschitz cycle}
The work of the previous two sections now places us in a position to describe the Hopf fibration as defining an odd Lipschitz cycle in $\Psi^{\ell}_{-1}(\Lip(\bS^{3}_{\theta}),\Lip(\bS^{2}))$. First, we describe $\mathcal{E}_{\theta}$ as a left $\Lip(\bS^{3}_{\theta})$-module. 

We first need a lemma about circle actions on $C^{*}$-algebras. Let $A$ be a $C^{*}$-algebra with a circle action, $F\subset A$ its fixed point algebra and $\tau_{0}:A\rightarrow F$ the associated conditional expectation. The $C^{*}$-module $\mathpzc{E}$ is defined by completing $A$ in the norm associated to the $F$-valued inner product $\langle a,b\rangle:=\tau_{0}(a^{*}b)$. Multiplication induces a $*$-homomorphism $A\rightarrow \End^{*}_{F}(\mathpzc{E})$. The infinitesimal generator $T$ of the circle action is then a self-adjoint regular operator with spectrum $\Z\subset \R$ and the commutators $[T,a]$ are bounded for all $a$ in the dense subalgebra generated by homogeneous elements.

\begin{lem} Let $a\in A$ be such that $[T,a]\in \End^{*}_{F}(\mathpzc{E})$ and let $0<\alpha<1$. There is a positive constant $C_{\alpha}$ independent of $a$ such that
\[ \|[|T|^{\alpha},a]\|_{cb}\leq C_{\alpha}\|[T,a]\|_{cb}.\] 
\end{lem}

\begin{proof} The proof relies on the fact that $T$ has discrete spectrum, so the function $x\mapsto |x|^{\alpha}$ can be smoothened around $0$ to a function $g$, without affecting the resulting operator $g(T)=|T|^{\alpha}$. The presence of the spectral projections $p_{n}:\mathpzc{E}\rightarrow A_{n}$ allows one to proceed as in \cite[Lem.~10.15, 10.17]{GVF01} to show that
\[\|[ |T|^{\alpha},a]\|\leq \frac{1}{2\pi}\int_{\R}|t\hat{g}(t)|dt\|[T,a]\|,\]
where $\hat g$ denotes the Fourier transform of $g$. By applying this reasoning to the finite direct sums $\mathpzc{E}^{\oplus n}$, one obtains the same bound for matrices $a_{ij}$ for which $[T,a_{ij}]$ is bounded. Since $g$ is smooth, we have $t\hat{g}(t)=\widehat{g'}(t)$, so it remains to show that $\widehat{g'}$ is integrable, which is done as in \cite[Lem.~10.17]{GVF01}.\end{proof}

\begin{cor}\label{cor:half} For any $0<\alpha <1$ the map $\pi_{T}(a)\mapsto \pi_{|T|^{\alpha}}(a)$ is completely bounded.
\end{cor}
\begin{proof} This now follows immediately by estimating
\[\|\pi_{|T|^{\alpha}}(a)\|\leq \|a\|+\|[|T|^{\alpha},a]\|\leq \|a\|+C_{\alpha}\|[T,a]\|\leq (C_{\alpha}+1)\|\pi_{T}(a)\|\]
for all $a\in A$ such that $[T,a]\in \End^{*}_{F}(\mathpzc{E})$.\end{proof}

\begin{thm} 
The norm \eqref{stabnorm} on $\E_\theta$ is equivalent to the norm given by
\[\|x\|^{2}_{\tau_0,\|T|^{\frac{1}{2}}}=\|\tau_0\left(\pi_{D_{0}}(x)^{*}\pi_{D_{0}}(x) \right)+\langle\begin{pmatrix} x \\ 2|T|^{\frac{1}{2}}x\end{pmatrix},\begin{pmatrix} x \\ 2|T|^{\frac{1}{2}}x\end{pmatrix}\rangle\|.\]
Consequently, the module $\mathcal{E}_{\theta}$ is cb-isomorphic to the completion of $\Lip(\bS^{3}_{\theta})$ in this norm. Multiplication in $\Lip(\bS^{3}_{\theta})$ induces a completely bounded $*$-homomorphism $\Lip(\bS^{3}_{\theta})\rightarrow\End^{*}_{\Lip(\bS^{2})}(\mathcal{E}_{\theta})$.
\end{thm}

\begin{proof} In Corollary~\ref{eqnorm} we showed that the norm \eqref{stabnorm} is cb-isometric to the norm \eqref{newnorm},
and that this norm is defined on all of $\Lip(\bS^{3}_{\theta})$.
By definition of $|T|^{\frac{1}{2}}$ we see that
\[\langle\begin{pmatrix} x \\ 2|T|^{\frac{1}{2}}x\end{pmatrix},\begin{pmatrix} x \\ 2|T|^{\frac{1}{2}}x\end{pmatrix}\rangle=\sum_{n\in\Z}(4|n|+1)x_{n}^{*}x_{n},\]
So the above norm differs from $\eqref{newnorm}$ by a term
\[\tau_{0}(x^{*}x)=\sum_{n\in\Z}x_{n}^{*}x_{n} ,\]
and since
\[\begin{pmatrix} x^{*}x & 0 \\ 0 & 0\end{pmatrix}\leq \pi_{D_{0}}(x)^{*}\pi_{D_{0}}(x)\]
we have the estimate
\[\|x\|^{2}\leq \|x\|^{2}_{\tau_{0},|T|^{\frac{1}{2}}}\leq 2\|x\|^{2}.\]
Furthermore,
because $\pi_{|T|^{\frac{1}{2}}}$ is a homomorphism, for each $b\in \Lip(\bS^3_\theta)$ and $x\in\E_\theta$ the estimate
\[\begin{split}\langle\begin{pmatrix} bx \\ |T|^{\frac{1}{2}}bx\end{pmatrix},\begin{pmatrix} bx \\ |T|^{\frac{1}{2}}bx\end{pmatrix}\rangle 
&=\langle\pi_{|T|^{\frac{1}{2}}}(b)\begin{pmatrix} x \\ |T|^{\frac{1}{2}}x\end{pmatrix},\pi_{|T|^{\frac{1}{2}}}(b)\begin{pmatrix} x \\ |T|^{\frac{1}{2}}x\end{pmatrix}\rangle\\
&\leq\|\pi_{|T|^{\frac{1}{2}}}(b)\|^{2}\langle\begin{pmatrix} x \\ |T|^{\frac{1}{2}}x\end{pmatrix},\begin{pmatrix} x \\ |T|^{\frac{1}{2}}x\end{pmatrix}\rangle \\
&\leq (C_{\frac{1}{2}}+1)\|\pi_{T}(b)\|^{2}\langle\begin{pmatrix} x \\ |T|^{\frac{1}{2}}x\end{pmatrix},\begin{pmatrix} x \\ |T|^{\frac{1}{2}}x\end{pmatrix}\rangle \\
&\leq( C_{\frac{1}{2}}+1)\|\pi_{D}(b)\|^{2}\langle\begin{pmatrix} x \\ |T|^{\frac{1}{2}}x\end{pmatrix},\begin{pmatrix} x \\ |T|^{\frac{1}{2}}x\end{pmatrix}\rangle,
\end{split}\]
holds by Lemma \ref{lem:contractive2} and Corollary \ref{cor:half}. Since $\pi_{D_{0}}$ is also a homomorphism, the estimate
\[
\pi_{D_{0}}(bx)^{*}\pi_{D_{0}}(bx)\leq \|\pi_{D_{0}}(b)\|^{2}\pi_{D_{0}}(x)^{*}\pi_{D_{0}}(x)
\]
is immediate. For the norm $\|\cdot\|_{\tau_0,|T|^{\frac{1}{2}}},$ we can now estimate that

\[\|bx\|_{\tau_0,|T|^{\frac{1}{2}} }^{2}\leq (3+C_{\frac{1}{2}})\|\pi_{D}(b)\|^{2}\|\tau_0\left(\pi_{D_{0}}(x)^{*}\pi_{D_{0}}(x) \right)+\langle\begin{pmatrix} x \\ |T|^{\frac{1}{2}}x\end{pmatrix},\begin{pmatrix} x \\ |T|^{\frac{1}{2}}x\end{pmatrix}\rangle \| ,\] 
which proves that multiplication induces a cb-homomorphism.
\end{proof}

\begin{prop} The map $T:\Dom (T)\rightarrow\mathcal{E}_{\theta}$ defined on homogeneous elements $x_n\in\mathcal{L}_n$ by 
$$
T:\Dom (T)\rightarrow\mathcal{E}_{\theta},\qquad x_{n}\mapsto nx_{n},
$$ 
is a self-adjoint regular linear operator with compact resolvent on the Lipschitz module $\mathcal{E}_{\theta}$. The map $a\mapsto [T,a]\in\End^{*}_{C(\bS^{2})}(\mathpzc{E}_{\theta})$ is a cb-derivation $\Lip(\bS^{3}_{\theta})\rightarrow \End^{*}_{C(\bS^{2})}(\mathpzc{E}_{\theta})$.
\end{prop}

\begin{proof} It is immediate that the operators $T\pm i$ have dense range, since they map the algebraic direct sum $\bigoplus_{n\in\Z}\mathcal{L}_{n}$ onto itself. Using the norm \eqref{stabnorm}, we see that the operators $(T\pm i)^{-1}$ are contractive for this norm. By Lemma~\ref{pmi}, the closure of $T$ is self-adjoint and regular in $\mathcal{E}_{\theta}$. The resolvents $(T\pm i)^{-1}$ are elements of $\K_{\Lip(\bS^{2})}(\mathcal{E}_{\theta})$, since they are the uniform limit of the operators $$r_{k}:=\sum_{n=-k}^{k}(n\pm i)^{-1}p_{n}$$ with $p_{n}$ the projections $x\mapsto\Psi_{n}\tau_0(\Psi_{n}^{*}x)$. Indeed, these are contractive as operators in $\mathcal{E}_{\theta}$ (as can be seen directly from the norm \eqref{stabnorm}) and thus for $m> k$ we find
\[\|r_{k}-r_{m}\|_{\mathcal{E}_{\theta}}=\|\sum_{k<|n|\leq m}(n\pm i)^{-1}p_{n}\|_{\mathcal{E}_{\theta}}\leq\frac{1}{\sqrt{1+k^{2}}}\rightarrow 0\]
as $k\rightarrow\infty$. 
The statement that the map $a\mapsto [T,a]$ is a cb-derivation from $\Lip(\bS^{3}_{\theta})$ into the $C^{*}$-algebra $\End^{*}_{C(\bS^{2})}(\mathpzc{E})$ follows directly from Lemma~\ref{lem:contractive2} and the fact that $C^{*}$-algebra representations are completely contractive.
\end{proof}
As with the noncommutative torus, the $C^{*}$-module version of this proposition is an example of a circle module as described in \cite{CNNR08}. The Lipschitz structure we have constructed allows us to study connections on this circle module.

Let us denote by $\nabla_n:\L_n\to\L_n\otimes_{\Lip(\bS^2)}\Omega^1_{D_0}(\bS^2)$ the canonical Grassmann connection on each of the projective modules $\L_n$ defined by $\n_n:=p_n\circ \d$, where $p_n=\Psi_n\Psi_n^*$ is the projection which defines the modules $\L_n$ for each $n\in\ZZ$.

\begin{prop}\label{KK-conn}
Under the isomorphism $\E_\theta\cong \bigoplus_{n\in\ZZ}\mathcal{L}_n$, the linear map
$$
\n:{\mathcal{E}_\theta}\to{\mathcal{E}_\theta}\,\hotimes_{\Lip(\bS^2)}\Omega^1_{D_0}(C(\bS^2),\Lip(\bS^{2}))\qquad \nabla:=\oplus_{n\in\ZZ} \nabla_n,
$$
yields a well defined connection on the right Lipschitz $\Lip(\bS^2)$-module ${\mathcal{E}_\theta}$. It has the property that $[\nabla,T]=0$.
\end{prop}

\proof It is clear that the maps $\n_n:\L_n\to \L_n\otimes_{\Lip(\bS^2)}\Omega^1_{D_0}(C(\bS^2),\Lip(\bS^2))$ make sense as connections on the line bundles $\L_n$ for each $n\in\ZZ$. Each of these connections is completely contractive for the Lipschitz topology on $\L_n$ and so the algebraic direct sum $\oplus_n \n_n$ extends to a well defined completely bounded map $\n$ on the Lipschitz direct sum as claimed. The property that $[\nabla,T]=0$ is immediate, since it evidently holds on each $\mathcal{L}_{n}$ and the algebraic direct sum is a core for the operator $T$. 
\endproof

Recall that all along it was our goal to identify a Lipschitz unbounded KK-cycle which captures the vertical part of the geometry of the noncommutative Hopf fibration. Following the accomplishments to this point, finally we arrive at the desired theorem.

\begin{thm} The datum $(\mathcal{E}_{\theta},T,\nabla)$ defines a Lipschitz cycle in $\Psi_{-1}^\ell(\Lip(\bS^{3}_{\theta}),\Lip(\bS^{2}))$.
\end{thm}

\begin{proof} The above discussion shows that all of the conditions prescribed in Definition~\ref{de:lipcyc} are indeed satisfied.
\end{proof}


\subsection{The noncommutative three-sphere as a fibration over the two-sphere} Just as we did for the noncommutative two-torus, we are now ready to spell out the factorization of the spin geometry of the noncommutative sphere $\bS^3_\theta$ over a classical base space. We shall present this factorization as a product in unbounded KK-theory of the canonical spectral triple over the two-sphere $\bS^2$ with the Lipschitz cycle $(\E_\theta,T,\n)$ constructed in the previous section. That is to say, we claim that 
\begin{equation}\label{fac-claim}
(\H_\theta,D-\half)\simeq (\mathcal{E}_\theta, T,\n) \otimes_{\Lip(\bS^2)} ( \H_0, D_0-\half)
\end{equation}
as unbounded Lipschitz cycles in $\Psi_{-1}(\Lip(\bS^3_\theta),\C)$, $\Psi^\ell_{-1}(\Lip(\bS^3),\Lip(\bS^2) )$ and $\Psi_0(\Lip(\bS^2), \C)$, respectively.

We begin with the Hilbert modules appearing in \eqref{fac-claim}. Let us write $(\cdot,\cdot)_{\bS^2}$ for the inner product on the Hilbert space $L^2(\bS^2)$ and consider the tensor product of Hilbert modules $\mathpzc{E}_\theta \hotimes_{C(\bS^2)} L^2(\bS^2)$, which we equip with the inner product
\begin{equation}
\label{eq:inner-tensor-sph}
\left ( e \otimes h, e' \otimes h' \right ) := (h, \langle e,e'\rangle h')_{\bS^2}
\end{equation}
for each $e\otimes h,~ e' \otimes h' \in \mathpzc{E}_\theta\, \otimes_{C(\bS^2)} L^2(\bS^2)$. Recall also our notation $\H_0:=L^2(\bS^2)\otimes\C^2$.

\begin{prop}\label{pr:fac2}
The Hilbert space $\H_\theta:=L^2(\bS^3_\theta)\otimes\C^2$ is isometrically isomorphic to the completion $\mathpzc{E}_\theta\hotimes_{C(\bS^2)} \H_0$ of the tensor product $\mathpzc{E}_\theta \otimes_{C(\bS^2)} \H_0$
with respect to the inner product \eqref{eq:inner-tensor-sph}.
\end{prop}

\proof  This follows immediately from the identification of Hilbert spaces in eq.~\eqref{sp-triv2} and the corresponding isomorphism in the classical case. \endproof

In order to form the product of an odd Lipschitz cycle with an even KK-cycle, we follow Example~\ref{ex:odd-ev}. Proposition \ref{pr:fac2} establishes the necessary isomorphism at the level of Hilbert spaces. Thus we turn to the connection prescribed in the datum $(\E_\theta,T,\n)$.

Let us introduce for each $n\in\ZZ$ the shorthand notation $\E_n:=p_n \Lip(\bS^2)^{|n|+1}$. Recall from Proposition~\ref{pr:projprop} that we have for each $n\in \ZZ$ a cb-isomorphism 
\begin{equation}\label{rec-iso}
\Psi_n:\L_{-n}\to \E_n,\qquad f\mapsto (f_k):=(\Psi_{n,k}f),
\end{equation}
where $k=0,1,\ldots n$. The Grassmann connections $\n_n:\L_n\to\L_n\hotimes_{\Lip(\bS^2)}\Omega^1_{D_0}(\Lip(\bS^2))$ are in fact defined via these isomorphisms to be
$$
\n_n:\E_n:\to\E_n\hotimes_{\Lip(\bS^2)}\Omega^1_{D_0}(\Lip(\bS^2)),\qquad \n_n:=p_n\circ\d,
$$
where $\d:\Lip(\bS^2)\to\Omega^1_{D_0}(\Lip(\bS^2))$ is the exterior derivative on the classical two-sphere. On the other hand, just as we did for the exterior derivative in eq.~\eqref{extder}, we may express each of the Grassmann connections in terms of the vector fields $Z_\pm$ on the total space $\bS^3_\theta$ of the Hopf fibration, as the following result now shows.

\begin{prop}
Under the decomposition $\E_\theta\cong \bigoplus_{n\in\ZZ}\L_n$, the connection $\n:=\oplus_n\n_n$ of Proposition~\ref{KK-conn} coincides with the linear map
$$
\n_Z:{\mathcal{E}_\theta}\to{\mathcal{E}_\theta}\,\hotimes_{\Lip(\bS^2)}\Omega^1_{D_0}(C(\bS^2),\Lip(\bS^{2})),\qquad \n_Z:=i Z_2\gamma^2+i Z_3\gamma^3.
$$
\end{prop}

\proof The analysis of the previous sections means that we are now free to check everything purely at the algebraic level, safe in the knowledge that our computations will extend to the level of Lipschitz modules. One readily verifies the equality
$$
i Z_2 \gamma^2 +i Z_3 \gamma^3 = Z_+ \sigma_+ + Z_- \sigma_-
$$
(in fact we already saw this in obtaining the expression \eqref{extder}).
Since the matrices $\sigma_\pm$ represent the Clifford multiplication corresponding to the one-forms $\omega_\pm$ of eq.~\eqref{oneforms} we therefore need to verify that, under each of the isomorphisms \eqref{rec-iso}, we have a commutative diagram
\begin{diagram}
\L_n & \rTo^{\n_Z} & & \L_n\hotimes_{\L_0}(\L_2\oplus\L_{-2})\\
\dTo^{\Psi_n} & & & \dTo_{\Psi_n\otimes\iota} \\
\E_n &  \rTo_{\n_n} & & \E_n\hotimes_{\Lip(\bS^2)}\Omega^1_{D_0}(\Lip(\bS^2)),
\end{diagram}
where $\n_Z:\L_n\to \L_{n+2}\oplus\L_{n-2}\simeq  \L_n\hotimes_{\L_0}(\L_2\oplus\L_{-2})$ and $\iota:(\L_2\oplus\L_{-2})\to\Omega^1_{D_0}(\Lip(\bS^2))$ is the map induced by the identifications of modules in Lemma~\ref{le:idents}.
To this end, we compute that
\begin{align*}
(p_n\d (\Psi_n f))_k&=p_{n,kl}\d(\Psi_{n,l}f)
\\
&=p_{n,kl}\Psi_{n,l}\left(Z_+(f)\omega_+ + Z_-(f)\omega_-\right) + \Psi_{n,k}\Psi_{n,l}^*[D_0,\Psi_{n,l}]f
\\
&=\Psi_{n,k}\left(Z_+(f)\omega_+ + Z_-(f)\omega_-\right)
\\
&=\left( \Psi_{n}(Z_+(f)\omega_+ + Z_-(f)\omega_-) \right)_k
\end{align*}
where we have written $p_n:=(p_{n,kl})$ for $k,l=0,1,\ldots,n$. The third equality follows by using Proposition~\ref{miracle} to deduce the vanishing of the appropriate terms, together with the identity $\Psi_n^*\Psi_n=1$.
\endproof

\begin{thm}
As an element of $\Psi_{-1}(\Lip(\bS^3_\theta),\C)$, the Riemannian spin geometry of $\bS^3_\theta$ factorizes as a Kasparov product of Lipschitz cycles, namely
$$
(\H_\theta, D-\half) \simeq (\mathcal{E}_\theta, T,\n) \otimes_{\Lip(\bS^2)} (\H_0, D_0-\half),
$$
where $(\mathcal{E}_\theta, T,\n)\in\Psi^\ell_{-1}(\Lip(\bS^3_\theta),\Lip(\bS^2) )$ and $(\H_0, D_0-\half)\in\Psi_0(\Lip(\bS^2), \C)$.
\end{thm}

\proof We have already established the necessary isomorphism of Hilbert spaces. 
The operator in the Kasparov product on the Hilbert module $\mathcal{E}_\theta\hotimes_{C(\bS^2)}\H_0$ is given by the expression \eqref{eq:ev-odd-op}, that is to say
$$
f\otimes s \mapsto Tf\otimes \Gamma_0 s+(\n f)s +f\otimes D_0s,\qquad \text{for each}~ f\in\mathpzc{E}_\theta,~s\in \H_0,
$$
where $\Gamma_0:\H_0\to\H_0$ is the grading on the spinors over the two-sphere. We shall check that this product operator agrees with the Dirac operator $D$ on the three-sphere $\bS^3_\theta$. 
To this end, let us consider the operator
$$
1 \otimes_\n D_0: \L_n \otimes_{\L_0} (\L_1 \oplus \L_{-1}) \to \L_n \otimes_{\L_0} (\L_1 \oplus \L_{-1})
$$
The matrices $\sigma_\pm$ appearing in $\nabla$ act on the Hilbert space $\H_0$, as does the Dirac operator $D_0$. Upon using eq.~\eqref{eq:line-product} we deduce from the fact that both $\nabla$ and $D_0$ can be written in terms of $Z_\pm$ and $\sigma_\pm$ that
$$
1 \otimes_\n (D_0-\half)  \simeq i Z_2 \gamma^2 + i Z_3 \gamma^3 ,
$$ 
through the identification $\L_n \otimes_{\L_0} (\L_1 \oplus \L_{-1})\simeq\L_{n+1} \oplus \L_{n-1}$. Next, we consider how $T \otimes \Gamma_0 = i Z_1 \gamma^1$ behaves under the latter identification of line bundles. Correcting for the shift in $n$, we find that
$$
T\otimes \Gamma_0 \simeq (T +\Gamma) \Gamma = (iZ_1 +\gamma^1) (\gamma^1) = i Z_1 \gamma^1 +1.
$$
As a consequence we find that 
$$
T \otimes \Gamma + 1 \otimes_\n (D_0-\half) \simeq i Z_1\gamma^1 + i Z_2 \gamma^2 + i Z_3 \gamma^3 + 1.
$$ 
This is the required explicit factorization of the Dirac operator on $\bS^3_\theta$ in terms of a vertical part $T$ and a horizontal part $D_0$, linked via the connection $\nabla$. 
\endproof

In conclusion, we have cast the gauge theory described by the spin geometry of the noncommutative three-sphere $\bS^3_\theta$ into a geometrical setting, consisting of a Hilbert bundle over $\bS^2$ equipped with a connection and a fibrewise endomorphism.  

Indeed, the $C^*$-module $\EE_\theta$ is the space of continuous sections of some Hilbert bundle $V \to \bS^2$ whose fibres are essentially copies of the Hilbert space $L^2(\bS^1)$. According to Definition~\ref{defn:KK-gauge}, the internal gauge group $\U(\Lip(\bS^3_\theta))$ of the fibration is a normal subgroup of the gauge group $\GG(\E_\theta)$, acting fibrewise upon the Hilbert bundle $V$. The internal gauge fields in $\Omega^1_D(\Lip(\bS^3_\theta))$ decompose according to Lemma \ref{lem:inner-decomp}: the scalar fields $\mathcal{C}_s$ act vertically upon the Hilbert bundle $V$, whilst the gauge fields $\mathcal{C}_g$ are given by connections thereon. 

As in the case of the noncommutative two-torus considered in the previous section, in passing from $\U(\Lip(\bS^3_\theta))$ to $\GG(\E_\theta)$ we find that the Pontrjagin dual group $\Z$ of $\TT$ acts vertically on the bundle $V$. In other words, we can consider the semi-direct product $\U(\Lip(\bS^3_\theta)) \rtimes \Z\subset \GG(\E_\theta)$ as a natural extension of $\U(\Lip(\bS^3_\theta))$. We leave the potential application of our notion of gauge theories ---after extending it to four-dimensional examples such as $\bS^4_\theta$--- to instanton moduli spaces \cite{BL09a,bvs,B13} for future work.

\appendix

\bibliography{references}

\newcommand{\noopsort}[1]{}\def\cprime{$'$}

\end{document}